\documentclass[11pt]{article}

\usepackage{amsmath,amsthm,amsfonts,amscd,amssymb,amstext,bbm}

\usepackage{xparse,mathrsfs,mathtools,todonotes,listings}
\usepackage[colorlinks,citecolor=blue]{hyperref}

\newtheorem{theorem}{Theorem}
\newtheorem{lemma}[theorem]{Lemma}
\newtheorem{corollary}[theorem]{Corollary}
\newtheorem{remark}[theorem]{Remark}
\newtheorem{definition}[theorem]{Definition}
\newtheorem{question}[theorem]{Question}
\newtheorem{example}[theorem]{Example}
\numberwithin{equation}{section}
\numberwithin{theorem}{section}

\newcommand{\R}{\mathbb{R}}
\newcommand{\D}{\mathbb{D}}
\newcommand{\T}{\mathbb{T}}
\newcommand{\C}{\mathbb{C}}
\newcommand{\Z}{\mathbb{Z}}

\newcommand{\N}{\mathbb N}

\newcommand{\Arg}{\operatorname{Arg}}

\renewcommand{\Pr}{\mathbb{P}}

\newcommand{\Exp}{\mathbb{E}}

\newcommand{\filt}{\mathscr{F}}
\newcommand{\Gfilt}{\mathscr{G}}
\DeclareDocumentCommand \one { o }
{%
  \IfNoValueTF {#1}
  {\mathbf{1}  }
  {\mathbf{1}\left\{ {#1} \right\} }%
}

\newcommand{\Beta}{\operatorname{Beta}}

\newcommand{\lawequals}{\overset{d}{=} }

\DeclareDocumentCommand{\Prto} {o O{\infty} } {
  \IfNoValueTF {#1}
  {\overset{p}{\longrightarrow}}
  { \xrightarrow[ #1 \to #2]{p }}
}
\DeclareDocumentCommand{\Asto} {o} {
  \IfNoValueTF {#1}
  {\overset{\operatorname{a.s.}}{\longrightarrow}}
  {
    \xrightarrow[ #1 \to \infty]{\operatorname{a.s.} }
  }
}
\DeclareDocumentCommand{\Mgfto} {o} {
  \IfNoValueTF {#1}
  {\overset{\operatorname{mgf}}{\longrightarrow}}
  { \xrightarrow[ #1 \to \infty]{\operatorname{mgf} }}
}

\DeclareDocumentCommand{\Wkto} {o} {
  \IfNoValueTF {#1}
  {\overset{d}{\longrightarrow}}
  { \xrightarrow[ #1 \to \infty]{d }}
}

\newcommand{\footremember}[2]{
    \footnote{#2}
    \newcounter{#1}
    \setcounter{#1}{\value{footnote}}
}
 
\title{Secular Coefficients and the Holomorphic Multiplicative Chaos}
\author{
  Joseph Najnudel\footremember{a}{School of Mathematics, University of Bristol, BS8 1UG, United Kingdom. \textit{joseph.najnudel@bristol.ac.uk}}
  \and Elliot Paquette\footremember{b}{Department of Mathematics and Statistics, McGill University, Montreal, Quebec, Canada. \textit{elliot.paquette@mcgill.ca}}
  \and Nick Simm\footremember{c}{Department of Mathematics, University of Sussex, BN1 9RH, United Kingdom. \textit{n.j.simm@sussex.ac.uk}}
  }
 \date{} 
\begin{document}

\maketitle
\begin{abstract}
  We study the \textit{secular coefficients} of $N \times N$ random unitary matrices $U_{N}$ drawn from the Circular $\beta$-Ensemble, which are defined as the coefficients of $\{z^n\}$ in the characteristic polynomial $\det(1-zU_{N}^{*})$. When $\beta > 4$ we obtain a new class of limiting distributions that arise when both $n$ and $N$ tend to infinity simultaneously. We solve an open problem of Diaconis and Gamburd \cite{DiaconisGamburd} by showing that for $\beta=2$, the middle coefficient of degree $n=\lfloor\frac{N}{2}\rfloor$ tends to zero 
  as $N \to \infty$. We show how the theory of Gaussian multiplicative chaos (GMC) plays a prominent role in these problems and in the explicit description of the obtained limiting distributions. We extend the remarkable \textit{magic square formula} of \cite{DiaconisGamburd} for the moments of secular coefficients to all $\beta>0$ and analyse the asymptotic behaviour of the moments. We obtain estimates on the order of magnitude of the secular coefficients for all $\beta > 0,$ and these estimates are sharp when $\beta \geq 2$.  These insights motivated us to introduce a new stochastic object associated with the secular coefficients, which we call \textit{Holomorphic Multiplicative Chaos (HMC)}. Viewing the HMC as a random distribution, we prove a sharp result about its regularity in an appropriate Sobolev space. Our proofs expose and exploit several novel connections with other areas, including random permutations, Tauberian theorems and combinatorics.
\\\\\\\\

\end{abstract}

\section{Introduction}
Let $\beta>0$ be a fixed parameter and consider the joint distribution on $N$ points
\begin{equation}
  \text{C$\beta$E}_N(\vartheta_1,\ldots,\vartheta_N) \propto \prod_{1 \leq k < j \leq N}|e^{i\vartheta_k}-e^{i\vartheta_j}|^{\beta} \label{betaens}
\end{equation}
where $\vartheta_{j} \in [0,2\pi)$ for all $j=1,\ldots,N$. This is known as the \textit{Circular $\beta$-Ensemble}. When $\beta=2$ this distribution arises as the law of the eigenvalues of a Haar distributed unitary matrix and is better known as the CUE (Circular Unitary Ensemble). When $\beta \neq 2$ it also arises from the eigenvalue distribution of certain random matrix models, see e.g. \cite{KillipNenciu}. Therefore it makes sense to consider the characteristic polynomial
\begin{equation}
  \chi_{N}(z) = \prod_{j=1}^{N}(1-ze^{-i\vartheta_j}), \label{charpoly}
\end{equation}
which would have the representation $\det(1-zU^*_{N})$ if $\{e^{i\vartheta_j} : 1 \leq j \leq N\}$ were the eigenvalues of a matrix $U_N.$

The goal of this paper is to formulate and solve a new class of probabilistic and combinatorial questions associated with such characteristic polynomials. Our main quantities of interest will be the so-called \textit{secular coefficients}, defined most simply as the coefficients $c^{(N)}_{n}$ in the expansion of \eqref{charpoly} in its Fourier basis,
\begin{equation}
\chi_{N}(z) = \sum_{n=0}^{N}c^{(N)}_{n}z^{n}. \label{chisec}
\end{equation}
In a remarkable paper \cite{DiaconisGamburd}, Diaconis and Gamburd studied these coefficients in the $\beta=2$ setting. They showed that the joint moments of $\{c^{(N)}_{n}\}_{n\geq1}$ are related to the enumeration of combinatorial objects known as \textit{magic squares} -- integer valued square matrices with prescribed row and column sums. For example, when $N \geq nk$, the moment $\mathbb{E}(|c^{(N)}_{n}|^{2k})$ is equal to the number of $k \times k$ magic squares whose rows and columns all sum up to $n$. In general, combinatorial results on magic squares imply that this quantity has order $n^{(k-1)^{2}}$ as $n \to \infty$, with a multiplicative constant given by the volume of the $k^{\mathrm{th}}$ Birkhoff polytope. The determination of these volumes remains a well studied and challenging problem in the combinatorics community. They have been explicitly computed only for $k \leq 10$ and this has required the use of high performance computers.  See \cite{BeckPixton, deLoeraLiuYoshida, CanfieldMcKay} for this and other perspectives, and for number theoretical applications, see \cite{ConreyGamburd, KRRGR}.

Given the rich combinatorial structure associated with the moments of $c^{(N)}_{n}$, there is a natural probabilistic question that comes to mind: 
  Is there a commensurately richly structured probabilistic object to which $c^{(N)}_{n}$ converges as $N \to \infty$?
What if $N$ and $n$ tend to infinity together, in a suitable way? This problem is mentioned in the same paper of Diaconis and Gamburd (see the discussion below Proposition $4$ of \cite{DiaconisGamburd}), however to our knowledge the answer has remained out of reach conjecturally or otherwise, now for almost 15 years. The purpose of this paper is to begin closing this gap, particularly in the general context of the $\beta$-ensembles \eqref{betaens}.

There is one exception where a limiting distribution for $c^{(N)}_{n}$ can be obtained with relative ease, namely when the index $n$ remains finite and we let the matrix size $N \to \infty$, as discussed in \cite{DiaconisGamburd} for $\beta=2$. By the Newton formula, which relates the elementary and power sum symmetric functions, each $c^{(N)}_{n}$ has a finite polynomial dependence on the first $n$ power traces, $p_{k} := \mathrm{Tr}(U_{N}^{k})$, via:
\begin{equation}
c^{(N)}_{n} = \frac{1}{n!}\,\mathrm{det}\begin{pmatrix} p_{1} & 1 & 0 & \ldots & 0\\
p_{2} & p_{1} & 2 & \ldots & 0\\
\vdots & \vdots & \vdots & \ddots & \vdots \\
p_{n-1} & p_{n-2} & p_{n-3} & \ldots & n-1\\
p_{n} & p_{n-1} & p_{n-2} & \ldots & p_{1}\end{pmatrix}. \label{newtonform}
\end{equation}
The distributional convergence of these power traces was famously studied by Diaconis and Shahshahani \cite{DS94} for $\beta=2$ and by Jiang and Matsumoto \cite{JM15} for general $\beta>0$, see also \cite{ChhaibiNajnudel}. These authors show that for any fixed $n$ we have the joint convergence in law
\begin{equation}
\{\mathrm{Tr}(U_{N}^{k})\}_{k=1}^{n} \overset{d}{\longrightarrow} \sqrt{\frac{2}{\beta}}\{\sqrt{k}\mathcal{N}_{k}\}_{k=1}^{n}, \qquad N \to \infty, \label{dsconv}
\end{equation}
where $\{\mathcal{N}_{k}\}_{k=1}^{n}$ are {i.i.d.}\ standard complex normal random variables, \textit{i.e.} the real and imaginary parts of $\mathcal{N}_{k}$ are independent normal random variables such that
\begin{equation}
  \Exp(\mathcal{N}_k) = 0,
  \quad
  \Exp (\mathcal{N}_k^{2}) = 0,
  \quad
  \text{and}
  \quad
  \Exp(|\mathcal{N}_k|^2) = 1. \label{iidnk}
\end{equation}
Therefore \eqref{dsconv} implies that \textit{for fixed $n$}, the sequence of random variables $c^{(N)}_{n}$ converges as $N \to \infty$ to a limit random variable $c_{n}$. Furthermore, each $c_{n}$ is explicitly characterized through a polynomial dependence on a family of {i.i.d.}\ Gaussian random variables via formulas \eqref{dsconv} and \eqref{newtonform}.

In contrast, the situation where the degree $n \to \infty$ turns out to be much more challenging. To give a flavour of the results we obtain, we provide here the limiting distribution in the case $\beta>4$; see Section \ref{sec:mainres} for further results and discussion. 
\begin{theorem}
  Let $N=N_n$ be a sequence such that $N \to \infty$ as $n \to \infty$ and such that $n/N \to 0$. Let $\mathcal{Z}$ denote a standard complex normal random variable and $\mathcal{E}(1)$ denote the standard exponential random variable with parameter $1$, sampled independently. Then for any $\beta>4$, we have the convergence in distribution
  \begin{equation}
    \frac{c_{n}^{(N)}}{\sqrt{\Exp(|c_n^{(N)}|^2)}} \Wkto[n] \frac{\mathcal{Z}\,\mathcal{E}(1)^{-\frac{1}{\beta}}}{\sqrt{\Gamma\left(1-\frac{2}{\beta}\right)}}, \label{limthm}
  \end{equation}
  where $\Gamma(z)$ is the Gamma function.
  \label{th:l2phase-N-intro}
 \end{theorem}
A quick computation shows that the right-hand side of \eqref{limthm} has finite moments of order $2k$ if and only if $\frac{2k}{\beta} < 1$. Our proof makes explicit use of the second moment method ($k=2$ in this context) and this gives rise to the restriction $\beta>4$ in the statement of Theorem \ref{th:l2phase-N-intro}. We believe that \eqref{limthm} persists to any $\beta>2$ and even the case $\beta=2$ after suitably re-normalizing on both sides of \eqref{limthm}. 

At the level of tightness, we are able to remove the restriction on $\beta$ and our results hold for any $\beta>0$. We can also relax the condition $n/N \to 0$ stated in Theorem \ref{th:l2phase-N-intro}. As an example, we have the following particular case that resolves a problem of Diaconis and Gamburd \cite{DiaconisGamburd} on the middle secular coefficient, defined by setting $n=\lfloor \frac{N}{2} \rfloor$.
\begin{theorem}
Let $\beta=2$ and set $w_{N} = \log(1+N)^{-1/4}$. Then we have that $\{c^{(N)}_{\lfloor \frac{N}{2} \rfloor}/w_{N}\}_{N \geq 1}$ and $\{w_{N}/c^{(N)}_{\lfloor \frac{N}{2} \rfloor}\}_{N \geq 1}$ are both tight families of random variables. In particular $c^{(N)}_{\lfloor \frac{N}{2} \rfloor} \overset{d}{\to} 0$ as $N \to \infty$. 
More generally, when $n_{N}\to \infty$ in such a way that $2n_N \leq N,$ 
 $c^{(N)}_{n_N} \overset{d}{\to} 0$ as $N \to \infty.$
\label{th:mid-coefficient}
\end{theorem}
Note that Theorem \ref{th:mid-coefficient} holds despite the fact that $\mathbb{E}(|c_{n}^{(N)}|^{2})=1$ identically. A similar phenomenon has been observed in a number theoretical context, namely in the theory of random multiplicative functions and `better than square root cancellation' \cite{harper}. In the regime $0 < \beta < 2$, we will establish a similar class of tightness results which show that the normalization $\sqrt{\mathbb{E}(|c_{n}^{(N)}|^{2})}$ overestimates the correct order of magnitude for the coefficients, see Section \ref{sec:cbe-results}.

\subsection{The holomorphic multiplicative chaos}

Statistical properties of random matrix characteristic polynomials have attracted considerable interest recently, in large part due to an intimate relationship to \textit{logarithmically correlated Gaussian fields} and \textit{Gaussian multiplicative chaos} (GMC), see \cite{DRSV17} and \cite{RhodesVargasSurvey} for general background on these topics. The connection to characteristic polynomials of random matrices arose quite recently in an influential work of Fyodorov, Hiary and Keating \cite{FHK12, FK14}. In particular the attempts to prove the conjectures in \cite{FHK12,FK14}, which are still unresolved in full generality, have motivated a number of recent studies on characteristic polynomials of random matrices, for a non-exhaustive list see e.g. \cite{ABB17,BWW, ChhaibiNajnudel, CMN, LOS18,NSW, PaquetteZeitouni,  Webb}, and also on various parallel questions concerning the Riemann zeta function \cite{ABBRS,ABR20, NajnudelZeta, SaksmanWebb}. We will show how the GMC theory also plays a prominent role in the analysis of the secular coefficients $c^{(N)}_{n}$ in \eqref{chisec}.

To describe formally how a log-correlated Gaussian field can arise from the characteristic polynomial \eqref{charpoly}, we expand the logarithm as
\begin{equation}
\log \chi_{N}(z) = -\sum_{k=1}^{\infty}\frac{z^{k}}{k}\mathrm{Tr}(U_{N}^{-k}). \label{logchi}
\end{equation}
By the convergence \eqref{dsconv} we can identify a candidate limiting Gaussian field by replacing the power traces $\mathrm{Tr}(U_{N}^{-k})$ with $\sqrt{\frac{2}{\beta}}\sqrt{k}\mathcal{N}_{k}$, where $\mathcal{N}_{k}$ are {i.i.d.}\ standard complex Gaussian variables as in \eqref{iidnk} and for the sake of simplicity we ignore the minus sign in \eqref{logchi}, noting the rotational invariance of each $\mathcal{N}_{k}$. Now let $G^{\C}$ be the Gaussian analytic function on the unit disc $\D$ 
\begin{equation}
  G^{\C}(z) = \sum_{k=1}^\infty \frac{z^k}{\sqrt{k}}\mathcal{N}_k,
  \label{eq:GAF}
\end{equation}
so that we expect $\chi_{N}(z)$ to be close to $e^{\sqrt{\frac{2}{\beta}}G^{\C}(z)}$ in a suitable sense. The covariance of the field $G^{\C}$ follows from the simple i.i.d.\ structure of the variables $\mathcal{N}_{k}$ as
\[
  \Exp[ G^{\C}(w){G^{\C}(z)} ] = 0
  \quad\text{and}\quad
  \Exp[ G^{\C}(w) \overline{G^{\C}(z)} ] = -\log(1-w\overline{z}),
\]
from which it follows that $G = 2\Re G^{\C}$ and $H = 2 \Im G^{\C}$ are identically distributed Gaussian fields with
\begin{equation}\label{eq:RealG}
  \Exp[ G(w){G(z)} ] 
  =
  -2\log|1-w \overline{z}|
  \quad\text{and}\quad
  \Exp[ G(w){H(z)} ]  
  =
  -2\Arg(1-w \overline{z}),
\end{equation}
where we take the principal branch of the argument. In particular, when defined on the unit circle $|z|=|w|=1$, the real valued field $G$ is a prototypical example of a log-correlated Gaussian field (see e.g. \cite{HKO01} where it first appeared explicitly).


For trigonometric polynomials $\phi$ if we let $z \mapsto \phi(z)$ for $z \in \D$ be the continuous harmonic extension to $\D,$ we can define the random distribution
\begin{equation}
  ( \mathrm{HMC}_\theta, \phi)
  \coloneqq
  \lim_{r \to 1} 
  \frac{1}{2\pi}
  \int_0^{2\pi} 
  e^{\sqrt{\theta} G^{\C}(re^{i\vartheta})}
  \overline{\phi(r\vartheta)} d\vartheta,
  \label{eq:hmc}
\end{equation}
which we will call the \emph{holomorphic multiplicative chaos} or HMC, where from now on we will adopt the notation
\begin{equation}
\theta := \frac{2}{\beta}.
\end{equation}
For the moment, we comment that for trigonometric polynomials, the existence of this limit is trivial and is sufficient to uniquely define the HMC.  In particular, if we take $\phi(\vartheta)=e^{i n \vartheta}$ for $n \in \N_0,$ then using analyticity and Cauchy's theorem the limit is just the $n$-th coefficient in the power series expansion of $e^{\sqrt{\theta} G^{\C}(z)}$ at $z=0.$  We define for $n \in \N_0$ the Fourier coefficient of the HMC
\begin{equation}\label{eq:cn}
  c_n \coloneqq ( \mathrm{HMC}_\theta, \vartheta \mapsto  e^{i n\vartheta}).
\end{equation}
Equivalently, the coefficients $c_{n}$ can be extracted from a generating function by the formula
\begin{equation}
c_{n} = [z^{n}]\,e^{\sqrt{\theta}G^{\mathbb{C}}(z)} = [z^{n}]\,\mathrm{exp}\left(\sqrt{\theta}\sum_{k=1}^{\infty}\frac{z^{k}}{\sqrt{k}}\,\mathcal{N}_{k}\right), \label{cngen}
\end{equation}
where the notation $[z^{n}]\,h(z)$ denotes the coefficient of $z^{n}$ in the power series expansion of $h(z)$ around the point $z=0$. We could as well define this for $n \in \Z,$ but for negative integers, this would be $0$.

The HMC is in some sense the distributional limit of the characteristic polynomial $\chi_N$ inside the unit circle.
The following is shown in \cite{ChhaibiNajnudel}:
\begin{theorem}[Proposition 3.1 of \cite{ChhaibiNajnudel}] \label{thm:hmc}
  For any $\beta > 0$ it is possible to define $\chi_{N}$ and the field $G^{\mathbb{C}}$ on a single probability space in such a way that for any $r \in (0,1)$,
  \[
    \sup_{|z| \leq r} \bigl|\chi_N(z)-e^{\sqrt{\theta} G^{\C}(z)}\bigr| \Asto[N] 0.
  \]
  The convergence also holds in $L^p$ for any $p \geq 1.$
\end{theorem}
\noindent Then as a corollary of Theorem \ref{thm:hmc} and Cauchy's integral formula, on the probability space therein, each $c_n^{(N)} \Asto[N] c_n.$  We shall continue to refer to $\left\{ c_n \right\}$ as the secular coefficients of the $\mathrm{HMC}_\theta$.

While $\mathrm{HMC}_\theta$ is uniquely determined by its Fourier coefficients and exists for all $\theta > 0$, we will characterize its regularity in the Sobolev sense.
We define the Sobolev norms for any $s \in \R$ on the trigonometric polynomials on $\T$ by 
\begin{equation}\label{eq:sobolevnorms}
  \| f \|_{H^s}^2 = \sum_{n \in \Z} (1+n^2)^{s}|\widehat{f}(n)|^2,
  \quad\text{where}
  \quad 
  \widehat{f}(n) = \frac{1}{2\pi}\int_0^{2\pi} f(e^{i\vartheta})e^{-in \vartheta} d\vartheta.
\end{equation}
The Sobolev spaces $H^{s}$ for $s \in \R$ are the closures in the space of distributions of the trigonometric polynomials on $\T$ under these norms.  
For any $s \in \R$ the norms defined in \eqref{eq:sobolevnorms} are well--defined for all functions $f \in L^2,$ with the understanding that they may be infinite, and for $s \geq 0$ the space $H^s$ is precisely the subspace of $L^2$ for which $\|\cdot\|_{H^s} < \infty.$  The formulas for the norms \eqref{eq:sobolevnorms} extend generally to the space of distributions on $\T$ by replacing $\hat f(n) \coloneqq (f, \vartheta \mapsto e^{in\vartheta}).$  For any $s \in \R,$ the pair $H^{s}$ and $H^{-s}$ are dual spaces with respect to the natural inner product on $\T.$

\begin{theorem}\label{thm:hmcregularity}  
  Define
  \[
    s_\theta \coloneqq \begin{cases}
      -\frac{\theta}{2}, & \text{ if }\theta \leq 1, \\
      - \sqrt{\theta} + \frac{1}{2}, & \text{ if }\theta > 1. \\
    \end{cases}
  \]
  Then for any $\theta > 0,$ $\mathrm{HMC}_\theta$ is in $H^{s}$ almost surely for any $s < s_\theta$ and it is almost surely not in $H^s$ for any $s > s_\theta.$
\end{theorem}
\begin{proof}
See Section \ref{sec:regular}.
\end{proof}
\noindent 
We remark that in the variable $\theta = \frac{\gamma^{2}}{4}$ this Theorem states that $\mathrm{HMC}_\theta$ is in $H^{s}$ provided that 
\begin{equation}
-2s+1 > \begin{cases} 1+\frac{\gamma^{2}}{4}, & \gamma < 2\\
 \gamma, & \gamma > 2.\end{cases}
\label{freezingtran} \end{equation} 
The threshold on the right-hand side of \eqref{freezingtran} is familiar from the study of the \textit{free energy} in various other log-correlated models (see e.g.\ \cite{FyodorovBouchaud}). 

\subsection{The combinatorial structure of the moments}
\label{sec:mainres}
We begin by discussing the moments of the secular coefficients. As we already mentioned, when $\beta=2$ (or $\theta=1$), Diaconis and Gamburd \cite{DiaconisGamburd} characterize these moments in terms of the enumeration of magic squares. We will state a result that generalizes this characterization to arbitrary $\beta>0$.
\begin{definition}
\label{def:magicsq}
  A magic square of size $k$ with row sums $\mu = (\mu_1,\ldots,\mu_k)$ and column sums $\nu = (\nu_1,\ldots,\nu_k)$ is a $k \times k$ matrix $A$ with entries in $\N_0$ and with the property that
  \begin{equation}
    \begin{split}
      \sum_{j=1}^{k}A_{ij} &= \mu_{i}, \qquad i=1,\ldots,k,\\
      \sum_{i=1}^{k}A_{ij} &= \nu_{j}, \qquad j=1,\ldots,k.
    \end{split}
  \end{equation}
\end{definition}
Throughout the article we will use the notation $\mathrm{Mag}_{\mu,\nu}$ to denote the collection of all such $k \times k$ magic squares. We recall that non-negative integer vectors  $\mu$ are called \emph{compositions} in the combinatorics literature, while the entries $\mu_{1},\ldots,\mu_{k}$ are known as \emph{parts}.

\begin{theorem}\label{thm:magic}
  Let $\mu$ and $\nu$ be any compositions with $k$ parts. Then for any $k \in \mathbb{N}$ and $\theta>0$, we have
  \begin{equation}
    \mathbb{E}\left(\prod_{j=1}^{k}c_{\mu_j}\overline{c_{\nu_j}}\right) = \sum_{A \in \mathrm{Mag}_{\mu,\nu}}\prod_{1\leq i, j \leq k}\binom{A_{ij}+\theta-1}{A_{ij}}. \label{genbetmoms}
  \end{equation}
\end{theorem}
\noindent We give the proof in Section \ref{sec:moments} and furthermore provide a combinatorial connection to Jack functions. When $\beta=2$ ($\theta=1$) the right-hand side of \eqref{genbetmoms} reduces to the cardinality of $\mathrm{Mag}_{\mu,\nu}$, recovering the result of \cite{DiaconisGamburd} in the particular case $N=\infty$. Choosing the column and row sums all to equal $n$, we obtain an expression for the absolute $2k^{\mathrm{th}}$ moments 
\begin{equation}\label{eq:kthmom}
  \mathbb{E}(|c_{n}|^{2k}) = \sum_{A \in \mathrm{Mag}_{\pi,\pi}}\prod_{1 \leq i, j \leq k}\binom{A_{ij}+\theta-1}{A_{ij}}
\end{equation}
where $\pi$ is the composition in which $n$ appears $k$ times, i.e.\ $\mathrm{Mag}_{\pi,\pi}$ is the set of all magic squares of size $k \times k$ with all row and column sums equal to $n.$
We note the special case
\begin{equation}\label{eq:2ndmom}
  \mathbb{E}(|c_{n}|^{2}) = \binom{n+\theta-1}{\theta-1} \sim \frac{1}{\Gamma(\theta)}\,n^{\theta-1}.
\end{equation}
It turns out it is possible to use this formula and \eqref{eq:kthmom} to probe asymptotic behaviour of the moments. 
We have:
\begin{theorem}\label{thm:morris}
  For any positive integer $k$ and any $\theta >0$ so that $\theta k < 1,$
  \begin{equation}
    \lim_{n \to \infty}\frac{\mathbb{E}(|c_{n}|^{2k})}{\mathbb{E}(|c_{n}|^{2})^{k}} = \frac{\Gamma(1-k\theta)}{\Gamma(1-\theta)^{k}}\,k! \label{morrismoms}
  \end{equation}
\end{theorem}
\begin{proof}
See Section \ref{sec:moments}.
\end{proof}
It is a simple computation, recalling that $\theta = \frac{2}{\beta}$, that the moments on the right-hand side of \eqref{morrismoms} are precisely those of the limiting distribution in Theorem \ref{th:l2phase-N-intro}. We mention in passing that besides magic squares, which play an important role in describing $\mathrm{HMC}_\theta,$ the \emph{Ewens sampling formula}, which defines a classical distribution on random permutations, plays a prominent role in our analysis, see Section \ref{sec:ewens} for details.


\subsection{From magic squares back to multiplicative chaos}

Theorem \ref{thm:morris} is strongly reminiscent of the \emph{freezing transition} observed for moments of random energy models \cite{FyodorovBouchaud} (c.f.\ Remark \ref{rem:highermoment} for the behavior of a moment above the critical threshold). Indeed, in \cite{FyodorovBouchaud} it is shown that the \emph{Morris integral}
\begin{equation}\label{eq:morris}
  \frac{\Gamma(1-k\theta)}{\Gamma(1-\theta)^{k}}
  =
  \frac{1}{(2\pi)^k}\int\limits_{[0,2\pi]^k} \prod_{1 \leq  a < b \leq k} |e^{i\vartheta_a}-e^{i\vartheta_b}|^{-2\theta}
  d\vartheta,
\end{equation}
describes the appropriately rescaled moments of the partition function of the logarithmically correlated \emph{random energy model} on the unit circle in the high temperature phase
-- see \cite{FyodorovBouchaud} for details.  

The presence of the Morris integral in Theorem \ref{thm:morris} is moreover indicative that the theory of Gaussian multiplicative chaos will be relevant here. There is a substantial literature on this random measure (see \cite{RhodesVargasSurvey} for a general overview), but we will be concerned only with the following specific instance: for $\theta \in (0,1),$ 
\begin{equation}
  \mathrm{GMC}_\theta(d\vartheta) \coloneqq 
  \lim_{r \to 1} (1-r^2)^{\theta}| e^{\sqrt{\theta}G^{\C}(r e^{i\vartheta})}|^2 d\vartheta
  =\lim_{r \to 1} (1-r^2)^{\theta}e^{\sqrt{\theta}G(r e^{i\vartheta})} d\vartheta.
  \label{eq:gmc}
\end{equation}
The existence of this limit as a random measure is shown in \cite{JunnilaSaksman} (c.f.\ \cite[Proposition 3.1]{ChhaibiNajnudel}).  Moreover it is shown in \cite{Remy} that the total mass of this particular random measure has law characterized by the natural analytic continuation of the Morris integral, i.e.\ for any $p >0$ with $p \theta < 1,$
\begin{equation}\label{eq:mass} 
  \Exp\bigl(\mathcal{M}_\theta\bigr)^p
  =\frac{\Gamma(1-p\theta)}{\Gamma(1-\theta)^{p}},
  \quad
  \text{where}
  \quad
  \mathcal{M}_\theta
  \coloneqq
  \frac{1}{2\pi}\int_0^{2\pi}\mathrm{GMC}_{\theta}(d\vartheta),
\end{equation}
with the sense of the limit being an in-probability, weak-* convergence.

We briefly summarize some of the qualitative properties of the GMC. The regime $\theta \in (0,1)$ is typically referred to as the \emph{subcritical} phase of the $\mathrm{GMC}_\theta.$  The limiting measure is supported on a set of Hausdorff dimension $1-\theta.$ This result is well known in the literature on GMC, see \cite{RhodesVargasSurvey}, though for the model discussed here, see \cite{ChhaibiNajnudel}. The subset $(0,\tfrac 12)$ is sometimes referred to as the $L^2$ phase, on account of the mass $\mathcal{M}_\theta$ gaining a finite second moment (and indeed Theorem \ref{thm:morris} applies as well).

The $L^2$ phase is technically simpler to manage, and has been the setting for some of the first convergence results of $\beta = 2$ characteristic polynomials and powers thereof to the GMC \cite{Webb,BWW}. For a power of the modulus of the CUE ($\beta=2$) characteristic polynomial, this has been improved to the whole subcritical phase in \cite{NSW}.   For general $\beta>0$ to the authors' knowledge there are no convergence results of this type. However for an object closely related to the characteristic polynomial, convergence to the GMC was established in \cite{ChhaibiNajnudel} for general $\beta \geq 2$.

The value $\theta=1$ is the \emph{critical} temperature, at which it is possible to establish \eqref{eq:gmc} under an additional logarithmic renormalization (\cite[Theorem 1.3]{JunnilaSaksman}) 
\begin{equation}
  \mathrm{GMC}_1(d\vartheta) \coloneqq 
  \lim_{r \to 1} \sqrt{\log\tfrac{1}{1-r^2}}(1-r^2) e^{G(r e^{i\vartheta})} d\vartheta.
  \label{eq:criticalgmc}
\end{equation}
This measure is known to have has Hausdorff dimension $0$ and is non-atomic, as proved in \cite{ChhaibiNajnudel}. We note in contrast that $\mathrm{HMC}_\theta$ is well defined as a random distribution on the unit circle for all $\theta > 0$ (and indeed on a single probability space), regardless of the phase of the associated GMC.

Returning to consideration of Theorem \ref{thm:morris}, we see that there is a factor of $k!$ beyond the Morris integral term.  This strongly suggests a limiting factorization into independent random variables, where the additional factor can be interpreted as the moments of a standard complex normal random variable $k! = \mathbb{E}(|\mathcal{Z}|^{2k})$. We show this is indeed the case in the $L^2$ phase.  More precisely:
\begin{theorem}[$L^{2}$-phase]
  For any $0 < \theta < \frac{1}{2}$, we have the convergence in distribution
  \begin{equation}
    \frac{c_{n}}{\sqrt{\mathbb{E}(|c_{n}|^{2})}} \Wkto[n] \sqrt{\mathcal{M}_{\theta}}\mathcal{Z} \label{l2phaselim}
  \end{equation}
  where $\mathcal{Z}$ and $\mathcal{M}_{\theta}$ are independent, $\mathcal{Z}$ is standard complex normal, and $\mathcal{M}_{\theta}$ has law \eqref{eq:mass}.
  \label{th:l2phase}
\end{theorem}
\begin{proof}
See Sections \ref{sec:ewens}-\ref{sec:mgle3}.
\end{proof}
\noindent We expect this result to persist to all $\frac{1}{2} \leq \theta < 1$ and also to $\theta=1$ (the critical temperature) subject to a different normalization. While we do not show this convergence in distribution for $\theta \geq \frac 12,$ we give a sharp estimate for the order of magnitude of the coefficient for larger $\theta$ (see Lemma \ref{lem:limpmoment} and Theorem \ref{thm:order}). 

By the mentioned work \cite{Remy}, the limiting random variable $\mathcal{M}_{\theta}$ appearing on the right-hand side of \eqref{l2phaselim} can be given an explicit characterization. For any $\theta < 1$ it is known that we have the equality in law
\begin{equation}
\mathcal{M}_{\theta} \overset{d}{=} \frac{1}{\Gamma(1-\theta)}\,\mathcal{E}(1)^{-\theta} \label{totalmasslaw}
\end{equation}
where $\mathcal{E}(1)$ is an exponential random variable with parameter $1$. Formula \eqref{totalmasslaw} was initially conjectured by Fyodorov and Bouchaud in \cite{FyodorovBouchaud} and proved quite recently in \cite{Remy} using techniques of Liouville conformal field theory, see also \cite{ChhaibiNajnudel} for an alternative proof. Using the explicit formula \eqref{totalmasslaw}, we see that Theorem \ref{th:l2phase} is the particular case $N=\infty$ of our earlier stated Theorem \ref{th:l2phase-N-intro}.

\subsection{The mass of the chaos}

We note that the presence of the mass of the chaos in the secular coefficients could potentially be anticipated.  On the one hand, by the generating function \eqref{cngen} and Parseval's identity, we have
\begin{equation}
  \begin{split}
    \sum_{n=0}^{\infty}|c_{n}|^{2}r^{2n}
    = \frac{1}{2\pi}\int_{0}^{2\pi} e^{\sqrt{\theta}G(re^{i\vartheta})}\,d\vartheta. 
  \end{split}
  \label{parsevalapp}
\end{equation}
On the other hand, by \eqref{eq:gmc} and \eqref{eq:criticalgmc} for all $0 < \theta \leq 1$,
\begin{equation}
  L(r,\theta)
  (1-r^{2})^{\theta}\sum_{n=0}^{\infty}|c_{n}|^{2}r^{2n} 
  \Prto[r][1]
  \mathcal{M}_{\theta},
  \quad
  L(r,\theta)\coloneqq \begin{cases}
    \sqrt{\log\tfrac{1}{1-r^2}}, & \text{if } \theta=1, \\
    1, & \text{else}.
  \end{cases}
  \label{rlim}
\end{equation}
So in a suitably averaged sense, the squared modulus of the secular coefficients gives the total mass of the chaos in the subcritical and critical phases.  In comparison, Theorem \ref{th:l2phase}, shows that each individual coefficient $\{c_n\}$ already contains much of the information about the mass of the GMC.
%

We also detour briefly to mention that from \eqref{rlim}, it is possible to derive other approximations to the total mass.  One which will be important here is 
\begin{equation}
  \mathcal{M}_{\theta,n} :=
  (\sqrt{\log n})^{\one[\theta=1]}
  \frac{\Gamma(\theta+1)}{n^{\theta}}\sum_{q=0}^{n}|c_{q}|^{2}. 
  \label{mthetan}
\end{equation}
If the convergence in \eqref{rlim} were almost sure, the Hardy--Littlewood Tauberian theorem would immediately imply that $\mathcal{M}_{\theta,n}$ converges almost surely to $\mathcal{M}_\theta.$  We show in 
Theorem \ref{th:tauberprob} in the Appendix that this Tauberian theorem generalizes to the setting of convergence in probability, and the following is an immediate consequence of \eqref{parsevalapp}, \eqref{rlim} and Theorem \ref{th:tauberprob}.
\begin{lemma}
  \label{lem:condvarapprox}
  For any $0 < \theta \leq 1$, we have the convergence in probability to the total mass
  \begin{equation}
    \label{gmc-conv}
    \mathcal{M}_{\theta,n} \Prto[n] \mathcal{M}_{\theta}, \qquad n \to \infty.
  \end{equation}
\end{lemma}

\subsection{The secular coefficients of C$\beta$E}
\label{sec:cbe-results}
Some of what we have proved for the HMC coefficients $\left\{ c_n \right\}$ adapt or transfer to the secular coefficients $\{ c_n^{(N)} \}$ of $N \times N$ random matrices, as defined in \eqref{chisec}. In particular, when $n\to \infty$ and $N\to \infty$ in such a way that $n/N \to 0,$ Theorem \ref{th:l2phase} transfers directly to $\{ c_n^{(N)} \}:$
\begin{theorem}[$L^{2}$-phase for slowly growing $n$]
  Let $N=N_n$ be chosen in such a way that $n/N \to 0$ as $n \to \infty.$ Recalling the notation $\theta := \frac{2}{\beta}$, for any $0 < \theta < \frac{1}{2}$, we have the convergence in distribution
  \begin{equation}
    \frac{c_{n}^{(N)}}{\sqrt{\Exp(|c_n^{(N)}|^2)}} \Wkto[n] \sqrt{\mathcal{M}_{\theta}}\mathcal{Z}
  \end{equation}
  where $\mathcal{Z}$ and $\mathcal{M}_{\theta}$ are independent, $\mathcal{Z}$ is standard complex normal, and $\mathcal{M}_{\theta}$ has law \eqref{eq:mass}.
  \label{th:l2phase-N}
\end{theorem}
\begin{proof}
See Section \ref{sec:cbe}.
\end{proof}
We remark that by formula \ref{totalmasslaw}, Theorem  \ref{th:l2phase-N} is simply a restatement of Theorem \ref{th:l2phase-N-intro}, but we have included it here for added clarity in the context of GMC theory. Let us discuss the necessity of the condition $n/N \to 0$ as $n \to \infty$ appearing in Theorem \ref{th:l2phase-N}. The normalization constant $\Exp(|c_n^{(N)}|^2)$ is known explicitly, due to \cite{Haake} who obtain for any $\theta>0$,
\begin{equation}
  \mathbb{E}(|c^{(N)}_{n}|^{2}) = \binom{N}{n}\,\frac{\Gamma(n+\theta)\Gamma(N-n+\theta)}{\Gamma(\theta)\Gamma(N+\theta)}. \label{haake2nd}
\end{equation}
We remark in passing that, to our knowledge, the work \cite{Haake} was likely the first explicit investigation of secular coefficients in the literature. From \eqref{haake2nd}, we observe two possible types of asymptotics. If $N$ grows with $n$ at a fast enough rate that $n/N \to 0$ as $n \to \infty$, then
\begin{equation}
  \mathbb{E}(|c^{(N)}_{n}|^{2}) \sim \frac{1}{\Gamma(\theta)}\,n^{\theta-1}, \qquad n \to \infty,
\end{equation}
which matches the asymptotic \eqref{eq:2ndmom}.
When however $n = \kappa N$ with $0 < \kappa < 1$, the asymptotics contain an additional pre-factor
\begin{equation}
  \mathbb{E}(|c^{(N)}_{\lfloor \kappa N \rfloor}|^{2}) \sim \frac{(1-\kappa)^{\theta-1}}{\Gamma(\theta)}\,n^{\theta-1}, \qquad N\to \infty.
\end{equation} 
This hints that these higher degree coefficients could display a different limiting behavior. 

On the other hand, we show that the order of magnitude of these secular coefficients is no larger than that of $c_n.$
\begin{theorem}[Order estimate]
  \label{thm:order}
  Let $w_n(\theta)$ be given by
  \[
    w_n = n^{(\theta-1)/2},
    \quad
    w_n = (\log(1+n))^{-1/4},
    \quad \text{or} \quad
    w_n = n^{\sqrt{\theta}-1}(\log(1+n))^{-\tfrac{3}{4}\theta}. 
  \]
  in the cases $\theta \in (0,1),$ $\theta = 1$ or $\theta > 1$ respectively.  
  Then for any $\theta \leq 1,$ with $N_0(n)=2n$ 
  the family
  \[
    \{ (w_n/c^{(N)}_{n}, c^{(N)}_{n}/w_n) : n,N \in \N, N \geq N_0(n) \}
  \]
  is tight.  For $\theta \in (1,2),$ there are constants $u_\theta,v_\theta >0$ so that with $N_0(n) = n^{u_\theta}$ the family
  \[
    \{ (c^{(N)}_{n}/w_n)(\log (1+n))^{-v_\theta} : n,N \in \N, N \geq N_0(n) \}
  \]
  is tight.
\end{theorem}
\begin{proof}
  See Theorems \ref{thm:trueCBEbound} and \ref{thm:antitank2}.
\end{proof}


We make similar statements for $c_n/w_n$ for any $\theta > 0$ in Lemma \ref{lem:limpmoment} and Theorem \ref{thm:antitank1} -- in particular the bounds we establish are consistent with Theorem \ref{th:l2phase} extending to all $\theta \in (0,1].$  For $\theta > 1,$ we do not expect either $u_\theta$ or $v_\theta$ to be sharp.  We may reasonably conjecture the correct value for $v_\theta = 0.$  See Remark \ref{rem:sharpu} for further discussion.  

\subsection{Discussion}

We have analyzed a random distribution, the HMC, which can be considered as a large-$N$ limit of the characteristic polynomials $\chi_N$ of the C$\beta$E.  This limiting process exists for all $\theta > 0$ and we have given the distributional convergence of the Fourier coefficients of this Schwartz distribution.  

Besides appearing implicitly in \cite{ChhaibiNajnudel},
we do not believe the $\mathrm{HMC}_\theta$ has appeared explicitly before.
The nearest connection for the case $\theta=1$ ($\beta=2$) appears in 
\cite[Theorem 1.3]{SaksmanWebb}.
Therein, the authors show there exists a sequence $\delta_N \to 0$ and a random Schwartz distribution $\eta$ on $\R$ so that
\[
  \chi_N( e^{i\delta_N x})e^{-Y_N} \Wkto[N] \eta(x),
\]
where $Y_N$ is a complex Gaussian variable having a nontrivial dependence on $\chi_N$.  The processs $\eta$ can be formally understood as 
\[
  \eta(x) = \exp\biggl( \int_0^\infty \frac{e^{-2\pi i x u}}{\sqrt{u}} d B^{\C}_u\biggr)
\]
for a complex Brownian motion $B^\C.$
It is also shown that this chaos appears as the limit of a randomized model of the Riemann $\zeta$-function.
We note that this process $\eta$ is a possible candidate for definition of $\mathrm{HMC}_\theta$ on the real line when $\theta=1$, and could also be a type of local scaling limit of $\mathrm{HMC}_\theta$ in a suitable vanishing window of $\vartheta.$

Another class of related objects which have been considered are the \emph{complex multiplicative chaoses}.  For example, consider the random distribution $\mathrm{CGMC}_\theta$ given (hypothetically) by
\begin{equation}\label{eq:CGMC}
  (\mathrm{CGMC}_\theta, \phi) = 
  \lim_{r\to 1} \frac{1}{2\pi}
  \int_0^{2\pi}
  e^{\sqrt{\theta/4}
  (G_1(re^{i\vartheta})
  +iG_2(re^{i\vartheta}))}
  \phi(\vartheta)d\vartheta,
\end{equation}
where $G_1$ and $G_2$ are i.i.d.\ copies of $G$ from \eqref{eq:RealG}.  This roughly fits within the frameworks of \cite{LacoinRhodesVargas} and \cite{Lacoin}, although the technical assumptions on the manner of regularization of $G(e^{i\vartheta})$ are not precisely the same.  
We comment that in the language of \cite{LacoinRhodesVargas} that $\mathrm{CGMC}_\theta$ would be in \emph{Phase I} for $\theta \in (0,1),$ in \emph{Phase II} for $\theta > 1$ and at the triple point for $\theta=1.$  We mention briefly that there is other related work on complex multiplicative cascades \cite{BarralJinMandelbrot1,BarralJinMandelbrot2} and imaginary chaoses in \cite{AruJegoJunnila,JunnilaSaksmanWebb}.

When $\theta < 1,$ an adaptation of \cite[Theorem 3.1]{LacoinRhodesVargas} would show the limit in \eqref{eq:CGMC} exists and so the $\mathrm{CGMC}_\theta$ is well defined.\footnote{From the theory in \cite{BarralJinMandelbrot2}, using the series truncation regularization in place of the harmonic regularization, the existence of the limit follows.}  We note that for $\theta \geq 1,$ \cite[Conjectures 5.2,5.3]{LacoinRhodesVargas} suggest that an additional logarithmic normalization is required for convergence, analogously to logarithmic factors needed for convergence of the critical GMC.  Indeed, in work of \cite{MRV,HartungKlimovsky1} (see also \cite{HartungKlimovsky2}), an analogous statement is proven for a complex random energy model built from branching processes.  Note this in constrast to $\mathrm{HMC}_\theta,$ \emph{which requires no further normalization} to converge for any $\theta > 0$.  

Because the correlations are relatively weak between the real and imaginary parts of the field which define $\mathrm{HMC}_\theta$, we expect that some results carry to $\mathrm{HMC}_\theta$ from the theory of the complex multiplicative chaoses, for example the multifractality for $\theta < 1, q<1$:
\[
  \Exp[ |(|\mathrm{CGMC}_\theta, \one[|\vartheta| \leq r ]) |^q] \underset{r\to 0}{\sim} C_q r^{q - \theta q^2},
\]
(see \cite[Theorem 3.6]{LacoinRhodesVargas}, \cite{BarralJin}, and see also the related work on the regularity of the complex Gaussian multiplicative chaoses in \cite{JunnilaSaksmanViitasaari}).
We expect the same to hold for $\mathrm{HMC}_\theta:$
\begin{question}\label{q:mfs}
  For $\theta < 1, q<1,$ does it hold that
  \[
    \Exp[ |(|\mathrm{HMC}_\theta, \one[|\vartheta| \leq r ]) |^q] \underset{r\to 0}{\sim} C_q r^{q - \theta q^2}?
  \]
\end{question}

We conclude by mentioning some unsolved questions on the properties of $\mathrm{HMC}_\theta$.  We have considered the $L^2$-phase of $\theta,$ that is $\theta < \tfrac 12.$  We assume that Theorem \ref{th:l2phase-N} generalizes without alteration to $\theta \in (0,1),$ the $L^1$-phase and in addition to the critical value $\theta=1$ after introducing an appropriate logarithmic factor.
\begin{question}\label{q:L1}
  Show Theorem \ref{th:l2phase} generalizes to all $\theta \in (0,1)$ and that it can be adapted to hold at the critical point $\theta=1.$
\end{question}
\noindent For supercritical $\theta > 1,$ it is reasonable to assume that there is still distributional convergence of $c_n,$ but the exact form of the limit is unclear.

We have also shown distributional convergence of the secular coefficients $c_n^{(N)}$ when $n/N \to 0$ in the $L^2$-phase.  This relies on making a comparison between $c_n$ and $c_n^{(N)}$ which is weaker when $n/N \to c \in (0,1).$  So, we ask:
\begin{question}\label{q:cnn}
  What is the distributional limit of $c_n^{(N)}/\sqrt{\Exp( |c_n^{(N)}|^2)}$ in the subcritical (or even $L^2$) phase when $n/N \to c \in (0,1)$ as $N \to \infty?$ 
\end{question}

One feature of the secular coefficients $c_n^{(N)}$ is that they may be expressed as combinations of certain conditional expectations of $c_n$ and $c_{N-n+1}$ (see \eqref{eq:chiNcoeff}).  Beyond this, it would be interesting to know the joint behavior of the secular coefficients $\left\{ c_n \right\}.$
\begin{question}\label{q:gp}
  For $\theta \in (0,1),$ to what does $\bigl(n^{(1-\theta)/2}c_{n+m} : m \in \Z\bigr)$ converge as $n \to \infty$ in the sense of finite-dimensional marginals?  
\end{question}

We conclude with one final question of a metamathematical nature.
For $r \in (0,1)$, the function 
$ \vartheta \mapsto e^{ \sqrt{\theta} G^{\mathbb{C} }(re^{i \vartheta})}$
from $[0, 2\pi)$ to $\mathbb{C}$ can be seen, after suitable normalization, as the (random) wave function of the position of  a particle 
which lives on $[0, 2\pi)$. Then, if one makes a quantum measurement of the position of the particle,  
the outcome of the measurement is distributed according to the (random) probability measure with density
$$\vartheta \mapsto \frac{e^{ \sqrt{\theta}  G(re^{i \vartheta})}}{ \int_0^{2 \pi} e^{ \sqrt{\theta}  G(r e^{i \vartheta})} d \vartheta}$$
with respect to the Lebesgue measure on $[0, 2\pi)$. Now, if we are in the subcritical or the critical phases ($\theta \leq 1$), 
and if we take the limit $r \rightarrow 1$, we see that 
$$ \vartheta \mapsto e^{ \sqrt{\theta} G^{\mathbb{C} }(\vartheta)} = \sum_{n \geq 0} c_n e^{n i \vartheta}$$
can be formally seen as as the wave function of a particle such that a quantum measurement gives an outcome distributed according to the 
(random) probability measure $\mathrm{GMC}_{\theta} /\mathcal{M}_\theta.$
\begin{question}\label{q:physics}
  Is there a physically meaniningful quantum mechanical system with wave function $\mathrm{HMC}_{\theta} /\sqrt{\mathcal{M}_\theta}?$ 
\end{question}

\subsection{Organization}

The structure of this paper is as follows. 
In Section \ref{sec:moments} we compute the joint moments of the secular coefficients $\left\{ c_n \right\},$ which we express in terms of magic squares.  We then relate these to Jack functions and compute their asymptotics, proving Theorems \ref{thm:magic} and \ref{thm:morris}.  In Section \ref{sec:ewens}, we make a connection between the moments of the secular coefficients and the Ewens sampling formula.  We then review known estimates of the Ewens sampling formula.  

In Sections \ref{sec:mgle1}, \ref{sec:mgle2}, \ref{sec:mgle3}, and \ref{sec:mgle4} we prove Theorem \ref{th:l2phase}.  We do this by ultimately using the martingale central limit theorem.  However $c_n$ itself is not suitable for a direct application.  So in Section \ref{sec:mgle1} we find a related random variable $\tilde{c}_n^{(\delta)}$ which is a close approximation to $c_n$ and whose Doob martingale with respect to a natural filtration has easily understood increments.  In Section \ref{sec:mgle2}, we give the proof of this normal approximation for $\tilde{c}_n^{(\delta)}$ and then a proof of Theorem \ref{th:l2phase} contingent on showing the bracket process of the Doob martingale for $\tilde{c}_n^{(\delta)}$ stabilizes in the $n\to \infty$ followed by $\delta \to 0$ limit.  In Section \ref{sec:mgle3} we compute moments of \emph{secular coefficients with restricted cycle count}, whose meaning will become apparent, and finally in Section \ref{sec:mgle4}, we prove the convergence needed to complete the proof of Theorem \ref{th:l2phase}.

In Section \ref{sec:regular} we prove Theorem \ref{thm:hmcregularity} on the regularity of $\mathrm{HMC}_\theta.$
In Section \ref{sec:cbe} we give the 
precise connection between the characteristic polynomial $\chi_N$ and $\mathrm{HMC}_\theta.$  We also show Theorems \ref{th:l2phase-N} and \ref{thm:order}. 
Finally in Section \ref{sec:sharptight}, we prove the sharpness of the estimates in Section \ref{sec:cbe} in the regime $\theta \in (0,1].$


\section*{Acknowledgements}

All three authors would like to thank the hospitality of the International Institute of Physics, in Natal, Brazil, and the program \emph{Random geometries and multifractality in condensed matter and statistical mechanics} from 2019 where this work began.  All authors would like to thank Yacine Barhoumi--Andr\'eani for bringing the mathematics around secular coefficients to their attention and for helpful conversations besides. 
E. P. gratefully acknowledges support from an NSERC Discovery grant.
N. S. gratefully acknowledges support of the Royal Society University Research Fellowship `Random matrix theory and log-correlated Gaussian fields', reference URF\textbackslash R1\textbackslash180707.  

\section{Moments of the secular coefficients}
\label{sec:moments}
The purpose of this section will be to prove Theorems \ref{thm:magic} and \ref{thm:morris} on the moments of $c_{n}$ for general $\beta>0$. We also discuss combinatorial properties of the moments and their relation to Jack functions.

\begin{proof}[Proof of Theorem \ref{thm:magic}]
  Recall from \eqref{cngen} that $c_{n}$ can be extracted from a generating function according to the formula
  \begin{equation}
    c_{n} = [z^{n}]\,\mathrm{exp}\left(\sqrt{\theta}\,\sum_{j=1}^{\infty}\frac{\mathcal{N}_{j}}{\sqrt{j}}\,z^{j}\right).
  \end{equation}
  Denoting the left-hand side of \eqref{genbetmoms} as $\mathcal{R}^{(k)}_{(\mu,\nu)}$, we have
  \begin{equation}
    \label{rcorrelator}
    \mathcal{R}^{(k)}_{(\mu,\nu)} = [z_1^{\mu_1}\ldots z_{k}^{\mu_k}w_1^{\nu_1}\ldots w_{k}^{\nu_k}]\mathbb{E}(F^{(k)}(\vec{z},\vec{w}))
  \end{equation}
  where
  \begin{equation}
    F^{(k)}(\vec{z},\vec{w}) = \mathrm{exp}\left(\sqrt{\theta}\sum_{r=1}^{k}\sum_{j=1}^{\infty}\left(\frac{\mathcal{N}_{j}}{\sqrt{j}}\,z_{r}^{j}+\frac{\overline{\mathcal{N}_{j}}}{\sqrt{j}}\,w_{r}^{j}\right)\right).
  \end{equation}
  A simple Gaussian computation using independence of the family $\{\mathcal{N}_{k}\}_{k=1}^{\infty}$ shows that
  \begin{equation}
    \mathbb{E}(F^{(k)}(\vec{z},\vec{w})) = \prod_{r_1,r_2=1}^{k}\frac{1}{(1-z_{r_1}w_{r_2})^{\theta}}. \label{gausscomp}
  \end{equation}
  Expanding \eqref{gausscomp} with the Newton binomial formula we obtain
  \begin{equation}
  \begin{split}
    \prod_{r_1,r_2=1}^{k}\frac{1}{(1-z_{r_1}w_{r_2})^{\theta}} =& \sum_{A}\prod_{i,j=1}^{k}\binom{A_{ij}+\theta-1}{\theta-1}\\
    &\times \prod_{i=1}^{k}z_{i}^{\sum_{j=1}^{k}A_{ij}}\prod_{j=1}^{k}w_{j}^{\sum_{i=1}^{k}A_{ij}}, \label{newtonbin}
  \end{split}
  \end{equation}
  where the sum runs over the set of all $k \times k$ matrices $A$ whose entries are non-negative integers. Equating coefficients according to \eqref{rcorrelator} fixes the row and column sums appearing in \eqref{newtonbin} and completes the proof of the Theorem.
\end{proof}
We remark that Diaconis and Gamburd \cite{DiaconisGamburd} prove this result specifically for the coefficients $c_{n}^{(N)}$ with $\theta=1$, related to random unitary matrices. In contrast to the above computation, they exploited the known orthogonality of the Schur functions and explicit results associated with the RSK correspondence. When $N=\infty$ their result recovers ours for the coefficients $c_{n}$ with $\theta=1$, but if $\theta \neq 1$ their result is distinct from ours. Despite this, in the following we discuss an interpretation of our result for general $\theta>0$ in terms of Jack functions (which extend the Schur functions to any $\theta >0$). 
\subsection{Connection to Jack functions}
We briefly recall some symmetric function notation. We follow \cite{Stanley} and \cite{Macdonald} for all notational conventions.  We refer the reader to \cite{Stanley} for a concise reminder of the definitions.

Let $\Lambda$ be the algebra of all symmetric formal power series in a countably infinite family of indeterminates.  For any partition $\lambda$ we let $\mathfrak{p}_\lambda$ be the power sum symmetric function, $\mathfrak{e}_\lambda$ be the elementary symmetric function, and $\mathfrak{m}_\lambda$ be the monomial symmetric function.

For any partition $\lambda = (1^{m_1},2^{m_2},3^{m_3}, \dots),$ we let
\[
  z_\lambda= 1^{m_1} \cdot 2^{m_2} \cdot 3^{m_3} \cdots m_1! m_2! m_3! \cdots.
\]
We also define $\ell(\lambda)$ to be the length of a partition.  We define an inner product on $\Lambda$ by
\[
  \langle \mathfrak{p}_\lambda, \mathfrak{p}_\mu \rangle = \mathbf{1}_{\lambda =\mu} z_\lambda \theta^{\ell(\lambda)}.
\]
The Jack functions $\{P_{\lambda}^\theta\}$ form another basis $\Lambda,$ which can be uniquely defined 
by (c.f. \cite[Theorem 1.1]{Stanley}):
\begin{enumerate}
  \item $\langle P_\lambda^{\theta}, P_\mu^{\theta} \rangle = 0$ if $\lambda \neq \mu.$
  \item Expanding the Jack function into monomial basis,
    \[
      P_\lambda^\theta = \sum u_{\lambda \mu}(\alpha) \mathfrak{m}_\mu,
    \]
    all nonzero coefficients $u_{\lambda \mu}(\alpha)$ satisfy $\mu \leq \lambda$ where $\leq$ is the dominance ordering (also known as the ``natural ordering'' in \cite[p.6]{Macdonald}).
  \item The leading coefficient $u_{\lambda \lambda} = 1.$
\end{enumerate}
These specialize to the Schur functions when $\theta=1.$

For any symmetric functions $p,g$ we define another inner product
\begin{equation}
  \langle p,g \rangle_n
  =\frac{1}{Z_{n,\beta}} \int_{\mathbb{T}} p(x)\overline{g(x)} \prod_{i \neq j} \left| x_i - x_j\right|^{1/\theta}\,dx,
  \label{eq:CBEIP}
\end{equation}
that is to say integration against the circular--$\beta$ ensemble.  Here we have specialized the functions $p$ and $g$ by sending all $x_j = 0$ for $j > n.$  The $Z_{n,\beta}$ is the usual normalization so that $\langle 1, 1\rangle = 1.$  
Then for any symmetric functions, $\langle p, g \rangle_n \to \langle p , g\rangle$ (see the discussion below (10.38) in \cite{Macdonald}).  Furthermore, one has that the polynomials $\{P_\lambda^{\theta} : \ell(\lambda) \leq n\}$ are orthogonal with respect to $\langle \cdot, \cdot\rangle_n$ \cite[(10.36)]{Macdonald}.

The Kostka numbers $K_{\lambda \mu}$ can be defined as
\[
  \mathfrak{s}_\lambda = \sum_{\mu} K_{\lambda \mu} \mathfrak{m}_\mu.
\]
As a corollary (apply the $\omega$ involution to \cite[Corollary 7.12.4]{StanleyVol2})
\[
  \mathfrak{e}_\mu = \sum_{\lambda} K_{\lambda' \mu} \mathfrak{s}_\lambda.
\]
Then it is possible to generalize these coefficients to the Jack setting by setting
\begin{equation}\label{eq:Jackkostka}
  \mathfrak{e}_\mu = \sum_{\lambda} K_{\lambda' \mu}^\theta P_\lambda^{\theta}
\end{equation}
The proof of \cite{DiaconisGamburd} exploited an identity for the Kostka numbers, which follows from the RSK bijection (see \cite[Section 7.11,]{StanleyVol2}).  This is given by
\[
  \sum_{\lambda} K_{\lambda \mu}K_{\lambda \nu} = \left| \mathrm{Mag}_{\mu,\nu} \right|
\]
see \cite[Corollary 7.12.3]{StanleyVol2}.

As a corollary of Theorem \ref{thm:magic}, we get a new proof of this fact, as well as a generalization to all $\theta.$  We mention that these connection coefficients are useful for the exact evaluation of some moments of $\beta$-ensembles \cite{MezzadriReynolds}.
\begin{theorem}
  \[
    \sum_{\lambda} K_{\lambda' \mu}^\theta K_{\lambda' \nu}^\theta \langle P_\lambda^\theta, P_\lambda^\theta \rangle
    = 
    \sum_{A \in \mathrm{Mag}_{\mu,\nu}} 
    \prod_{i,j} \binom{ \theta + A_{ij} - 1}{A_{ij}}
  \]
\end{theorem}
\begin{proof}
  Recall $\chi_N$ is the characteristic polynomial of a circular-$\beta$ random matrix,
  and so
  \[
    \chi_N(t) = \sum_{k=0}^N \mathfrak{e}_k( \lambda )t^k,
  \]
  with ${\lambda}$ distributed as C$\beta$E.
  Then from Theorem \ref{thm:hmc},
  \[
    \mathbb{E}(\mathfrak{e}_\mu(\lambda) \overline{\mathfrak{e}_\nu(\lambda)}) \to
    \mathbb{E} \prod_{j=1}^\infty c_j^{\mu_j} \overline{c_j}^{\nu_j}.
  \]
  On the other hand
  \[
    \mathbb{E}(\mathfrak{e}_\mu(\lambda) \overline{\mathfrak{e}_\nu(\lambda)})
    =\langle \mathfrak{e}_\mu, \mathfrak{e}_\nu \rangle_N \to
    \langle \mathfrak{e}_\mu, \mathfrak{e}_\nu \rangle,
  \]
  and hence by \eqref{eq:Jackkostka} and Theorem \ref{thm:magic}
  \[
    \sum_{\lambda} K_{\lambda' \mu}^\theta K_{\lambda' \nu}^\theta \langle P_\lambda^\theta, P_\lambda^\theta \rangle
    = 
    \sum_{A \in \mathrm{Mag}_{\mu,\nu}} 
    \prod_{i,j} \binom{ \theta + A_{ij} - 1}{A_{ij}}.
  \]
  The normalization constant 
  $\langle P_\lambda^\theta, P_\lambda^\theta \rangle$ is given in \cite[(10.16)]{Macdonald}.
\end{proof}

\subsection{Asymptotics of the moments}
In this subsection we will give the proof of Theorem \ref{thm:morris}. It is instructive to begin with a particular case, so we first discuss the case of the fourth moment, or $k=2$ in Theorem \ref{thm:morris}. Then we show how to generalize the approach to all positive integers $k$.
\begin{example}
\label{ex:4thmom}
  Using Theorem \ref{thm:magic}, we can for instance take $k=2$, $\mu_{1}=\mu_{2}=n$, and $\nu_{1}=\nu_{2}=n$ with all other exponents equal to zero. In this case one is summing over all $2 \times 2$ magic squares whose row and column sums are equal to $n$. Such magic squares are parameterized by a single variable (denoted here $j$) and formula Theorem \ref{thm:magic} yields
  \begin{equation}
    \mathbb{E}(|c_{n}|^{4}) = \sum_{j=0}^{n}\binom{j+\theta-1}{\theta-1}^{2}\binom{n-j+\theta-1}{\theta-1}^{2}. \label{4thmom}
  \end{equation}
\end{example}

When $\theta=1$, the right-hand side of \eqref{4thmom} is given by $n+1$ (the number of $2 \times 2$ magic squares), as obtained in \cite{DiaconisGamburd}. For general $\theta>0$ we can compute the asymptotics as $n \to \infty$ from the sum representation \eqref{4thmom} as follows.
\begin{lemma}
  \label{lem:4thmom}
  For any $0 < \theta < 1/2$, we have the following:
  \begin{equation}
    \frac{\mathbb{E}(|c_{n}|^{4})}{\mathbb{E}(|c_{n}|^{2})^{2}} \sim 2\,\frac{\Gamma(1-2\theta)}{\Gamma(1-\theta)^{2}}, \qquad n \to \infty, \label{4thmomlim}
  \end{equation}
  where we recall the normalization
  \begin{equation}
    \mathbb{E}(|c_{n}|^{2}) = \binom{n+\theta-1}{\theta-1} \sim \frac{1}{\Gamma(\theta)}\,n^{\theta-1}.
  \end{equation}
\end{lemma}

\begin{proof}
  Fix $0 < \delta < 1$ and split the sum in \eqref{4thmom} according to whether $j \leq \lfloor \delta n \rfloor$, $\lfloor \delta n \rfloor+1 \leq j \leq n - \lfloor \delta n \rfloor-1$ or $n-\lfloor \delta n \rfloor \leq j \leq n$, denoting each sum $\mathcal{S}_{1}$, $\mathcal{S}_{2}$ or $\mathcal{S}_{3}$ respectively. Then in the sum $\mathcal{S}_{1}$, the term $n-j$ is always large, so that the second binomial coefficient is uniformly bounded by $c_{\delta,\theta}\,n^{\theta-1}$ for some constant $c_{\delta,\theta}>0$. Applying dominated convergence then gives
  \begin{align}
    \lim_{n \to \infty}\frac{\mathcal{S}_{1}}{(\mathbb{E}(|c_{n}|^{2}))^{2}}  &= \sum_{j=0}^{\infty}\binom{j+\theta-1}{\theta-1}^{2}\lim_{n \to \infty}\frac{\binom{n-j+\theta-1}{\theta-1}^{2}}{\binom{n+\theta-1}{\theta-1}^{2}}\\
    & = \sum_{j=0}^{\infty}\binom{j+\theta-1}{\theta-1}^{2} \label{4thmomsum}\\
    &= \frac{\Gamma(1-2\theta)}{\Gamma(1-\theta)^{2}} \label{4thmomlimit},
  \end{align}
  where in \eqref{4thmomlimit} we used Lemma \ref{infsumk}. An identical argument holds for the sum $\mathcal{S}_{3}$ (by symmetry of $j \to n-j$) and this gives the factor $2!$ in \eqref{4thmomlim}. The sum $\mathcal{S}_{2}$ is negligible: the second binomial coefficient is still bounded by $c_{\theta,\delta}n^{\theta-1}$, which gives the same order of magnitude, but now both the upper and lower limits of the sum are growing. Since \eqref{4thmomsum} converges, we have that $n^{-2(\theta-1)}\mathcal{S}_{2} \to 0$. This completes the proof of the Lemma.
\end{proof}

\begin{remark}\label{rem:highermoment}
  When $\theta \geq 1/2$ this argument breaks down because the sum in \eqref{4thmomlimit} is divergent, which leads to slightly different asymptotic behaviour. When $\theta > 1/2$, the sum $\mathcal{S}_{2}$ can be approximated by a Riemann integral which now gives the main contribution: As $n \to \infty$ we have
  \begin{equation}
    \begin{split}
      \mathbb{E}(|c_{n}|^{4})\bigg{|}_{\theta > 1/2} &\sim n^{4(\theta-1)+1}\,\frac{1}{\Gamma(\theta)^{4}}\int_{0}^{1}x^{2(\theta-1)}(1-x)^{2(\theta-1)}\,dx\\
      &= n^{4(\theta-1)+1}\,\frac{\Gamma(2\theta-1)^{2}}{\Gamma(4\theta-2)\Gamma(\theta)^{4}}
    \end{split}
  \end{equation}
  and conversely note that this integral becomes divergent when $\theta \leq 1/2$, with the leading power of $n$ matching at the transition $\theta=1/2$. In the case $\theta=1/2$, one can show that
  \begin{equation}
    \mathbb{E}(|c_{n}|^{4})\bigg{|}_{\theta=1/2} \sim \frac{2}{\pi^{2}}\,\frac{\log n}{n}, \qquad n \to \infty.
  \end{equation}
  This is analogous to the \textit{freezing transitions} in \cite{CK15}.
\end{remark}
The fact that the argument leading to \eqref{4thmomlim} can be generalized to all higher moments is the subject of the next result.

\begin{theorem}
  Let $k$ be a positive integer such that $k\theta < 1$. Then
  \[
    \lim_{n \to \infty}\frac{\mathbb{E}(|c_{n}|^{2k})}{(\mathbb{E}(|c_{n}|^{2}))^{k}} = k!\,\frac{\Gamma(1-k\theta)}{\Gamma(1-\theta)^{k}}\,.
  \]
\end{theorem}
\begin{proof}
  Recall from Theorem \ref{thm:magic}
  \begin{equation}\label{eq:moment}
    \mathbb{E}(|c_n|^{2k})
    =
    \sum_{A \in \mathrm{Mag}_{\pi,\pi}}
    \prod_{i,j}
    \binom{ \theta + A_{ij} - 1}{ A_{ij}},
  \end{equation}
  where $\pi$ is the partition which has $k$ parts of length $n$. In particular the magic squares $\mathrm{Mag}_{\pi,\pi}$ are $k \times k$ and have all row and column sums equal to $n$. Fix a $\delta>0$. Let $E_{\delta,n} \subset \mathrm{Mag}_{\pi,\pi}$ be those magic squares in which there is a row $i$ for which there are two $j$ so that $A_{ij} > \delta n.$  We claim that
  \begin{equation} \label{eq:as0}
    \sum_{A \in E_{\delta,n}}
    \prod_{i,j}
    \binom{ \theta + A_{ij} - 1}{  A_{ij} }
    = o(n^{k(\theta -1)}).
  \end{equation}
  We shall return to this point, but for the moment we give the proof contingent on \eqref{eq:as0}.

  Note that for $\delta$ sufficiently small, each $A \in \mathrm{Mag}_{\pi,\pi} \setminus E_{\delta,n}$ has in each row exactly one entry with size larger than $n(1-k\delta).$  Again for $\delta$ sufficiently small, this implies that each column additionally has exactly one such entry.  Thus for any such $A$ there is a $k\times k$ permutation matrix $P,$ the support of which coincides with the entries of $A$ larger than $n(1-k\delta).$

  Moreover, for each permutation matrix $P$ such $A$ exist and the contributions of each to the sum in \eqref{eq:moment} are all equal.  Let $S \subset \mathrm{Mag}_{\pi,\pi} \setminus E_{\delta,n}$ be those terms in which the diagonal entries are larger than $n(1-k\delta).$  Then using \eqref{eq:as0},
  \begin{equation}\label{eq:as1}
    \mathbb{E}(|c_n|^{2k})
    =
    k!
    \sum_{A \in S}
    \prod_{i = j}
    \biggl\{\binom{\theta + n - 1}{n}
  (1+O(\delta)) \biggr\}
  \cdot
  \prod_{i \neq j}
  \biggl\{
    \binom{ \theta + A_{ij} - 1}{ A_{ij}} \biggr\}
    +
    o( n^{k(\theta -1)}).
  \end{equation}

  Let $H_k$ be the $k \times k$ non-negative integer matrices $A$ in which for each $1 \leq i \leq k$
  \[
    \sum_{j} A_{ij} = \sum_{j} A_{ji}.
  \]
  We claim that
  \begin{equation} \label{eq:as2}
    \sum_{A \in H_k}
    \prod_{i \neq j}
    \binom{ \theta + A_{ij} - 1}{ A_{ij}}
    =\frac{\Gamma(1-k\theta)}{\Gamma(1-\theta)^k}.    
  \end{equation}
  From the Newton binomial formula,
  we have that for $\left\{ z_i \right\}$ and $\left\{ y_j \right\}$ in the unit disk,
  \[
    \prod_{i \neq j}\left( \frac{1}{1-z_i\overline{y_j}} \right)^\theta
    =
    \sum_{ (\mu_j), (\nu_j) \in \mathbb{N}^{k}}
    \prod_{j}
    z_j^{\mu_j} \overline{y_j}^{\nu_j}
    \cdot \biggl\{
      \sum_{A \in \mathrm{Mag}_{0,\mu,\nu}}
      \prod_{i \neq j} \binom{ \theta + A_{ij} - 1}{A_{ij}}
    \biggr\},
  \]
  where
  $\mathrm{Mag}_{0,\mu,\nu} \subset \mathrm{Mag}_{\mu,\nu}$
  have $0$ on the diagonal.
  Hence on setting $z_j = re^{i \omega_j}$ and $y_j = r e^{i \omega_j}$ for $r \in (0,1)$ and averaging over all $\omega_j$, we have
  \[
    \begin{aligned}
      \frac{1}{(2\pi)^k}
      \int
      &\prod_{i \neq j}\left( \frac{1}{1-r^2e^{\sqrt{-1}(\omega_i - \omega_j)}} \right)^\theta
      d\omega \\
      = 
      &\sum_{ (\mu_j) \in \mathbb{N}^{k}}
      r^{2\sum \mu_j}
      \cdot\biggl\{
	\sum_{A \in \mathcal{M}_0(\mu,\mu)}
	\prod_{i \neq j} \binom{ \theta + A_{ij} - 1}{A_{ij}}
      \biggr\}.
    \end{aligned}
  \]
  This integral on the left hand side is convergent on sending $r \to 1,$ and moreover gives exactly the Morris integral \eqref{eq:morris}.
  The right hand side meanwhile converges to the left hand side of \eqref{eq:as2} which completes the proof of \eqref{eq:as2}.

  As \eqref{eq:as2} is convergent, it follows from dominated convergence that
  \[
    \lim_{n \to \infty}
    n^{-k(\theta-1)}
    \sum_{A \in S}
    \prod_{i = j}
    \biggl\{\binom{\theta + n - 1}{n}
  \biggr\}
  \cdot
  \prod_{i \neq j}
  \biggl\{
    \binom{ \theta + A_{ij} - 1}{ A_{ij}} \biggr\}
    =
    \frac{\Gamma(1-k\theta)}{\Gamma(1-\theta)^k \Gamma(\theta)^k}.
  \]
  Hence, on combining this with \eqref{eq:as1} and sending $\delta \to 0,$ we conclude that
  \[
    \lim_{n \to \infty}
    n^{-k(\theta-1)}
    \mathbb{E}(|c_n|^{2k})
    =
    \frac{k!\Gamma(1-k\theta)}{\Gamma(1-\theta)^k \Gamma(\theta)^k}.
  \]

  Finally we turn to the proof of \eqref{eq:as0}. By the Birkhoff--von Neumann theorem, the doubly stochastic matrices are the convex hull of the permutation matrices.  It follows that for every $A \in \mathrm{Mag}_{\pi,\pi}$ there is a permutation matrix $P$ such that each entry of $A$ in the support of $P$ has size at least $n/k!$  Hence by symmetry it suffices to restrict the sum in \eqref{eq:as0} to those $A \in E_{\delta,n}$ whose every diagonal entry is at least $n/k!$. Denote this subset of $E_{\delta,n}$ by $E'_{\delta_n}$. In short we have the bound
  \[
    \sum_{A \in E_{\delta,n}}
    \prod_{i,j}
    \binom{ \theta + A_{ij} - 1}{  A_{ij} }
    \leq
    k!
    \sum_{A \in E_{\delta,n}'}
    \prod_{i}
    \biggl\{
      C_\theta \left( \frac{n}{k!} \right)^{\theta -1}
    \biggr\}
    \prod_{i\neq j}
    \biggl\{
      \binom{ \theta + A_{ij} - 1}{  A_{ij} }
    \biggr\},
  \]
  for some absolute constant $C_\theta.$
  Hence once more using the absolute convergence of \eqref{eq:as2} and dominated convergence, \eqref{eq:as0} follows.
\end{proof}

\section{The Ewens sampling formula}
\label{sec:ewens}
In this section we describe a connection between secular coefficients and random permutations that we believe is interesting in its own right. While detailing this we shall revise some of the key results about random permutations, as these will be useful in the subsequent sections.

The secular coefficients defined in \eqref{eq:cn} can be given the following explicit formula
\begin{equation}
  c_{n} = \sum_{m \in S_{n}}\,\prod_{k=1}^{n}\frac{\mathcal{N}_{k}^{m_{k}}}{m_{k}!}\,\left(\frac{\theta}{k}\right)^{m_{k}/2}. \label{cnsum}
\end{equation}
The summation is over the set $S_{n}$ of all compositions of $n$, that is $m = (m_1,m_2,\ldots,m_n)$ is such that each $m_{k}$ is a non-negative integer satisfying
\begin{equation}
  \sum_{k=1}^{n}km_{k} = n.
\end{equation}
From \eqref{cnsum}, we have that $c_n$ is measurable with respect to $\Gfilt_n \coloneqq \sigma\left\{ \mathcal{N}_1, \dots, \mathcal{N}_n \right\}.$  So, going forward, we will make use of the filtration
\( \Gfilt = ( \Gfilt_n : n \in \N)\).

Taking the $L^{2}$-norm of \eqref{cnsum} gives
\begin{equation}
  \mathbb{E}(|c_{n}|^{2}) = \sum_{\substack{\vec{l} \in S_{n},\vec{m}\in S_{n}}}\prod_{k=1}^{n}\frac{\mathbb{E}(\mathcal{N}_{k}^{m_k}\overline{\mathcal{N}_{k}}^{l_k})}{m_{k}!l_{k}!}\left(\frac{\theta}{k}\right)^{(m_{k}+l_{k})/2} \label{l2first}
\end{equation}
where we used independence of the family $\{\mathcal{N}_{k}\}_{k=1}^{\infty}$ to take the expectation inside the product. Next applying the Gaussian formula
\begin{equation}
  \mathbb{E}(\mathcal{N}_{k}^{m_k}\overline{\mathcal{N}_{k}}^{l_k}) = \mathbbm{1}_{m_k=l_k}(m_k)! \label{gausscov}
\end{equation} 
implies that the compositions in the sum \eqref{l2first} must coincide in order to give a non-zero term. This gives
\begin{equation}
  \mathbb{E}(|c_{n}|^{2}) = \sum_{\substack{\vec{m}\in S_{n}}}\prod_{k=1}^{n}\frac{1}{m_{k}!}\left(\frac{\theta}{k}\right)^{m_{k}}. \label{l2sum}
\end{equation}
Given a permutation $\sigma$ on $n$ symbols, we can characterize it using its cycle structure $(m_1,\ldots,m_n)$ where $m_{j}$ denotes the number of cycles in $\sigma$ having length $j$, and $\sum_{j=1}^{n}jm_{j} = n$. The summation in \eqref{l2sum} is well known in the theory of random permutations where it appears as the normalizing factor for the Ewens sampling formula.
\begin{definition}
  The \textit{Ewens sampling formula} is a probability distribution on cycle counts $\vec{M}^{(n)} = (M_1,\ldots,M_n)$ given by
  \begin{equation}
    \mathbb{P}(\vec{M}^{(n)} = m) = \mathbbm{1}_{\sum_{k=1}^{n}km_{k}=n}\,\frac{n!}{\theta^{(n)}}\prod_{k=1}^{n}\left(\frac{\theta}{k}\right)^{m_k}\frac{1}{m_{k}!}. \label{ewens}
  \end{equation}
  where $\theta^{(n)}$ is the rising factorial 
  \begin{equation}
    \theta^{(n)} = \frac{\Gamma(\theta+n)}{\Gamma(\theta)}.
  \end{equation}
\end{definition}
The fact that \eqref{ewens} is normalized gives the explicit form of the sum in \eqref{l2sum}:
\begin{equation}
  \mathbb{E}(|c_{n}|^{2}) = \frac{\theta^{(n)}}{n!}.
\end{equation}
In general, if we restrict the summation in \eqref{cnsum} to a subset $P \subset S_{n}$ and denote this $c_{n,P}$, an identical computation holds. This yields the fundamental correspondence
\begin{equation}
  \mathbb{E}(|c_{n,P}|^{2}) = \frac{\theta^{(n)}}{n!}\,\mathbb{P}(\vec{M}^{(n)} \in P) \label{corres}
\end{equation}
where on the left-hand side of \eqref{corres} the expectation is taken over the Gaussian random variables in \eqref{cnsum}, while on the right-hand side $\vec{M}$ follows the Ewens sampling formula in \eqref{ewens} with parameter $\theta$.

We can also use the Ewens sampling formula to describe conditional expectations of $|c_n|^2.$
In analogy with \eqref{l2first}, for $q \leq n,$
\begin{equation} \label{eq:l2cond}
  \mathbb{E}(|c_{n}|^{2} ~\vert~ \Gfilt_q) 
  = \sum_{\substack{\vec{l} \in S_{n},\vec{m}\in S_{n}}}
  \prod_{k=1}^{n}
  \frac{\mathbb{E}(\mathcal{N}_{k}^{m_k}\overline{\mathcal{N}_{k}}^{l_k} ~\vert~ \Gfilt_q)}
  {m_{k}!l_{k}!}
  \left(\frac{\theta}{k}\right)^{(m_{k}+l_{k})/2}. 
\end{equation}
The only nonzero pairs $(\vec{l},\vec{m})$ in this sum have $m_k = l_k$ for $k>q,$ and this allows us to greatly simplify this expression. 
Let us define
\begin{equation}
  c_{n,q} = \sum_{\substack{(m_{k}) : 1 \leq k \leq q\\ \sum_{k=1}^{q}km_{k}=n}}\prod_{k=1}^{q}\frac{\theta^{m_{k}/2}}{m_{k}!k^{m_k/2}}\,\mathcal{N}_{k}^{m_k}. \label{cnq}
\end{equation}
\begin{remark}
  \label{rem:cnm}
  This \eqref{cnq} is a special case of $c_{n,P}$ (c.f.\ \eqref{corres}), in which $P$ are those partitions with no parts greater than $q$. This is also equivalent to setting the Gaussians $\mathcal{N}_{k} = 0$ for all $k > q$ in the sum \eqref{cnsum}. In particular $c_{n,q}$ is $\Gfilt_q$--measruable.
\end{remark}

In terms of \eqref{cnq} we can therefore give the sum formula by partitioning \eqref{eq:l2cond} according to $r=\sum_{k=1}^q km_k.$  Note that when $r \geq n-q,$ we have no way to complete the partition except by choosing all larger $m_k =0,$
and so
\[
  \mathbb{E}(|c_{n}|^{2} ~\vert~ \Gfilt_q) 
  =
  |c_{n,q}|^2
  +
  \sum_{r=0}^{n-q-1} 
  |c_{r,q}|^2 
  \sum_{\substack{(m_{k}) : q < k \leq n\\ \sum_{k=q+1}^{n-q-1}km_{k}=n-r}}
  \prod_{k=q+1}^{n}
  \frac{\theta^{m_{k}}}{m_{k}!k^{m_k}}.
\]
Note that we may have no $m_k > 0$ for $k > n-r$ in the inner sum
and so using \eqref{ewens} 
\begin{equation} \label{eq:l2conda}
  \mathbb{E}(|c_{n}|^{2} ~\vert~ \Gfilt_q) 
  =
  |c_{n,q}|^2
  +
  \sum_{r=0}^n |c_{r,q}|^2 \frac{ \theta^{(n-r)}}{(n-r)!} 
  \Pr[ 
    \vec{M}_j^{(n-r)} = 0, \text{ for all } 1 \leq j \leq q
  ].
\end{equation}
We note that sum need only run to $n-q,$ as the probability therein is $0$ for larger $r.$
We will do an asymptotic analysis of this conditional expectation in Section \ref{sec:qmoments}.

\subsection{Properties of the Ewens sampling formula}

We remark that the case $\theta=1$ in \eqref{ewens} corresponds to the uniform measure on the set of all permutations, while the general case $\theta>0$ corresponds to a tilting of the uniform measure. For any $\theta>0$, a wealth of results are known concerning statistical properties of the cycle counts, see the text \cite{ABT03} from which we will borrow from repeatedly in what follows. The main point exploited in \cite{ABT03} is that apart from the indicator function, \eqref{ewens} is a conditional joint law of $n$ independent random variables $(Z_1,\ldots,Z_{n})$ where each $Z_{k}$ is Poisson distributed with parameter $\theta/k$. Therefore, statistics of the cycle counts can be reduced to studying the independent random variables $(Z_1,\ldots,Z_n)$ paired with the condition that $T_{0n}=n$ where $T_{0n} = \sum_{j=1}^{n}jZ_{j}$. We will refer to this as the \textit{conditioning relation}.

In fact, the random variable $T_{0n}$ and its limiting distribution play an important role in what follows and we record some of the key results about it from \cite{ABT03}.
\begin{lemma}
  \label{lem:t0n}
  Suppose that $r = r_{n} \in \mathbb{N}$ satisfies $r/n \to y \in (0,\infty)$ as $n \to \infty$. Then
  \begin{equation}
    \lim_{n \to \infty}n\mathbb{P}(T_{0n}=r) = p_{\theta}(y) \label{convt0n}
  \end{equation}
  where $p_{\theta}(y)$ is an (explicit) probability density function satisfying the following properties:
  \begin{enumerate}
    \item An explicit formula at $y=1$:
      \begin{equation}
	p_{\theta}(1) = \frac{e^{-\gamma_{\mathrm{E}}\theta}}{\Gamma(\theta)}
      \end{equation}
      where $\gamma_{\mathrm{E}}$ is the Euler-Mascheroni constant.
    \item Rapid decay at $y = +\infty$:
      \begin{equation}
	\mathrm{sup}_{y \geq n}p_{\theta}(y) \leq \frac{\theta^{n}}{n!} \label{rapid}
      \end{equation}
    \item The derivative identity, for $x \not\in \{0,1\}$
      \begin{equation}
	\frac{d}{dx}[x^{1-\theta}p_{\theta}(x)] = -\theta x^{-\theta}p_{\theta}(x-1).
      \end{equation}
  \end{enumerate}
\end{lemma}
\begin{proof}
  These are proved in \cite[Section 4]{ABT03} using size biasing techniques.
\end{proof}
We will also make use of the following finite $n$ uniform bound.
\begin{lemma}[Lemma 4.12 (i) in \cite{ABT03}]
  \label{lem:unifbnd}
  If $0 \leq \theta \leq 1$, then
  \begin{equation}
    \mathrm{max}_{k \geq 0}\, \mathbb{P}(T_{0n}=k) \leq e^{-\theta\,h(n+1)}
  \end{equation}
  where $h(n+1)$ is the harmonic number
  \begin{equation}
    h(n+1) = \sum_{j=1}^{n}\frac{1}{j}.
  \end{equation}
\end{lemma}
Now let $L^{(n)}$ denote the length of the longest cycle. This quantity is closely related to the distribution of $T_{0n}$, because by the conditioning relation
\begin{equation} \label{eq:Lndef}
  \begin{split}
    \mathbb{P}(L^{(n)} \leq r) &= \mathbb{P}(\{c_{r+1}=0\} \land \ldots \land \{c_{n}=0\})\\
    &= \mathbb{P}(\{Z_{r+1}=0\} \land \ldots \land \{Z_{n}=0\} \, | \, T_{0n}=n)\\
    &= \mathbb{P}(\{Z_{r+1}=0\} \land \ldots \land \{Z_{n}=0\})\frac{\mathbb{P}(T_{0r}=n)}{\mathbb{P}(T_{0n}=n)}.
  \end{split}
\end{equation}
Now Lemma \ref{lem:t0n} gives the following, as quoted in \cite[Lemma 4.23]{ABT03} and attributed to Kingman (1977).
\begin{lemma}
  \label{lem:kingman}
  As $n \to \infty$, we have the convergence in distribution $n^{-1}L^{(n)} \overset{d}{\longrightarrow} L^{(\infty)}$ where $L^{(\infty)}$ is a random variable with distribution function $F_{\theta}$ given by
  \begin{equation}
    F_{\theta}(x) = e^{\gamma_{\mathrm{E}}\theta}x^{\theta-1}\Gamma(\theta)p_{\theta}(1/x).
  \end{equation}
\end{lemma}

We also need a tail bound that controls the probability that the largest cycle is unusually small.
\begin{lemma}\label{lem:Lnlowertail}
  For any $\theta > 0,$
  there is a constant $c_\theta > 0$ so that
  \[
    \Pr( L^{(n)} \leq n/\log n) \leq n^{-c_\theta \log\log n}.
  \]
\end{lemma}
\begin{proof}
  Using \eqref{eq:Lndef}, and the convergence of $\Pr(T_{0n} =n)$, there is some constant $C_\theta$ so that for any $r \in \N$
  \[
    \Pr( L^{(n)} \leq r)
    \leq
    \frac{\Pr( T_{0r} = n)}{\Pr( T_{0n} = n)}
    \leq
    C_\theta
    \Pr( T_{0r} = n).
  \]
  When $r$ is much smaller than $n$ this probability becomes very small.
  Using standard concentration results for functionals of Poisson fields (see \cite[Proposition 3.1]{Wu} or \cite{Ledoux}),
  there is a constant $C_\theta$ sufficiently large that for all $t > 0$ and all $r \in \N,$
  \[
    \Pr( T_{0r} > rt) \leq \exp(-\tfrac{t}{4}\log(1+\tfrac{t}{C_\theta})).
  \]
  The proof follows by taking $r = \lfloor n/\log n \rfloor $ and $t = \lfloor \log n \rfloor.$ 
\end{proof}

To analyze \eqref{eq:l2conda} we will also need information on the shortest cycle $S^{(n)}$ in a Ewens distributed permutation.  From the asymptotic independence of the cycle counts in a Ewens permutation, one expects $\Pr(S^{(n)} > q) \to e^{-\theta h(q+1)}$ as $n \to \infty.$
We will make use of a nonasymptotic bound that has this behavior.
\begin{lemma} \label{lem:shortees}
  For all $\theta > 0$ there is a constant $C_\theta >0$
  so that
  for all $n,q \in \N$
  \[
    \Pr(S^{(n)} > q)
    \leq 
    \frac{\theta (q-1)!}{\theta^{(q)}}
    \leq C_\theta e^{-\theta h(q+1)}.
  \]
\end{lemma}
\begin{proof}
  We use the \emph{Feller} description of the Ewens sampling formula (see \cite[Section 3]{ArratiaTavare2}). Let $\xi_1,\xi_2,\xi_3, \dots$ be independent Bernoulli variables having parameter
  \[
    \Pr(\xi_j = 1) = \frac{\theta}{\theta+j-1}, \quad \text{for all } j \in \N.
  \]
  The Ewens distribution on $(\vec{M}_k^{(n)})_{k=1}^n$ has the same law as the joint frequency of spacings between consecutive ones in the vector
  \[
    (1, \xi_n, \xi_{n-1}, \dots, \xi_3, \xi_2, 1).
  \]
  In this case a spacing of length $\ell$ is a pair $\{k,k+\ell\}$ where $\xi_{k} = \xi_{k+\ell} = 1$ and $\xi_j = 0$ for all $j$ with $k < j < k+\ell.$  In the case $k+\ell=q+1,$ we instead use $1$ in place of $\xi_{k+\ell}.$
  In particular, we have the inequality
  \[
    \begin{aligned}
      \Pr\left[ S^{(n)} > q \right]
      &\leq
      \Pr\left[
	\xi_k = 0, \text{ for all } 2 \leq k \leq q
      \right]     \\
      &= \prod_{k=2}^q \left( 1- \frac{\theta}{\theta+k-1} \right) \\
      &= \frac{\theta (q-1)!}{\theta^{(q)}}.
    \end{aligned}
  \]
\end{proof}

\section{Martingale approximation and convergence}
\label{sec:mgle1}
We begin this section by finding an approximate martingale structure in the sum \eqref{cnsum}. This opens the possibility of applying a central limit theorem for martingales and we explain why this is relevant for our proof of Theorem \ref{th:l2phase}. The results on permutations in the previous section will be used to arrive at the martingale approximation, as we now show. 
\subsection{Martingale approximation}
Given a composition $m$, we define
\begin{equation}
  \nu(m) := \mathrm{max}\{k=1,\ldots,n \mid m_{k} \geq 1\}.
\end{equation}
If the $m_k$ are interpreted as cycle counts, the quantity $\nu(m)$ is the length of the largest cycle in the corresponding permutation. Given $\delta > 0$, consider permutations whose largest cycle is smaller than $\lfloor \delta n \rfloor$:
\begin{equation}
  P_{\delta} := \{m \in S_{n} \mid \nu(m) < \lfloor \delta n \rfloor\}. \label{Pdeltaset}
\end{equation}
We will show that the contribution of $P_{\delta}$ to the sum in \eqref{cnsum} can be neglected for large $n$ and small $\delta$. In the sum over the remaining terms $P_{\delta}^{\mathrm{c}}$ we define another negligible set, the set of all compositions where there are multiple longest cycles of the same length:
\begin{equation}
  R := \{m \in S_{n} \mid m_{\nu(m)} \geq 2\} \label{Rset}
\end{equation}
We define $\tilde{c}_{n}^{(\delta)}$ to be the sum over compositions whose largest cycle is \textit{greater} than $\lfloor \delta n \rfloor$, and where there is only one such largest cycle, in other words $\tilde{c}_{n}^{(\delta)} := c_{n,P_{\delta}^{\mathrm{c}}\cap R^{\mathrm{c}}}$. We have
\begin{equation}
  c_{n} = \tilde{c}_{n}^{(\delta)}+c_{n,P_{\delta}}+c_{n,P_{\delta}^{\mathrm{c}}\cap R} \label{decomp}
\end{equation}
and summing over the possible lengths of the longest cycle, we can decompose the sum as
\begin{equation}
  \tilde{c}_{n}^{(\delta)} = \sum_{q=\lfloor \delta n \rfloor}^{n}\mathcal{N}_{q}\sqrt{\frac{\theta}{q}}\,c_{n-q,q-1} \label{martingalesum}
\end{equation}
where $c_{n-q,q-1}$ are as in \eqref{cnq}.

\begin{lemma}[Martingale approximation]
  \label{le:martingalerep}
  Let $0 < \delta < 1$ and assume $\theta \in [0,1]$. Then $c_{n}$ is well approximated by $\tilde{c}_{n}^{(\delta)}$ given by \eqref{martingalesum} in the following sense. After proper normalization, the error terms $c_{n,P_{\delta}}$ and $c_{n,P_{\delta}^{\mathrm{c}}\cap R}$ in \eqref{decomp} have $L^{2}$-norm satisfying
  \begin{align}
    &\lim_{\delta \to 0}\lim_{n \to \infty}\frac{\mathbb{E}(|c_{n,P_{\delta}}|^{2})}{\mathbb{E}(|c_{n}|^{2})} = 0 \label{statement1}\\
    &\frac{\mathbb{E}(|c_{n,P_{\delta}^{\mathrm{c}}\cap R}|^{2})}{\mathbb{E}(|c_{n}|^{2})} = O(n^{-\theta}), \qquad n \to \infty. \label{statement2}
  \end{align}
  Taken together, using \eqref{decomp} we have
  \begin{equation}
    \lim_{\delta \to 0}\lim_{n \to \infty}\frac{\mathbb{E}(|c_{n}-\tilde{c}_{n}^{(\delta)}|^{2})}{\mathbb{E}(|c_{n}|^{2})} = 0.
  \end{equation}
\end{lemma}

\begin{proof}
  By the correspondence \eqref{corres}, we have
  \begin{equation}
    \frac{\mathbb{E}(|c_{n,P_{\delta}}|^{2})}{\mathbb{E}(|c_{n}|^{2})} = \mathbb{P}(L^{(n)}_{1} < \lfloor \delta n \rfloor)
  \end{equation}
  where the right-hand side is the probability that the longest cycle is bounded by $\lfloor \delta n \rfloor$. By Lemma \ref{lem:kingman}, we have
  \begin{equation}
    \lim_{n \to \infty}\mathbb{P}(L^{(n)}_{1} < \lfloor \delta n \rfloor) = e^{\gamma_{\mathrm{E}}\theta}\delta^{\theta-1}\Gamma(\theta)p_{\theta}(1/\delta)
  \end{equation}
  and the fast decay \eqref{rapid} gives
  \begin{equation}
    \lim_{\delta \to 0}\lim_{n \to \infty}\mathbb{P}(L^{(n)}_{1} < \lfloor \delta n \rfloor) = 0
  \end{equation}
  which proves statement \eqref{statement1}. For \eqref{statement2} we have
  \begin{align}
    &\mathbb{P}(P^{\delta} \cap R) = \sum_{q=\lfloor \delta n \rfloor}^{n}\mathbb{P}(\{m_{q} \geq 2\} \land \{\nu(m)=q\})\\
    &=\sum_{q=\lfloor \delta n \rfloor}^{n}\mathbb{P}(\{m_{q} \geq 2\} \land \{m_{q+1}=0\} \land \ldots \land \{m_{n}=0\})\\
    &=\sum_{q=\lfloor \delta n \rfloor}^{n}\sum_{\ell=2}^{\infty}\mathbb{P}(\{Z_{q} = \ell\} \land \{Z_{q+1}=0\} \land \ldots \land \{Z_{n}=0\})\frac{\mathbb{P}(T_{0(q-1)} = n-q\ell)}{\mathbb{P}(T_{0n}=n)}
  \end{align}
  where in the last step we employed the conditioning relation and we interpret contributions to the sum above as zero whenever $q\ell > n$. In order to bound the quantities above, first take $r=n$ in \eqref{convt0n} to see that $\mathbb{P}(T_{0n}=n) \sim n^{-1}p_{\theta}(1)$. By Lemma \ref{lem:unifbnd} we have
  \begin{equation}
    \mathbb{P}(T_{0(q-1)} = n-q\ell) \leq e^{-\theta\,h(q)}
  \end{equation}
  while
  \begin{equation}
    \mathbb{P}(\{Z_{q+1}=0\} \land \ldots \land \{Z_{n}=0\}) = e^{-\theta\,h(n+1)+\theta\,h(q+1)}.
  \end{equation}
  A simple bound on the harmonic number gives $e^{-\theta\,h(n+1)} \leq c_{\theta}n^{-\theta}$ for some positive constant $c_{\theta}$. Combining these facts gives, for some other constant $c_{\theta}$, the bound
  \begin{equation}
    \mathbb{P}(P^{\delta} \cap R) \leq n^{1-\theta}\,c_{\theta}\,\sum_{q=\lfloor \delta n \rfloor}^{n}\mathbb{P}(Z_{q} \geq 2) = O_{\delta}(n^{-\theta}), \qquad n \to \infty,\\
  \end{equation}
  where we used that $\mathbb{P}(Z_{q} \geq 2) \sim \frac{\theta^{2}}{2q^{2}}$ as $q \to \infty$. This completes the proof of the martingale approximation.
\end{proof}

\subsection{Martingale central limit theorem}
\label{sec:mgle2}
Having demonstrated the martingale approximation, in this section we will discuss the type of central limit theorems we can apply. In the following we will give a proof of Theorem \ref{th:l2phase} contingent on a certain $L^{2}$ estimate that will be dealt with separately in a later section.

To recap, Lemma \ref{le:martingalerep} states that if we define increments
\begin{equation}
  Z_{n,q} := \frac{1}{\sqrt{\mathbb{E}(|c_{n}|^{2})}}\,\mathcal{N}_{q}\sqrt{\frac{\theta}{q}}\,c_{n-q,q-1} \label{cincrements}
\end{equation}
then the quantity
\begin{equation}
  \frac{\tilde{c}_{n}^{(\delta)}}{\sqrt{\mathbb{E}(|c_{n}|^{2})}} = \sum_{q=\lfloor \delta n \rfloor}^{n}Z_{n,q}
\end{equation}
is a good approximation of the normalized coefficient $c_n/\sqrt{\Exp(|c_n|^2)}.$
Now we see by construction (recalling Remark \ref{rem:cnm}) that $c_{n-q,q-1}$ only depends on the first $q-1$ Gaussians and this implies that the random variables $\mathcal{N}_{q}$ and $c_{n-q,q-1}$ are independent. We have
\begin{equation}
  \mathbb{E}(Z_{n,q} \,|\, \Gfilt_{q-1}) = 0,
\end{equation}
in other words $Z_{n,q}$ are the increments of a martingale with respect to the filtration generated by the first $q-1$ Gaussians.

In order to get convergence in distribution of $\tilde{c}_{n}^{(\delta)}$ (and therefore of $c_{n}$), we will apply a central limit theorem for martingales. The majority of these limit theorems rely to a large extent on the analysis of a quantity known as the \textit{bracket process} (sometimes also referred to as the \textit{conditional variance} \cite{HH80}). In our setting it is given by
\begin{equation}
  \begin{split}
    \mathcal{M}_{\theta,\delta,n}&:=\sum_{q=\lfloor \delta n \rfloor}^{n}\mathbb{E}(|Z_{n,q}|^{2} \,|\, \Gfilt_{q-1})
    = \frac{1}{\mathbb{E}(|c_{n}|^{2})}\,\sum_{q=\lfloor \delta n \rfloor}^{n}\frac{\theta}{q}\,|c_{n-q,q-1}|^{2}. \label{condvar}
  \end{split}
\end{equation}
A key hypothesis usually involves showing that this type of quantity converges in probability to a constant as $n \to \infty$. An interesting feature here is that $\mathcal{M}_{\theta,\delta,n}$ will not have a deterministic limit, and so we need a sufficiently general form of the martingale CLT that allows for fluctuations in the limit $n \to \infty$ of the bracket process. These limiting fluctuations will be described in terms of the total mass of Gaussian multiplicative chaos and are ultimately responsible for the structure of the distribution given in Theorem \ref{th:l2phase}. The appropriate CLT is the following:
\begin{theorem}[Martingale Central Limit Theorem - Section 3.2 in \cite{HH80}]
  Let $\{X_{n,q}\}_{q=1}^{n}$ for $n\geq 1$ be an array of real valued martingale increments with respect to a filtration $\mathcal{F}_{n,q}$ indexed by $q$. Define the random variables
  \begin{align}
    \nu_{n} &:= \sum_{q=1}^{n}\mathbb{E}(X_{n,q}^{2} \mid \mathcal{F}_{n,q-1}), \\
    \xi_{n} &:= \sum_{q=1}^{n}\mathbb{E}(X_{n,q}^{2}\mathbbm{1}_{|X_{n,q}|>\epsilon} \mid \mathcal{F}_{n,q-1})
  \end{align}
  Suppose we have the convergence in probability $\nu_{n} \overset{p}{\longrightarrow} \nu$ where $\nu$ is an a.s. finite random variable, and $\xi_{n} \overset{p}{\longrightarrow} 0$. Then we have the convergence in distribution
  \begin{equation}
    \sum_{q=1}^{n}X_{n,q} \overset{d}{\longrightarrow} \sqrt{\nu}\,N_{\mathbb{R}}, \qquad n \to \infty,
  \end{equation}
  where $N_{\mathbb{R}}$ is a standard (real) Gaussian, independent of $\nu$.
  \label{th:mart}
\end{theorem}
Note that \cite{HH80} only deals with real valued random variables, while \eqref{cincrements} are complex. Although it is probably straightforward to generalise their result to the complex case, for our particular problem the simple {i.i.d.}\ structure of the real and imaginary parts of $\mathcal{N}_{q}$ allows us to apply Theorem \ref{th:mart} directly.
\begin{corollary}
  \label{cor:realtocomplex}
  Suppose we have the convergence in probability of the quantity defined in \eqref{condvar},
  \begin{equation}
    \mathcal{M}_{\theta,\delta,n} \overset{p}{\longrightarrow} \nu, \qquad n \to \infty \label{condvarcor}
  \end{equation}
  where $\nu$ is a.s. finite and 
  \begin{equation}
    \frac{1}{\mathbb{E}(|c_{n}|^{2})^{2}}\sum_{q=\lfloor \delta n \rfloor}^{n}\frac{\theta^{2}}{q^{2}}\mathbb{E}|c_{n-q,q-1}|^{4} \to 0, \qquad n \to \infty. \label{lindebergcor}
  \end{equation}
  Then it follows that we have the convergence in distribution,
  \begin{equation}
    \frac{\tilde{c}_{n}^{(\delta)}}{\sqrt{\mathbb{E}(|c_{n}|^{2})}} \overset{d}{\longrightarrow} \sqrt{\nu}\,\mathcal{N}_{1}, \qquad n \to \infty, \label{complexconv}
  \end{equation}
  where $\mathcal{N}_{1}$ is a standard complex Gaussian, independent of $\nu$.
\end{corollary}

\begin{proof}
  For any pair $(a,b) \in \mathbb{R}^{2}$, consider the real valued martingale
  \begin{equation}
    \frac{\tilde{c}_{n,a,b}}{\sqrt{\mathbb{E}(|c_{n}|^{2})}} := \frac{1}{\sqrt{\mathbb{E}(|c_{n}|^{2})}}\,\left(a\mathrm{Re}(\tilde{c}_{n}^{(\delta)})+b\mathrm{Im}(\tilde{c}_{n}^{(\delta)})\right) = \sum_{q=\lfloor \delta n \rfloor}^{n}X_{n,q}
  \end{equation}
  with filtration generated by the real and imaginary parts of $\mathcal{N}_{1},\ldots,\mathcal{N}_{q-1}$ and, by \eqref{cincrements}, increments given by
  \begin{equation}
    X_{n,q} = \frac{1}{\sqrt{\mathbb{E}(|c_{n}|^{2})}}\,\sqrt{\frac{\theta}{q}}\,\left(a\mathrm{Re}(\mathcal{N}_{q}c_{n-q,q-1})+b\mathrm{Im}(\mathcal{N}_{q}c_{n-q,q-1})\right).
  \end{equation}
  We will now check the conditions of Theorem \ref{th:mart}. A straightforward computation using $\mathbb{E}(\mathcal{N}_{q}^{2}) = \mathbb{E}(\overline{\mathcal{N}_{q}}^{2}) = 0$ and $\mathbb{E}(|\mathcal{N}_{q}|^{2}) = 1$ shows that
  \begin{equation}
    \nu_{n} \coloneqq
    \sum_{q=\lfloor \delta n \rfloor}^{n}\mathbb{E}(X_{n,q}^{2} \,|\,  \Gfilt_{q-1}) 
    = \frac{a^{2}+b^{2}}{2}\,\mathcal{M}_{\theta,\delta,n}.
  \end{equation}
  So \eqref{condvarcor} implies that 
  \begin{equation}
    \nu_{n} \overset{p}{\longrightarrow} \frac{a^{2}+b^{2}}{2}\,\nu.
  \end{equation}
  For the Lindeberg type condition we have the bound
  \begin{equation}
    X_{n,q}^{2}1_{|X_{n,q}|>\epsilon} \leq \frac{1}{\epsilon^{2}}|X_{n,q}|^{4} \leq \frac{(a^{2}+b^{2})^{2}}{\epsilon^{2}}|Z_{n,q}|^{4}
  \end{equation}
  where the last inequality follows from Cauchy-Schwarz. By \eqref{lindebergcor} we have
  \begin{equation}
    \mathbb{E}|\xi_{n}| \leq 2\theta^{2}\,\frac{(a^{2}+b^{2})^{2}}{\mathbb{E}(|c_{n}|^{2})\epsilon^{2}}\,\sum_{q = \lfloor \delta n \rfloor}^{n}\frac{\mathbb{E}(|c_{n-q,q-1}|^{4})}{q^{2}} \to 0
  \end{equation}
  which implies $\xi_{n} \overset{p}{\longrightarrow} 0$. Now Theorem \ref{th:mart} applies and shows that we have the convergence in distribution
  \begin{equation}
    \begin{split}
      \frac{\tilde{c}_{n,a,b}}{\sqrt{\mathbb{E}(|c_{n}|^{2})}} &\overset{d}{\longrightarrow} \sqrt{\nu}\,\sqrt{\frac{a^{2}+b^{2}}{2}}N_{\mathbb{R}}\\
      &\overset{d}{=} \sqrt{\nu}\,\left(a\,\frac{N^{(1)}_{\mathbb{R}}}{\sqrt{2}}+b\,\frac{N^{(2)}_{\mathbb{R}}}{\sqrt{2}}\right)
    \end{split}
  \end{equation}
  where $N_{\mathbb{R}}, N^{(1)}_{\mathbb{R}}$ and $N^{(2)}_{\mathbb{R}}$ are independent and identically distributed standard (real) Gaussians. This establishes the joint convergence of the real and imaginary parts of $\tilde{c}_{n}^{(\delta)}$ to the appropriate limit, and concludes the proof of \eqref{complexconv}.
\end{proof}

We start by checking the Lindeberg condition \eqref{lindebergcor}, which is relatively straightforward.
\begin{lemma}[Lindeberg condition]
  The condition \eqref{lindebergcor} is satisfied for any $0 < \theta < 1$.
\end{lemma}
\begin{proof}
Recalling the explicit normalization \eqref{eq:2ndmom} for $\mathbb{E}(|c_{n}|^{2})$, we have
  \begin{equation}
    \begin{split}
      & \frac{\theta^{2}}{\binom{n+\theta-1}{\theta-1}^{2}}\sum_{q=\lfloor \delta n \rfloor}^{n}\frac{\mathbb{E}(|c_{n-q,q-1}|^{4})}{q^{2}}\leq \frac{\theta^{2}}{\binom{n+\theta-1}{\theta-1}^{2}\lfloor \delta n \rfloor^{2}}\sum_{q=0}^{n}\mathbb{E}(|c_{q}|^{4}) \label{cqbound}\\
      &=\frac{\theta^{2}}{\binom{n+\theta-1}{\theta-1}^{2}\lfloor \delta n \rfloor^{2}}\sum_{k=0}^{n}\binom{k+\theta-1}{\theta-1}^{2}\sum_{q=0}^{n-k}\binom{q+\theta-1}{\theta-1}^{2}\\
      &\leq \frac{\theta^{2}}{\binom{n+\theta-1}{\theta-1}^{2}\lfloor \delta n \rfloor^{2}}\left(\sum_{k=0}^{n}\binom{k+\theta-1}{\theta-1}^{2}\right)^{2}.
    \end{split}
  \end{equation}
  To obtain the first inequality above, we used that removing the $q-1$ constraint only \textit{increases} the corresponding fourth moment, this follows from definition \eqref{cnq} and \eqref{eq:2ndmom}. Then the middle equality uses the magic square formula \eqref{4thmom} and re-orders the sum. Finally, the last bound is of order if $n^{-2\theta}$ if $\theta < 1/2$, of order $\log^{2}(n)/n$ if $\theta=1/2$ and of order $n^{2(\theta-1)}$ if $\theta \in (1/2,1)$. In each case \eqref{lindebergcor} follows.
\end{proof}

Verifying \eqref{condvarcor} and identifying the limit $\nu$ turns out to be more difficult. It will turn out that $\mathcal{M}_{\theta,\delta,n}$ has a similar behaviour to the quantity $\mathcal{M}_{\theta,n}$ defined in \eqref{mthetan}, pre-multiplied by a certain explicit constant $C_{\delta}$. To see how this constant arises, we begin with the following warm up exercise.
\begin{lemma}
\label{lem:cdef}
  Consider the quantity $\mathcal{M}_{\theta,\delta,n}$ defined in \eqref{condvar} and set
  \begin{equation}
    C_{\delta} := \theta\int_{\delta}^{1}(1-x)^{\theta-1}\mathbb{P}\left(L^{(\infty)} \leq \frac{x}{1-x}\right)\,\frac{dx}{x} \label{cdeldef}
  \end{equation}
  where $L^{(\infty)}$ is the limiting random variable from Lemma \ref{lem:kingman}. Then we have
  \begin{equation}
    \lim_{n \to \infty}\mathbb{E}(\mathcal{M}_{\theta,\delta,n}) = C_{\delta}.
  \end{equation}
  Furthermore, the constant $C_{\delta}$ can be explicitly computed as
  \begin{equation}
    C_{\delta} = 1-\Gamma(\theta)e^{\gamma_{\mathrm{E}}\theta}\,\delta^{\theta-1}p_{\theta}(1/\delta), \label{cdelident}
  \end{equation}
  and satisfies the bound
  \begin{equation}
    C_{\delta} = 1+O(\delta), \qquad \delta \to 0. \label{cto1}
  \end{equation}
\end{lemma}

\begin{proof}
  By the correspondence \eqref{corres}, we have the identity
  \begin{equation} \label{eq:expcnq}
    \begin{split}
      \mathbb{E}(|c_{n-q,q-1}|^{2}) &= \sum_{\substack{(m_{k}) : 1 \leq k \leq q-1\\ \sum_{k=1}^{q-1}km_{k}=n-q}}\prod_{k=1}^{q-1}\frac{\theta^{m_{k}}}{m_{k}!k^{m_k}}\\
      &= \binom{n-q+\theta-1}{\theta-1}\mathbb{P}(L^{(n-q)} \leq q-1)
    \end{split}
  \end{equation}
  where $L^{(n-q)}$ is the longest cycle in a Ewens distributed random permutation of length $n-q$. Furthermore, in the regime of interest, $q = \lfloor \delta n \rfloor, \ldots, n$ so that $q$ and $n-q$ are proportional to $n$. In this regime Lemma \ref{lem:kingman} applies: if $q/n \to x \in (\delta,1)$ then
  \begin{equation}
    \lim_{n \to \infty}\mathbb{P}(L^{(n-q)} \leq q-1) = \mathbb{P}\left(L^{(\infty)} \leq \frac{x}{1-x}\right)
  \end{equation}
  and consequently the expectation of $\mathcal{M}_{\theta,\delta,n}$ converges to a Riemann integral,
  \begin{equation}
  \begin{split}
    \lim_{n \to \infty}\mathbb{E}(\mathcal{M}_{\theta,\delta,n}) &= \lim_{n \to \infty}\frac{\theta}{n}\sum_{q=\lfloor \delta n \rfloor}^{n}\frac{\mathbb{E}(|c_{n-q,q-1}|^{2})}{\binom{n+\theta-1}{\theta-1}\,\frac{q}{n}}\\
    &=\lim_{n \to \infty}\frac{\theta}{n}\sum_{q=\lfloor \delta n \rfloor}^{n}\frac{\binom{n-q+\theta-1}{\theta-1}}{\binom{n+\theta-1}{\theta-1}}\frac{\mathbb{P}(L^{(n-q)} \leq q-1)}{\frac{q}{n}} \label{riemsum}\\
    &=\theta\int_{\delta}^{1}(1-x)^{\theta-1}\mathbb{P}\left(L^{(\infty)} \leq \frac{x}{1-x}\right)\,\frac{dx}{x}\\
    &= C_{\delta}.
  \end{split}
  \end{equation}
  To compute $C_{\delta}$, we express the probability in the integrand of \eqref{cdeldef} in terms of the function $p_{\theta}(y)$ and use the properties in Lemma \ref{lem:t0n}. We have
  \begin{equation}
    \begin{split}
      C_{\delta}&=\theta\Gamma(\theta)e^{\gamma_{\mathrm{E}} \theta}\int_{\delta}^{1}x^{\theta-2}p_{\theta}(1/x-1)\,dx\\
      &= \Gamma(\theta)e^{\gamma_{\mathrm{E}}  \theta}\int_{1}^{1/\delta}\theta\,x^{-\theta}p_{\theta}(x-1)\,dx\\
      &= -\Gamma(\theta)e^{\gamma_{\mathrm{E}}  \theta}\int_{1}^{1/\delta}\frac{d}{dx}(x^{1-\theta}p_{\theta}(x))\,dx\\
      &= \Gamma(\theta)e^{\gamma_{\mathrm{E}} \theta}p_{\theta}(1) - \Gamma(\theta)e^{\gamma_{\mathrm{E}}  \theta}\delta^{\theta-1}p_{\theta}(1/\delta).
    \end{split}
  \end{equation}
  The identity \eqref{cdelident} now follows from the explicit formula $p_{\theta}(1) = e^{-\gamma_{\mathrm{E}}\theta}/\Gamma(\theta)$ and the limiting behaviour $C_{\delta}=1+O(\delta)$ follows from the rapid decay in \eqref{rapid}.
\end{proof}

We will now describe how these ideas lead to the proof of Theorem \ref{th:l2phase}. In the following sections we will generalise the argument of Lemma \ref{lem:cdef} to show that for any $0 < \theta < \frac{1}{2}$ we have
\begin{equation}
  \lim_{n \to \infty}\mathbb{E}(|\mathcal{M}_{\theta,\delta,n}-C_{\delta}\mathcal{M}_{\theta,n}|^{2}) = 0, \label{2ndmomexp}
\end{equation}
allowing us to replace the convergence in probability of $\mathcal{M}_{\theta,\delta,n}$ with that of $C_{\delta}\mathcal{M}_{\theta,n}$ in \eqref{mthetan}. For the latter, the appropriate limit is given in Lemma \ref{lem:condvarapprox} and, combined with \eqref{2ndmomexp}, this implies that 
\begin{equation}
  \mathcal{M}_{\theta,\delta,n} \overset{p}{\longrightarrow} C_{\delta}\mathcal{M}_{\theta}, \qquad n \to \infty,
\end{equation}
which is the identification of $\nu$. Now the martingale central limit theorem of Corollary \ref{cor:realtocomplex} gives a limiting distribution for $\tilde{c}_{n}^{(\delta)}$, 
\begin{equation}
  \frac{\tilde{c}_{n}^{(\delta)}}{\sqrt{\mathbb{E}(|c_{n}|^{2})}} \overset{d}{\longrightarrow} \sqrt{C_{\delta}\,\mathcal{M}_{\theta}}\,\mathcal{N}_{1} \qquad n \to \infty. \label{tildecnlimit}
\end{equation}
The proof of Theorem \ref{th:l2phase} contingent on \eqref{2ndmomexp} now follows from a standard approximation argument.
\begin{proof}[Proof of Theorem \ref{th:l2phase} assuming \eqref{2ndmomexp}]
  Let us define the appropriately normalized version of $c_{n}$ as 
  \begin{equation}
    \hat{c}_{n} := \frac{c_{n}}{\sqrt{\mathbb{E}(|c_{n}|^{2})}}.
  \end{equation}
  The joint characteristic functions of the real and imaginary parts of $\hat{c}_{n}$ and the proposed limiting random variable are:
  \begin{equation}
    \begin{split}
      \Phi_{n}(s,t) &:= \mathbb{E}\left(e^{is\mathrm{Re}(\hat{c}_{n})+it\mathrm{Im}(\hat{c}_{n})}\right),\\
      \Phi(s,t) &:= \mathbb{E}\left(e^{is\mathrm{Re}(\sqrt{\mathcal{M}_{\theta}}\mathcal{N}_{1})+it\mathrm{Im}(\sqrt{\mathcal{M}_{\theta}}\mathcal{N}_{1})}\right).
    \end{split}
  \end{equation}
  For our martingale approximation $\tilde{c}_{n}^{(\delta)}$, similarly, we define
  \begin{equation}
    \hat{\tilde{c}}_{n}^{(\delta)} := \frac{\tilde{c}_{n}^{(\delta)}}{\sqrt{\mathbb{E}(|c_{n}|^{2})}}
  \end{equation}
  and the corresponding characteristic functions
  \begin{equation}
    \begin{split}
      \tilde{\Phi}_{n,\delta}(s,t) &:= \mathbb{E}\left(e^{is\mathrm{Re}(\hat{\tilde{c}}_{n}^{(\delta)})+it\mathrm{Im}(\hat{\tilde{c}}_{n}^{(\delta)})}\right),\\
      \tilde{\Phi}_{\delta}(s,t) &:= \mathbb{E}\left(e^{is\mathrm{Re}(\sqrt{C_{\delta}\mathcal{M}_{\theta}}\mathcal{N}_{1})+it\mathrm{Im}(\sqrt{C_{\delta}\mathcal{M}_{\theta}}\mathcal{N}_{1})}\right).
    \end{split}
  \end{equation}
  Then we have
  \begin{equation}
    \begin{split}
      |\Phi_{n}(s,t)-\Phi(s,t)| &\leq |\Phi_{n}(s,t)-\tilde{\Phi}_{n,\delta}(s,t)|+|\tilde{\Phi}_{n,\delta}(s,t)-\tilde{\Phi}_{\delta}(s,t)|\\
      &+|\tilde{\Phi}_{\delta}(s,t)-\Phi(s,t)|.
    \end{split}
  \end{equation}
  By the convergence in distribution \eqref{tildecnlimit} the term $|\tilde{\Phi}_{n,\delta}(s,t)-\tilde{\Phi}_{\delta}(s,t)| \to 0$ as $n \to \infty$, while by the asymptotics \eqref{cto1} we have  $|\tilde{\Phi}_{\delta}(s,t)-\Phi(s,t)| \to 0$ as $\delta \to 0$. A standard bound on the exponential function (see Lemma \ref{le:l2bound}) gives the inequality
  \begin{equation}
    |\Phi_{n}(s,t)-\tilde{\Phi}_{n,\delta}(s,t)| \leq (|t|+|s|)\sqrt{\mathbb{E}(|\hat{c}_{n}-\hat{\tilde{c}}_{n}^{(\delta)}|^{2})}.\label{2ndmombound}
  \end{equation}
  By Lemma \ref{le:martingalerep} the expression on the right-hand side of \eqref{2ndmombound} tends to zero in the limit $n \to \infty$ followed by $\delta \to 0$. This implies that $|\Phi_{n}(s,t)-\Phi(s,t)| \to 0$ as $n \to \infty$ and the statement of Theorem \ref{th:l2phase} follows.
\end{proof}

\section{Convergence of the bracket process and the proof of $L^{2}$-convergence}
\label{sec:mgle3}
The goal of this section will be to verify \eqref{2ndmomexp}. To do this we will need to calculate some higher moments of secular coefficients with a cycle constraint, namely of the $c_{n,q}$ defined in \eqref{cnq}. Then we use these results to analyze the second moment in \eqref{2ndmomexp}.
\subsection{Moments of secular coefficients with a cycle constraint}
The following Lemma generalises the second moment formula \eqref{eq:expcnq} and the magic square formula \eqref{genbetmoms}.
\begin{lemma}
Let $m$ be a positive integer and assume without loss of generality that the positive integers $q_{1}\leq \ldots \leq q_{n}$ are ordered. Recall that $L^{(n)}_{1}$ is the longest cycle of a random permutation of length $n$ under the Ewens measure \eqref{ewens}. Then we have
\begin{equation}
\mathbb{E}\left(\prod_{i=1}^{m}|c_{n_i,q_i}|^{2}\right) = \sum_{A \in \mathrm{Mag}_{\vec{n},\vec{n}}}\prod_{l_1,l_2=1}^{m}\binom{A_{l_1,l_2}+\theta-1}{\theta-1}\mathbb{P}(L^{(A_{l_1,l_2})}_{1} \leq q_{l_1 \land l_2}), \label{finalresultcorrelator}
\end{equation}
where we used the notation $l_{1}\land l_2 := \mathrm{min}(l_1,l_2)$.
\end{lemma}
\begin{proof}
The strategy of the proof is very similar to that of Theorem \ref{thm:magic}. By definition \eqref{cnq}, the coefficients $c_{n,q}$ can be extracted from a generating function
\begin{equation}\label{eq:cnqgen}
  c_{n,q} = [z^{n}]\,\mathrm{exp}\left(\sqrt{\theta}\,\sum_{k=1}^{q}\frac{\mathcal{N}_{k}}{\sqrt{k}}\,z^{k}\right),
\end{equation}
so that the left-hand side of \eqref{finalresultcorrelator} is $[z_{1}^{n_1}\ldots z_{m}^{n_m}\overline{w_{1}}^{n_1}\ldots\overline{w_m}^{n_m}]\,\mathbb{E}(F^{(m)}_{\vec{q}}(\vec{z},\vec{\overline{w}}))$ where
\begin{equation}
  F^{(m)}_{\vec{q}}(\vec{z},\vec{\overline{w}}) := \mathrm{exp}\left(\sqrt{\theta}\sum_{i=1}^{m}\left(\sum_{k=1}^{q_i}\frac{\mathcal{N}_{k}}{\sqrt{k}}\,z_{i}^{k}+\frac{\overline{\mathcal{N}_{k}}}{\sqrt{k}}\,\overline{w_i}^{k}\right)\right).
\end{equation}
Using that $q_{1} \leq \ldots \leq q_{m}$ we can re-arrange the summation (interchange the $i$ and $k$ indices) using the identity
\begin{equation}
  \sum_{i=1}^{m}\sum_{k=1}^{q_i}a_{k,i} = \sum_{i=1}^{m}\sum_{k=q_{i-1}+1}^{q_i}\sum_{r=i}^{m}a_{k,r} \label{sumid1}
\end{equation}
valid for arbitrary $a_{k,i}$ where we set $q_{0}:=0$. Then using independence we get
\begin{equation}
  \begin{split}
    \mathbb{E}\left(F^{(m)}_{\vec{q}}(\vec{z},\vec{\overline{w}})\right) &= \prod_{i=1}^{m}\mathbb{E}\left(\mathrm{exp}\left(\sqrt{\theta}\,\sum_{k=q_{i-1}+1}^{q_i}\sum_{r=i}^{m}\left(\frac{\mathcal{N}_{k}}{\sqrt{k}}\,z_{r}^{k}+\frac{\overline{\mathcal{N}_{k}}}{\sqrt{k}}\,\overline{w_r}^{k}\right)\right)\right)\\
    &= \mathrm{exp}\left(\theta\,\sum_{i=1}^{m}\sum_{k=q_{i-1}+1}^{q_{i}}\sum_{l_1,l_2=i}^{m}\frac{(z_{l_1}\overline{w_{l_2}})^{k}}{k}\right)\\
    &= \prod_{l_1,l_2=1}^{m}\mathrm{exp}\left(\theta\, \sum_{k=1}^{q_{l_{1} \land l_{2}}}\frac{(z_{l_1}\overline{w_{l_2}})^{k}}{k}\right), \label{truncatedgen}
  \end{split}
\end{equation}
where to obtain the last line, we again re-arranged the sums using the identity
\begin{equation}
  \sum_{i=1}^{m}\sum_{k=q_{i-1}+1}^{q_{i}}\sum_{r_1,r_2=i}^{m}a_{k,r_1,r_2} = \sum_{r_1,r_2=1}^{m}\sum_{k=1}^{q_{r_{1} \land r_{2}}}a_{k,r_1,r_2}, \label{sumid2}
\end{equation}
valid for arbitrary $a_{k,r_1,r_2}$. Now to pick out coefficients, we expand each exponential in \eqref{truncatedgen} using 
\begin{equation}
  \begin{split}
    &\mathrm{exp}
    \left(\theta\, \sum_{k=1}^{q_{l_1 \land l_2}}\frac{(z_{l_1}\overline{w_{l_2}})^{k}}{k}\right)\\
    &=
    \sum_{A_{l_1,l_2}=0}^{\infty}
    (z_{l_1}\overline{w_{l_2}})^{A_{l_1,l_2}}
    \sum_{\substack{(m_{k}) : 1 \leq k \leq q_{l_1 \land l_2}\\ \sum_{k=1}^{q_{l_1 \land l_2}}km_{k}=A_{l_1,l_2}}}
    \prod_{k=1}^{q_{l_1 \land l_2}}\left(\frac{\theta}{k}\right)^{m_{k}}\,\frac{1}{m_{k}!}
    \\
    &=\sum_{A_{l_1,l_2}=0}^{\infty}(z_{l_1}\overline{w_{l_2}})^{A_{l_1,l_2}}
    \binom{A_{l_1,l_2}+\theta-1}{\theta-1}
    \mathbb{P}(L^{(A_{l_1,l_2})}_{1} \leq q_{l_1 \land l_2}).
  \end{split}
\end{equation}
Inserting this into \eqref{truncatedgen} leads to the representation \eqref{finalresultcorrelator}.
\end{proof}
For the proof of \eqref{2ndmomexp}, we will need the case $m=2$ corresponding to fourth moments of the secular coefficients, and involving sums over $2 \times 2$ magic squares (in analogy with Example \ref{ex:4thmom}). Let us assume that $n_{1} \leq n_{2}$ and $q_{1} \leq q_{2}$. Then \eqref{finalresultcorrelator} implies that
\begin{equation}
  \begin{split}
    &\mathbb{E}\left(|c_{n_1,q_1}|^{2}|c_{n_2,q_2}|^{2}\right) = \sum_{k=0}^{n_{1}}\binom{k+\theta-1}{\theta-1}^{2}\mathbb{P}(L^{(k)}_{1} \leq q_{1})^{2}\\
    &\times\binom{n_{1}-k+\theta-1}{\theta-1}\mathbb{P}(L^{(n_1-k)}_{1} \leq q_{1})\binom{n_{2}-k+\theta-1}{\theta-1}\mathbb{P}(L^{(n_2-k)}_{1} \leq q_{2}). \label{4thmomconstr}
  \end{split}
\end{equation}

\subsection{The $L^{2}$-phase and proof of  \eqref{2ndmomexp}}
\label{sec:mgle4}
We recall the definition of the two quantities:
\begin{align}
  \mathcal{M}_{\theta,n} &:= \frac{\Gamma(\theta+1)}{n^{\theta}}\sum_{q=0}^{n}|c_{q}|^{2} \label{Mn}\\
  \mathcal{M}_{\theta,\delta,n} &:= \frac{1}{\binom{n+\theta-1}{\theta-1}}\sum_{q=\lfloor \delta n \rfloor}^{n}|c_{n-q,q-1}|^{2}\,\frac{\theta}{q} \label{Mtilden}
\end{align}
and the deterministic constant 
\begin{equation}
  C_{\delta} := \theta\int_{\delta}^{1}\frac{(1-x)^{\theta-1}\mathbb{P}(L_{1}^{(\infty)} \leq x/(1-x))}{x}\,dx.
\end{equation}

Our goal is to prove \eqref{2ndmomexp} for any $0 < \theta < 1/2$. The condition $\theta < 1/2$ is a technical requirement in the proof to ensure the convergence of the infinite sum
\begin{equation}
  \sum_{k=0}^{\infty}\binom{k+\theta-1}{\theta-1}^{2} = \mathbb{E}(\mathcal{M}_{\theta}^{2}) = \frac{\Gamma(1-2\theta)}{\Gamma(1-\theta)^{2}},
\end{equation}
or equivalently existence of the second moment of the total mass. For a proof of this identity, see Lemma \ref{infsumk}.

\begin{theorem}[$L^{2}$-convergence]
  \label{th:thml2convergence}
  Fix $0 < \theta < \frac{1}{2}$ and $\delta>0$. Then the following limit holds
  \begin{equation}
    \lim_{n \to \infty}\mathbb{E}\left(|\mathcal{M}_{\theta,\delta,n}-C_{\delta}\mathcal{M}_{\theta,n}|^{2}\right) = 0. \label{l2convergence}
  \end{equation}
\end{theorem}
\begin{proof}
  We start by expanding the square in \eqref{l2convergence} and endeavour to compute the three terms $\mathcal{S}_{1,n} := \mathbb{E}(\mathcal{M}_{\theta,\delta,n}^{2})$, $\mathcal{S}_{2,n}:=\mathbb{E}(\mathcal{M}_{\theta,n}^{2})$, and $\mathcal{S}_{3,n} := \mathbb{E}(\mathcal{M}_{\theta,\delta,n}\mathcal{M}_{\theta,n})$. It is clearly sufficient to show that the limit of each of these quantities exists is given by the second moment of the total mass, up to the appropriate factor of $C_{\delta}$. It will be convenient to abbreviate
  \begin{equation}
    P_{n,k} := \mathbb{P}(L_{1}^{(n)} \leq k-1)
  \end{equation}
  as the probability that the longest cycle in a Ewens distributed random permutation of size $n$ is less than or equal to $k-1$.

  Starting with the second moment of $\mathcal{M}_{\theta,\delta,n}$ in \eqref{Mtilden}, we have
  \begin{equation}
\mathcal{S}_{1,n} = \frac{\theta^{2}}{\binom{n+\theta-1}{\theta-1}^{2}}\sum_{q_{1}=\lfloor \delta n \rfloor}^{n}\sum_{q_2=\lfloor \delta n \rfloor}^{n}\frac{\mathbb{E}\left(|c_{n-q_1,q_1-1}|^{2}|c_{n-q_2,q_2-1}|^{2}\right)}{q_1 q_2}.
  \end{equation}
Because of the obvious symmetry in $q_{1}$ and $q_{2}$, it will be convenient to consider the contribution of such sums on the region $q_{1} \leq q_{2}$, which we denote $\mathcal{S}_{1,n}(q_{1}\leq q_{2})$ and similarly for the other sums. Applying \eqref{4thmomconstr} and interchanging the order of summation we obtain
  \begin{equation}
    \begin{split}
      &\mathcal{S}_{1,n}(q_{1}\leq q_{2}) = \sum_{q_{1}=\lfloor \delta n \rfloor}^{n}\sum_{q_2=q_1}^{n}\sum_{k=0}^{n-q_{2}}\binom{k+\theta-1}{\theta-1}^{2}P_{k,q_1}^{2}\,\frac{\theta^{2}}{q_{1}q_{2}}\,\frac{1}{\binom{n+\theta-1}{\theta-1}^{2}}\, \times\\
      &\binom{n-q_1-k+\theta-1}{\theta-1}P_{n-q_1-k,q_1}\binom{n-q_2-k+\theta-1}{\theta-1}P_{n-q_2-k,q_2}.\label{mtildeeq}
    \end{split}
  \end{equation}
  
To estimate the sum \eqref{mtildeeq} as $n \to \infty$, the idea is to treat $k$ as fixed while $q_1$ and $q_2$ are large, applying dominated convergence to bring the limit $n \to \infty$ inside the sum over $k$.

As in \eqref{riemsum}, the sums over $q_{1}$ and $q_{2}$ will always be treated as Riemann sum approximations. For $\mathcal{S}_{1,n}(q_{1}\leq q_{2})$, the double sum over $q_1$ and $q_2$ is symmetric (we can assume that $P_{k,q_1}^{2}=1$ for fixed $k$ and $n$ large enough, since $q_{1} \geq \lfloor \delta n \rfloor$ which is larger than $k$) and we can write it as $\frac{1}{2}$ times the square of the same sum appearing in \eqref{riemsum} (neglecting for now the diagonal $q_{1}=q_{2}$, see below). This immediately yields the limit as $\frac{1}{2}C_{\delta}^{2}$. By symmetry the same reasoning and limiting value applies to the sum with $q_{1} \geq q_{2}$. Assuming we can apply dominated convergence to the sum over $k$, we get
  \begin{equation}
    \lim_{n \to \infty}\mathcal{S}_{1,n} = C_{\delta}^{2}\sum_{k=0}^{\infty}\binom{k+\theta-1}{\theta-1}^{2} = C_{\delta}^{2}\mathbb{E}(\mathcal{M}_{\theta}^{2}) \label{final-t1}
  \end{equation}
  as required. 

Now we will justify the dominated convergence by finding a uniform and summable bound in \eqref{mtildeeq}. Notice that the sums over $q_1$ and $q_2$ are uniformly bounded by
  \begin{align*}
    &\frac{\theta^{2}}{\lfloor \delta n \rfloor^{2}\binom{n+\theta-1}{\theta-1}^{2}}\sum_{q_1=\lfloor \delta n \rfloor}^{n-k}\sum_{q_2=q_1}^{n-k}\binom{n-q_1-k+\theta-1}{\theta-1}\binom{n-q_2-k+\theta-1}{\theta-1}\\
    &=\frac{\theta^{2}}{\lfloor \delta n \rfloor^{2}\binom{n+\theta-1}{\theta-1}^{2}}\sum_{q_{1}=0}^{n-\lfloor \delta n \rfloor - k}\binom{q_1+\theta-1}{\theta-1}\binom{q_1+\theta}{\theta}
 \end{align*}
where we used the binomial sum identity \eqref{sumlemmaeq}. The final expression can be bounded uniformly by extending the summation to $q_{1}=n$ and noting that the summand is of order $q_{1}^{2\theta-1}$ as $q_{1} \to \infty$. If we took just the terms where $q_{2}=q_{1}$, we would have the bound
  \begin{equation}
    \begin{split}
      &\frac{\theta^{2}}{\binom{n+\theta-1}{\theta-1}^{2}}\frac{1}{\lfloor \delta n \rfloor^{2}}\sum_{q_1=\lfloor \delta n \rfloor}^{n-k}\binom{n-q_1-k+\theta-1}{\theta-1}^{2}\\
      &\leq \frac{\theta^{2}}{\binom{n+\theta-1}{\theta-1}^{2}}\frac{1}{\lfloor \delta n \rfloor^{2}}\sum_{q_1=0}^{n-\lfloor \delta n \rfloor}\binom{q_1+\theta-1}{\theta-1}^{2} \to 0, \qquad n \to \infty.
    \end{split}
  \end{equation}
This justifies overcounting terms where $q_{2}=q_{1}$ and completes the proof of \eqref{final-t1}.

  Now we consider the second moment $\mathbb{E}(\mathcal{M}^{2}_{\theta,n})$. Again using \eqref{4thmomconstr} in the unconstrained case (or \eqref{genbetmoms}) we obtain
    \begin{equation}
    \begin{split}
      \mathcal{S}_{2,n} &= \frac{\Gamma(\theta+1)^{2}}{n^{2\theta}}\sum_{q_1=0}^{n}\sum_{q_2=0}^{n}\sum_{k=0}^{\mathrm{min}(q_1,q_2)}\binom{k+\theta-1}{\theta-1}^{2}\\
      &\times\binom{q_1-k+\theta-1}{\theta-1}\binom{q_2-k+\theta-1}{\theta-1}. \label{2ndmom}
    \end{split}
  \end{equation}
  An application of Lemma \ref{sumlemma} shows that the contribution to \eqref{2ndmom} from the region $q_{1} \leq q_{2}$ is equal to
  \begin{align}
    & \mathcal{S}_{2,n}(q_{1} \leq q_{2}) = \frac{\Gamma(\theta+1)^{2}}{n^{2\theta}}\sum_{q_1=0}^{n}\sum_{k=0}^{q_1}\binom{k+\theta-1}{\theta-1}^{2}\\
    &\times\binom{q_1-k+\theta-1}{\theta-1}\left(\binom{n-k+\theta}{\theta}-\binom{q_1-k+\theta-1}{\theta}\right)
  \end{align}
  Interchanging the summation and summing over $q_1$ again using Lemma \ref{sumlemma} we obtain
  \begin{align}
    &\mathcal{S}_{2,n}(q_{1} \leq q_{2}) =  \frac{\Gamma(\theta+1)^{2}}{n^{2\theta}}\sum_{k=0}^{n}\binom{k+\theta-1}{\theta-1}^{2}\binom{n-k+\theta}{\theta}^{2} \label{firsttermdom}\\
    &-\frac{\Gamma(\theta+1)^{2}}{n^{2\theta}}\sum_{k=0}^{n}\binom{k+\theta-1}{\theta-1}^{2}\sum_{q_1=0}^{n-k}\binom{q_1+\theta-1}{\theta-1}\binom{q_1+\theta-1}{\theta}. \label{twoterms}
  \end{align}
  By dominated convergence, the first term \eqref{firsttermdom} converges to $\mathbb{E}(\mathcal{M}_{\theta}^{2})$, where a suitable bound is obtained by noticing that the second binomial coefficient is decreasing in $k$. For the second term \eqref{twoterms}, a uniform upper bound is obtained by summing up to $q_{1}=n$ and noting that the summand is of order $q_{1}^{2\theta-1}$ with a corresponding sum of order $n^{2\theta}$ as $n \to \infty$. Then dominated convergence implies that \eqref{twoterms} has a limit given by
  \begin{equation}
    \begin{split}
      &\sum_{k=0}^{\infty}\binom{k+\theta-1}{\theta-1}^{2}\lim_{n \to \infty}\frac{\Gamma(\theta+1)^{2}}{n^{2\theta}}\sum_{q_{1}=0}^{n-k}\binom{q_1+\theta-1}{\theta-1}\binom{q_1+\theta-1}{\theta}\\
      &= \sum_{k=0}^{\infty}\binom{k+\theta-1}{\theta-1}^{2}\lim_{n \to \infty}\frac{\theta}{n^{2\theta}}\sum_{q_{1}=1}^{n}q_{1}^{2\theta-1}\\
      &=\frac{1}{2}\sum_{k=0}^{\infty}\binom{k+\theta-1}{\theta-1}^{2} = \frac{1}{2}\mathbb{E}(\mathcal{M}_{\theta}^{2})
    \end{split}
  \end{equation}
  An analogous computation shows that the contribution to \eqref{2ndmom} from the diagonal $q_{1}=q_{2}$ is order $O(n^{-1})$. By symmetry in $q_1$ and $q_2$, we obtain
  \begin{equation}
    \lim_{n \to \infty}\mathcal{S}_{2,n} = \sum_{k=0}^{\infty}\binom{k+\theta-1}{\theta-1}^{2} = \mathbb{E}(\mathcal{M}_{\theta}^{2}).
  \end{equation}
It remains to consider the mixed term $\mathbb{E}(\mathcal{M}_{\theta,\delta,n}\mathcal{M}_{\theta,n})$ denoted $\mathcal{S}_{3,n}$. By definition we have
  \begin{equation}
    \begin{split}
      &\mathcal{S}_{3,n} = \mathbb{E}(\mathcal{M}_{\theta,\delta,n}\mathcal{M}_{\theta,n}) \\
      &=\frac{\Gamma(\theta+1)}{n^{\theta}}\,\frac{1}{\binom{n+\theta-1}{\theta-1}}\,\sum_{q_{1}=0}^{n}\sum_{q_{2}=\lfloor \delta n \rfloor}^{n}\mathbb{E}\left(|c_{q_1}|^{2}|c_{n-q_2,q_2-1}|^{2}\right)\,\frac{\theta}{q_2}.
    \end{split}
  \end{equation}
  We split this up as two sums, one for values of $q_{1}=0,\ldots,n-q_2$, and one for values of $q_{1}=n-q_{2}+1,\ldots ,n$. Interchanging the order of summation, the first sum gives
  \begin{equation}
    \begin{split}
      &\mathcal{S}_{3,n}(q_{1} \leq n-q_{2}) =   
      \frac{\Gamma(\theta+1)}{n^{\theta}\binom{n+\theta-1}{\theta-1}}\,\sum_{k=0}^{n-\lfloor \delta n \rfloor}\sum_{q_2=\lfloor \delta n \rfloor}^{n-k}\sum_{q_{1}=k}^{n-q_{2}}\binom{k+\theta-1}{\theta-1}^{2}\\
      &\times P_{k,q_{2}}^{2}\binom{q_1-k+\theta-1}{\theta-1}\binom{n-q_{2}-k+\theta-1}{\theta-1}P_{n-q_{2}-k,q_{2}}\,\frac{\theta}{q_{2}}. \label{s3nsum1}
    \end{split}
  \end{equation}
The sum over $q_{1}$ in \eqref{s3nsum1} is handled with Lemma \ref{sumlemma} and we have
  \begin{equation}
    \begin{split}
      &\mathcal{S}_{3,n}(q_{1} \leq n-q_{2}) = \frac{\Gamma(\theta+1)}{n^{\theta}\binom{n+\theta-1}{\theta-1}}\,\sum_{k=0}^{n-\lfloor \delta n \rfloor}\sum_{q_2=\lfloor \delta n \rfloor}^{n-k}\binom{k+\theta-1}{\theta-1}^{2}\\
      &\times P_{k,q_{2}}^{2}\binom{n-q_{2}-k+\theta}{\theta}\binom{n-q_{2}-k+\theta-1}{\theta-1}P_{n-q_{2}-k,q_{2}}\,\frac{\theta}{q_{2}}. \label{sumq1exp}
    \end{split}
  \end{equation}
  Now we look at the contribution to $\mathcal{S}_{3,n}$ indexed by $q_{1}=n-q_{2}+1,\ldots,n$. This gives
  \begin{equation}
    \begin{split}
      &\mathcal{S}_{3,n}(q_{1} > n-q_{2}) = \frac{\Gamma(\theta+1)}{n^{\theta}\binom{n+\theta-1}{\theta-1}}\,\sum_{q_{2}=\lfloor \delta n \rfloor}^{n}\sum_{q_{1}=n-q_{2}+1}^{n}\sum_{k=0}^{n-q_{2}}\binom{k+\theta-1}{\theta-1}^{2}\\
      &\times P_{k,q_{2}}^{2}\binom{q_1-k+\theta-1}{\theta-1}\binom{n-q_{2}-k+\theta-1}{\theta-1}P_{n-q_{2}-k,q_{2}}\,\frac{\theta}{q_{2}}.
    \end{split}
  \end{equation}
  Summing over $q_{1}$ with Lemma \ref{sumlemma} and interchanging the order of summation gives the identity 
  \begin{equation}
    \begin{split}
      &\mathcal{S}_{3,n}(q_{1} > n-q_{2})  = \frac{\Gamma(\theta+1)}{n^{\theta}\binom{n+\theta-1}{\theta-1}}\,\sum_{k=0}^{n-\lfloor \delta n \rfloor}\sum_{q_{2}=\lfloor \delta n \rfloor}^{n-k}\binom{k+\theta-1}{\theta-1}^{2}P_{k,q_{2}}^{2}\,\\
      &\times\left(\binom{n-k+\theta}{\theta}-\binom{n-q_{2}-k+\theta}{\theta}\right)\\
      &\times\binom{n-q_{2}-k+\theta-1}{\theta-1}\frac{\theta}{q_{2}}P_{n-q_{2}-k,q_{2}}.
    \end{split}
  \end{equation}
  Combining with \eqref{sumq1exp} yields a cancellation and we are left with the identity
  \begin{equation}
    \begin{split}
      \mathcal{S}_{3,n} &= \sum_{k=0}^{n-\lfloor \delta n \rfloor}\sum_{q_{2}=\lfloor \delta n \rfloor}^{n-k}\binom{k+\theta-1}{\theta-1}^{2}P_{k,q_{2}}^{2}\\
      &\times\Gamma(\theta+1)\frac{\binom{n-k+\theta}{\theta}}{n^{\theta}}\frac{\binom{n-q_{2}-k+\theta-1}{\theta-1}}{\binom{n+\theta-1}{\theta-1}}P_{n-q_{2}-k,q_{2}}\,\frac{\theta}{q_{2}}. \label{covidentitymtilde}
    \end{split}
  \end{equation}
  As in the treatment of \eqref{mtildeeq} (see also \eqref{riemsum}), for $n \to \infty$ and fixed $k$, the $k^{\mathrm{th}}$ term of \eqref{covidentitymtilde} is approximated by a Riemann integral and $P_{k,q_{2}}=1$. Applying dominated convergence then gives the limit as $n \to \infty$ of \eqref{covidentitymtilde} as $C_{\delta}\mathbb{E}(\mathcal{M}_{\theta}^{2})$ as required. For a suitable dominating function, note that apart from the factor $\binom{k+\theta-1}{\theta-1}^{2}$, the $k^{\mathrm{th}}$ term in \eqref{covidentitymtilde} is uniformly bounded by 
  \begin{align}
    &\frac{\theta}{\lfloor \delta n \rfloor}\Gamma(\theta+1)\frac{\binom{n-k+\theta}{\theta}}{n^{\theta}}\sum_{q_{2}=\lfloor \delta n \rfloor}^{n-k}\frac{\binom{n-q_{2}+\theta-1}{\theta-1}}{\binom{n+\theta-1}{\theta-1}}\\
    &\leq \frac{\theta}{\lfloor \delta n \rfloor}\Gamma(\theta+1)\frac{\binom{n+\theta}{\theta}}{n^{\theta}}\frac{\binom{n-\lfloor \delta n \rfloor+\theta}{\theta}}{\binom{n+\theta-1}{\theta-1}},
  \end{align}
  where we used Lemma \ref{sumlemma} and monotonicity of the binomial coefficients as a function of $k$. The final bound is independent of $k$ and bounded in $n$. This shows that
  \begin{equation}
    \lim_{n \to \infty}\mathcal{S}_{3,n} = C_{\delta}\sum_{k=0}^{\infty}\binom{k+\theta-1}{\theta-1}^{2} = C_{\delta}\mathbb{E}(\mathcal{M}_{\theta}^{2}),
  \end{equation}
 as required. Putting all three limits for the moments together completes the proof of the theorem.
\end{proof}

\section{Regularity of the holomorphic multiplicative chaos}
\label{sec:regular}

We have to determine, for $s \in \mathbb{R}$, if the series 
\begin{equation}
A_{s,\theta} \coloneqq \sum_{n=0}^{\infty}(1+n^{2})^{s}|c_{n}|^{2} \label{Aseries}
\end{equation}
is convergent. We will study the cases $\theta \leq 1$ and $\theta > 1$ separately. 
\subsection{Sub-critical and critical cases, $\theta \in (0,1]$}
We have that $\mathbb{E} [ |c_n|^2] \leq C_{\theta} (1+n)^{\theta-1}$ for some constant $C_{\theta} > 0$, and then 
$$\mathbb{E} [ A_{s, \theta} ] \leq C_{\theta} \sum_{n=0}^{\infty}(1+n)^{2s + \theta-1},$$
which is finite as soon as $s < -\theta/2$. Hence, HMC$_{\theta}$ is a.s.\ in $H^s$ for all $s < - \theta/2$. Notice that this reasoning remains true in the supercritical phase, but 
the bound $-\theta/2$ is not optimal for this phase. Let us now show that HMC$_{\theta}$ is a.s.\ not in $H^{-s}$ for $s > -\theta/2$. 

By Parseval's identity, for $1/2 < r < 1$, 
\begin{equation}
  \begin{split}
    \sum_{n=0}^{\infty}|c_{n}|^{2}r^{2n}
    = \frac{1}{2\pi}\int_{0}^{2\pi} e^{\sqrt{\theta}G(re^{i\vartheta})}\,d\vartheta. 
  \end{split}
\end{equation}
For any $s < 0$, the map $x \mapsto x^{2s} r^{-2x}$ from $(0, \infty)$ to $\mathbb{R}$ has a logarithmic derivative 
$2s/x - 2 \log r$, and then it reaches its minimum at $x = s / \log r$. Hence, for all $n \geq 1,$ 
\[
  (1+n^2)^{s}  r^{-2n} 
  \geq
  2^{s} n^{2s}  r^{-2n} 
  \geq 2^s (s/\log r)^{2s} r^{-2 s / \log r} 
  \geq C_s |\log r|^{-2s}
\]
where $C_s > 0$ depends only on $s$. 
Hence, for all $r \in (1/2,1)$, shrinking $C_s$ as needed to account for $n=0$ term and the $2\pi$,
\begin{equation}\label{eq:SobolevtoAbel}
  A_{s,\theta} 
  \geq 
  \sum_{n=0}^{\infty}(1+n^2)^{s}r^{-2n} r^{2n} |c_{n}|^{2}
  \geq
  \frac{C_s}{|\log r|^{2s}} \int_{0}^{2\pi} e^{\sqrt{\theta}G(re^{i\vartheta})}\,d\vartheta.
\end{equation}
Now, for $\theta \in (0,1]$, the quantity
$$(1- r^2)^{\theta} |\log (1-r^2)|^{(1/2) \mathbf{1}_{\theta = 1} } \int_{0}^{2\pi} e^{\sqrt{\theta}G(re^{i\vartheta})}\,d\vartheta$$
converges in probability to the total mass of the GMC$_{\theta}$ when $r$ goes to $1$, and then converges a.s.\ to a non-zero limit along a subsequence. 
Since for $s \in (-\theta/2, 0)$, 
$$  |\log r|^{-2s}  (1- r^2)^{-\theta} |\log (1-r^2)|^{-(1/2) \mathbf{1}_{\theta = 1} }$$
tends to infinity when $r \rightarrow 1$, we deduce that $A_{s, \theta}$ is a.s. infinite.

\subsection{Super-critical case, $\theta > 1$}
Let us prove that HMC$_{\theta}$ is a.s.\ in $H^s$ for all $s < - \sqrt{\theta} + 1/2$. 
One  checks that  $A_{s, \theta}$ is dominated by 
$$\sum_{u = 1}^{\infty} 2^{2su}  \sum_{n = 0}^{\infty} (1 - 2^{-u})^n |c_n|^2,$$
after considering the terms for which $u = \lfloor \log (2+n)/ \log 2 \rfloor$.  
Hence $A_{s, \theta} $ is dominated by 
$$\sum_{u = 1}^{\infty} 2^{2su} \int_{0}^{2\pi} e^{\sqrt{\theta}G((1-2^{-u}) e^{i\vartheta})}\,d\vartheta. $$
We have that $G((1-2^{-u}) e^{i\vartheta})$ is a centered  Gaussian variable with variance 
$$- 2 \log (1 - (1-2^{-u}) ^2) = -2 \log (2^{1-u} - 2^{-2u}) = 2u \log 2 + \mathcal{O}(1). $$
Hence, for $u$ large enough, the probability that $G((1-2^{-u}) e^{i\vartheta}) \geq 2u \log 2 + 10 \log u $ is dominated by 
\begin{align*}
\mathbb{P} \left[ \mathcal{N}(0,1) \geq  \frac{2u \log 2 +10  \log u}{  \sqrt{2u \log 2 + \mathcal{O}(1)}}  \right] & 
\leq e^{-  \frac{(2u \log 2 + 10 \log u)^2} { 4 u \log 2 + \mathcal{O}(1)} } 
\\ & \leq  e^{-  \frac{4 u ^2 (\log 2)^2 + 40 u (\log 2)( \log u)} { 4 u \log 2 + \mathcal{O}(1)} } 
\\ & \leq e^{ - (u  \log 2   + 10 \log u ) ( 1 + \mathcal{O}(1/u)) } = \mathcal{O}( 2^{-u} u^{-10} )
\end{align*}
By Borel-Cantelli lemma, we have, almost surely, 
$$\sup_{ \vartheta \in 2 \pi \mathbb{Z} / 2^u} G((1-2^{-u}) e^{i\vartheta}) \leq  2u \log 2 + 10 \log u$$
for $u$ large enough.  We let $\mathcal{G}_u$ be the event the previous display holds. 
In order to prove that $A_{s, \theta}$ is a.s.\  finite, it is then sufficient to show that 
 \[
   \sum_{u = 1}^{\infty} 2^{2su} 
   \int_{0}^{2\pi} e^{\sqrt{\theta}G((1-2^{-u}) e^{i\vartheta})} 
   \one[\mathcal{G}_u]
   \,d\vartheta < \infty 
 \]
almost surely. It is then enough to have 
\[
  \sum_{u = 1}^{\infty} 2^{2su} 
  \int_{0}^{2\pi} 
  \mathbb{E} [ e^{\sqrt{\theta}G((1-2^{-u}) e^{i\vartheta})}
  \one_{\mathcal{G}_u}
]
  \,d\vartheta < \infty.
\]
We let  $\chi(u, \vartheta)$ be a multiple of $(2 \pi)/2^u$ minimizing its distance $\vartheta.$
Using a change of measure formula, we can bound the expectation inside the integral by
\[
  \begin{aligned}
    &\mathbb{E} [ e^{\sqrt{\theta}G((1-2^{-u}) e^{i\vartheta})}
      \one_{\mathcal{G}_u}
    ] \\
    &\leq
    \mathbb{E} [ e^{\sqrt{\theta}G((1-2^{-u}) e^{i\vartheta})}
      \one_{ G((1-2^{-u}) e^{i \chi(u, \vartheta)}) \leq  2u \log 2 + 10 \log u}
    ] \\
    &=\mathbb{E} [ e^{\sqrt{\theta}G((1-2^{-u}) e^{i\vartheta})} ] 
    \mathbb{P} [ \widetilde{G}((1-2^{-u}) e^{i \chi(u, \vartheta)}) \leq  2u \log 2 + 10 \log u ]. \\
  \end{aligned}
\]
Here $ \widetilde{G}((1-2^{-u}) e^{i \chi(u, \vartheta)})$ is the Gaussian variable obtained from $G((1-2^{-u}) e^{i \chi(u, \vartheta)})$ after changing the underlying probability 
measure by a density proportional to  $e^{\sqrt{\theta}G(1-2^{-u})e^{i\vartheta}) }$. By Girsanov's theorem, this change of measure introduces a drift 
$$\operatorname{Cov} ( G ( (1-2^{-u}) e^{i \chi(u, \vartheta)}), \sqrt{\theta}G((1-2^{-u}) e^{i\vartheta}) )
 = -2 \sqrt{\theta} \log |  1- (1-2^{-u})^2 e^{ i v} |$$
 where $$|v| = |  \chi(u, \vartheta) - \vartheta |  \leq 2 \pi / 2^u  = \mathcal{O}(2^{-u}).$$
Hence, 
 $$|  1- (1-2^{-u})^2 e^{ i v} | = \mathcal{O}(2^{-u}),$$
which  implies  
$$\operatorname{Cov} ( G ( (1-2^{-u}) e^{i \chi(u, \vartheta)}), \sqrt{\theta}G((1-2^{-u}) e^{i\vartheta}) )
\geq \sqrt{\theta} ( 2 u \log 2 + \mathcal{O}(1)).$$
We then get, for $u$ large enough depending on $\theta$, 
\begin{align*}
& \mathbb{E} [ e^{\sqrt{\theta}G((1-2^{-u}) e^{i\vartheta})} \mathbf{1}_{ G((1-2^{-u}) e^{i \chi(u, \vartheta)}) \leq  2u \log 2 + 10 \log u} ] 
\\ & \leq \mathbb{E} [ e^{\sqrt{\theta}G((1-2^{-u}) e^{i\vartheta})} ] 
\\ & \times \mathbb{P} [G((1-2^{-u}) e^{i \chi(u, \vartheta)}) + \sqrt{\theta} ( 2 u \log 2 + \mathcal{O}(1)) \leq  2u \log 2 + 10 \log u ].
\\ & \leq e^{ \theta  ( 2u \log 2 + \mathcal{O}(1)) /2}  \cdot e^{ - ( 2u  ( \sqrt{ \theta}- 1)\log 2  + \mathcal{O} (\log u))^2 / 2 ( 2u \log 2 + \mathcal{O}(1))}
\\ &  \leq e^{ \theta  ( 2u \log 2 + \mathcal{O}(1)) /2}  \cdot e^{ -  ( u  ( \sqrt{ \theta}- 1)^2 \log 2 + \mathcal{O} ( \sqrt{\theta} ( \log u) ) ) ( 1 + \mathcal{O}(u^{-1}))}
\\ & \leq C_{\theta}  2^{u ( \theta - ( \sqrt{ \theta}- 1)^2 )} u^{m_{\theta}}  = C_{\theta} 2^{u (2 \sqrt{\theta} - 1)} u^{m_{\theta}} 
\end{align*}
for  $C_{\theta}, m_{\theta} > 0$ depending only on $\theta$. 
Hence, 
$$\sum_{u = 1}^{\infty} 2^{2su} \int_{0}^{2\pi} \mathbb{E} [ e^{\sqrt{\theta}G((1-2^{-u}) e^{i\vartheta})} \mathbf{1}_{ G((1-2^{-u}) e^{i \chi(u, \vartheta)}) \leq  2u \log 2 + 10 \log u} ] \,d\vartheta < \infty$$
as soon as $2s + 2 \sqrt{\theta} - 1 < 0$, which implies that for $\theta > 1$, HMC$_{\theta}$ is a.s. in $H^s$ for all  $s <  - \sqrt{\theta} + 1/2$. 

On the other hand, we have seen, from \eqref{eq:SobolevtoAbel}
$$
 A_{s,\theta} 
  \geq
  \frac{C_s}{|\log r|^{2s}} \int_{0}^{2\pi} e^{\sqrt{\theta}G(re^{i\vartheta})}\,d\vartheta.
$$
for all $r \in (1/2,1)$. 
Since $G$ is the real part of a holomorphic function, it is harmonic on the unit disc, and then 
by Jensen's inequality, for all $\vartheta_0 \in \mathbb{R}$, 
$$e^{\sqrt{\theta}G(r^2 e^{i \vartheta_0})} \leq  \frac{1}{2 \pi} \int_0^{2 \pi}  e^{\sqrt{\theta} G(r e^{i (\vartheta_0 + \vartheta )})} P_r(\vartheta) d \vartheta$$
where $P_r$ is the Poisson kernel: 
$$P_r( \vartheta) = \sum_{p \in \mathbb{Z}} r^{|p|} e^{i p \vartheta}.$$
The Poisson kernel is bounded by $(1+r)/(1-r)$, and then 
\[
  \begin{aligned}
2e^{\sqrt{\theta} \sup_{\vartheta_0 \in \mathbb{R}}  G(r^2 e^{i \vartheta_0})}  
&\leq 
\sup_{\vartheta_0 \in \R}
\frac{1}{2\pi}\int_0^{2 \pi}  e^{\sqrt{\theta} G(r e^{i (\vartheta_0 + \vartheta )})} P_r(\vartheta) d \vartheta \\
&\leq
\frac{1+r}{2\pi(1-r)} 
\int_0^{2 \pi}  e^{\sqrt{\theta} G(r e^{i  \vartheta })} 
d \vartheta.
\end{aligned}
\]
Hence, using \eqref{eq:SobolevtoAbel}, 
there is a constant $C_s > 0$ so that for any $r \in (\tfrac 12,1),$
\[
  A_{s, \theta}  \geq (1-r)^{1 - 2s} e^{\sqrt{\theta} \sup_{\vartheta \in \mathbb{R}}  G(r^2 e^{i \vartheta})}.
\]
With probability going to $1$ when $r \rightarrow 1$, the maximum of the logarithmically correlated field $G(r^2 e^{i \vartheta})$
is larger than $2 \log (1/(1-r) ) - 10 \log ( \log (1/(1-r)))$ (using the results of \cite{DingRoyZeitouni} on a suitably chosen net), and then $A_{s, \theta}$ is bounded from below by $(1-r)^{1-2s - 2 \sqrt{\theta} }  \log (1/(1-r))^{m_{\theta}}$
for $m_{\theta}$ depending only on $\theta$. For $s > -\sqrt{\theta} + 1/2$, the bound tends to infinity when $r \rightarrow 1$, which shows that HMC$_{\theta}$ is a.s.\ not in $H^s$.

\section{The circular $\beta$--ensemble} \label{sec:cbe}
In this section we develop some of the properties of the secular coefficients of C$\beta$E $\{ c_n^{(N)} \}.$  We begin by recalling the \emph{Verblunsky} coefficients which will allow us to formulate the exact relationship, due to \cite{KillipNenciu}.  We let $\left\{ \alpha_n : n \in \N_0\right\}$ be independent complex random variables on the unit disk, with each $\alpha_n$ rotationally invariant in law and $|\alpha_n|^2$ distributed like $\Beta(1, \tfrac{\beta(n+1)}{2})$.  The \emph{Szeg\H{o}} recurrence is, for all $N \geq 0,$ 
\begin{equation}
  \begin{pmatrix}
    \Phi_{N+1}(z) \\
    \Phi^*_{N+1}(z)
  \end{pmatrix}
  \coloneqq
  \begin{pmatrix}
    z & -\overline{\alpha_N} \\
    -\alpha_N z & 1
  \end{pmatrix}
  \begin{pmatrix}
    \Phi_{N}(z) \\
    \Phi^*_{N}(z)
  \end{pmatrix},
  \quad
  \biggl\{
    \begin{aligned}
      &\Phi_0(z) \equiv 1, \\
      &\Phi_N^*(z) = z^N \overline{ \Phi_N(1/\overline{z})}.
    \end{aligned}
    \label{eq:szego}
  \end{equation}
  where $\Phi_{N}^{*}$ and $\Phi_{N}$ are polynomials of degree at most $N$. Note that $\Phi_N^*$ and $\Phi_N$ are related by being reversals of one another, in that their vector of coefficients is reversed and conjugated. 
  Now we give the connection to the $C\beta E$ characteristic polynomial as defined in \eqref{charpoly}, due to \cite{KillipNenciu}. We let $\eta$ be uniformly distributed random variable on the unit circle independent of $\left\{ \alpha_n : n \in \N_0 \right\}.$
  The characteristic polynomial $\chi_N$ has the distribution of
  \begin{equation}
    \chi_N(z) = 
    \Phi_{N-1}^*(z) - \eta z{\Phi_{N-1}(z)}.
    \label{eq:chiN}
  \end{equation}
  From \cite[Proposition 3.1]{ChhaibiNajnudel}, we have almost sure convergence as $N \to \infty$ of $\Phi_N^*(z)$ to a process $\Phi_\infty^*(z)$ on the unit disk uniformly on compact sets.  Moreover, this limit is none other than the log-Gaussian process $e^{\sqrt{\theta}G^{\C}(z)}.$  In \cite[Proposition 3.1]{ChhaibiNajnudel}, 
  the following uniform estimate is also 
  proven:
  \begin{equation}
    \Exp |\Phi_N^*(z)|^p \leq (1-|z|^2)^{-\theta p^2/4}
    \quad
    \text{for all}
    \quad z \in \D, p > 0, N \in \N.
    \label{eq:PhiNbnd}
  \end{equation}

  We let $\filt$ be the filtration $\filt = \left( \filt_N \coloneqq \sigma(\alpha_0, \dots, \alpha_{N-1}) : N \geq 0 \right).$  Then the process $\left\{ \Phi_N^*(z) \right\}$ is adapted to $\filt$ and moreover is a complex martingale.  Using Cauchy's theorem, and the bound \eqref{eq:PhiNbnd}, we have that secular coefficient
  \begin{equation}\label{eq:Mnn}
    \mathscr{M}_{n,N} \coloneqq \frac{1}{2\pi i}\oint \Phi_N^*(z) z^{-n-1} dz,
  \end{equation}
  for a fixed simple closed contour enclosing $0$ in the unit disk,
  forms a uniformly integrable martingale adapted to $\filt_N.$ Hence we have the representation
  \[
    \mathscr{M}_{n,N} = \Exp[ c_n ~\vert~\filt_N] \Asto[N] c_n.
  \]

 Using the Szeg\H{o} recurrence \eqref{eq:szego} and \eqref{eq:Mnn}, we have the identity
  \begin{equation}
    \begin{aligned}
      \mathscr{M}_{n,N+1} 
      &=\mathscr{M}_{n,N} - \alpha_N \frac{1}{2\pi i}\oint \Phi_N(z) z^{-n} dz \\
      &=\mathscr{M}_{n,N} - \alpha_N \overline{\mathscr{M}_{N-n+1,N}}.
    \end{aligned}
    \label{eq:Mnnrecurrence}
  \end{equation}
  We will let $\mathfrak{B}_{n,N}$ be the bracket process of $\mathscr{M}_{n,N},$ i.e.\ for $n \geq 1$, and for any $N \geq n-1$,
  \begin{equation} \label{eq:Mnnbrack}
    \begin{aligned}
      \mathfrak{B}_{n,N}
      &\coloneqq\sum_{j=n-1}^{N-1} \Exp[ |\mathscr{M}_{n,j+1}-\mathscr{M}_{n,j}|^2~\vert~\filt_j] \\
      &=\sum_{j=n-1}^{N-1} |{\mathscr{M}_{j-n+1,j}}|^2\Exp(|\alpha_j|^2) 
      =\sum_{j=n-1}^{N-1} \frac{|{\mathscr{M}_{j-n+1,j}}|^2}{1+\frac{1}{\theta}(j+1)}.
    \end{aligned}
    \end{equation}
  Here, we can notice that $\mathscr{M}_{n,n-1}= 0$, since $\Phi^*_{n-1}$ is a polynomial of degree at most $n-1$.

  \begin{lemma}\label{lem:expectedbracket}

    For any $\theta >0$ there exists a constant $C_{\theta} > 0$ such that the following holds. For  $N \geq n \geq 1$,  we have
   $\Exp |\mathscr{M}_{n,N}|^2 
      =\Exp(\mathfrak{B}_{n,N})$, and 
    \[
      \Exp |\mathscr{M}_{n,N}|^2 
      \leq C_\theta  
      \left\{
	\begin{aligned}
	  &\frac{(N-n+1)^{\theta}}{n}, & \text{if } \theta < 1,  \\
	&\frac{(N-n+1)}{n} \left( 1 + \max\biggl(\log \biggl(\frac{n}{N-n+1}\biggr), 0\biggr) \right), 
	& \text{if } \theta = 1,  \\
	&\frac{n^{\theta}}{n}\log\biggl(1+\frac{(N-n+1)}{n}\biggr), 
	& \text{if } \theta > 1.  \\
      \end{aligned}
	\right.
    \]
  Moreover, for all  $n,N_1,N_2 \in \N$ such that  $1 \leq 3n/2 \leq N_1 < N_2$, we have 
    \[
      \Exp |\mathscr{M}_{n,N_2} - \mathscr{M}_{n,N_1}|^2 
      =\Exp(\mathfrak{B}_{n,N_2}-\mathfrak{B}_{n,N_1})
      \leq C_\theta n^{\theta}\biggl(\frac{1}{N_1}-\frac{1}{N_2}\biggr).
    \]
    This holds with $N_2=\infty$ as well, in which case $\mathscr{M}_{n,N_2}=c_n$ and $1/N_2 := 0$.
   
  \end{lemma}
  \begin{proof}
  For the last part of the lemma, as $\mathscr{M}_{n,N_2}$ is bounded in $L^p$ for all $p,$ the $N_2=\infty$ case follows from uniform integrability and taking $N_2 \to \infty.$ Hence it suffices to show only the case of $N_2 < \infty.$  Also, it follows the bracket process is uniformly bounded in $L^p$ for all $p$ from the Burkholder-Davis-Gundy inequalities. 
For $j = n-1$, we have $$|{\mathscr{M}_{j-n+1,j}}|^2 = |{\mathscr{M}_{0,n-1}}|^2 = 1$$
since the constant term of $\Phi^*_{n-1}$ is equal to $1$. 
Hence, 
for $N \geq n \geq 1$, we can write
$$\mathfrak{B}_{n,N} = \frac{1}{ 1 + \frac{n}{\theta}} +   \sum_{j=n}^{N-1} \frac{|{\mathscr{M}_{j-n+1,j}}|^2}{1+\frac{1}{\theta}(j+1)}.$$

    Recall that for $j \geq n$, the bracket is such that $k\mapsto |{\mathscr{M}_{j-n+1,k}}|^2-\mathfrak{B}_{j-n+1,k}$ is again a martingale, starting
    at zero for $k = j-n$,  and hence
    \begin{equation} \label{eq:Mnnexpected}
      \begin{aligned}
	\Exp \mathfrak{B}_{n,N}
	= \frac{1}{ 1 + \frac{n}{\theta}}  + \sum_{j=n}^{N-1} \frac{\Exp \mathfrak{B}_{j-n+1,j}}{1+\frac{1}{\theta}(j+1)}.
      \end{aligned}
    \end{equation}
    We may develop this equation to produce
    \begin{equation} \label{eq:Mnnexpected2}
      \begin{aligned}
	\Exp \mathfrak{B}_{n,N}
	&=\frac{1}{ 1 + \frac{n}{\theta}} 
	+ \sum_{j=n}^{N-1} \frac{1}{1+\frac{1}{\theta}(j+1)}
	\left( \frac{1}{ 1 + \frac{j-n + 1}{\theta}}  +  \sum_{k=j-n+1}^{j-1}
	\frac{\Exp \mathfrak{B}_{k-(j-n+1)+1,k}}{1+\frac{1}{\theta}(k+1)} \right)\\
	&= \frac{1}{ 1 + \frac{n}{\theta}}  + \sum_{j=n}^{N-1} 
	\frac{1}{(1+\frac{1}{\theta}(j+1))( 1 + \frac{j-n + 1}{\theta}) }
	\\ &+ \sum_{j=n}^{N-1} \sum_{k=1}^{n-1} \frac{1}{1+\frac{1}{\theta}(j+1)}
	\frac{\Exp \mathfrak{B}_{k,k+j-n}}{1+\frac{1}{\theta}(k+j-n+1)}.
      \end{aligned}
    \end{equation}
    Using uniform integrability of the bracket, we have for any $k \geq 0,$
    \[
      \Exp \mathfrak{B}_{k,\infty} = \Exp |c_k|^2.
    \]
    Let us now prove the  estimate of  $\Exp |\mathscr{M}_{n,N}|^2 = \Exp \mathfrak{B}_{n,N}$ given in the lemma. We have 
     \begin{equation}
     \begin{aligned}
	\Exp \mathfrak{B}_{n,N}
	& \leq \frac{1}{ 1 + \frac{n}{\theta}}  + \sum_{j=n}^{N-1} 
	\frac{1}{(1+\frac{1}{\theta}(j+1))( 1 + \frac{j-n + 1}{\theta}) }
	\\ &+ \sum_{j=n}^{N-1} \sum_{k=1}^{n-1} \frac{1}{1+\frac{1}{\theta}(j+1)}
	\frac{\Exp |c_k|^2}{1+\frac{1}{\theta}(k+j-n+1)}.
      \end{aligned}
      \end{equation}
    Since
   $j \geq n$  and  $\mathbb{E} [ |c_k|^2]$ is dominated by $k^{\theta-1}$, $\Exp \mathfrak{B}_{n,N}$ is bounded, up to a constant depending 
    only on $\theta$, by
    $$\frac{1}{n} \left( 1 + \sum_{j=n}^{N-1} \frac{1}{j-n+1} + \sum_{j=n}^{N-1} \sum_{k= 1}^{n-1} \frac{k^{\theta-1}}{k+j-n+1} \right).$$
    For $N =n$, we immediately 
   deduce a bound of order $1/n$ (the sums are empty) which is enough for our purpose. We can then assume $N \geq n+1$, and in this case, we have 
    
   \begin{equation}\label{eq:Mnnisum}
     \frac{1}{n}
     \sum_{j=1}^{N-n} \sum_{k= 0}^{n-1}  
     \frac{  (1+k)^{\theta-1} }{ k+j}
     \leq 
     \frac{1}{n} 
     \sum_{j=1}^{N-n} \sum_{k= 0}^{n-1}  
     \frac{  (1+k)^{\theta-1} }{ \max\{j,1+k\} }.
   \end{equation}
    If $j \leq n$, the inner sum can be estimated as 
    \[
      \sum_{0 \leq k \leq j-1} \frac{  (1+k)^{\theta-1} }{ j} + \sum_{j-1 < k \leq n-1} (1+k)^{\theta-2} 
      \leq
      \begin{cases}
	C_\theta j^{\theta-1}, & \text{ if } \theta < 1, \\
	1+\log(1+n/j), & \text{ if } \theta = 1, \\
	C_\theta n^{\theta-1}, & \text{ if } \theta > 1, \\
      \end{cases}
    \]
    where $C_\theta > 0$ is some sufficiently large constant.
    If $j > n$, the inner sum of \eqref{eq:Mnnisum}
    is
    \[
      \sum_{0 \leq k \leq n-1} \frac{  (1+k)^{\theta-1} }{ j}
      \leq 
      \frac{C_\theta n^{\theta}}{j}
    \]
    for some $C_\theta>0$ sufficiently large.
    Summing in $j$, we get from \eqref{eq:Mnnisum} and the bounds just established, that for $\theta < 1,$
    \[
      \frac{1}{n}
      \sum_{j=1}^{N-n} \sum_{k= 0}^{n-1}  
      \frac{  (1+k)^{\theta-1} }{ k+j}
      \leq \frac{C_\theta}{n} \biggl( (N-n+1)^\theta + n^{\theta} \log\biggl( 1 +\frac{N-n+1}{n}\biggr)  \biggr),
    \]
    which is dominated by $(N-n+1)^{\theta}/n$. 
    For $\theta > 1$, we get 
    \[
      \frac{1}{n}
      \sum_{j=1}^{N-n} \sum_{k= 0}^{n-1}  
      \frac{  (1+k)^{\theta-1} }{ k+j}
      \leq \frac{C_\theta}{n} \biggl(  n^{\theta-1}( (N-n+1) \wedge n) + n^{\theta} \log\biggl( 1 +\frac{N-n+1}{n}\biggr)  \biggr),
    \]
    a domination by $n^{\theta-1} \log\bigl( 1 +\frac{N-n+1}{n}\bigr)$. 
    For $\theta = 1$, we get a domination by 
    \[
      \frac{1}{n}
      \sum_{j=1}^{N-n} \sum_{k= 0}^{n-1}  
      \frac{  (1+k)^{\theta-1} }{ k+j}
      \leq
      \frac{C_\theta}{n}\sum_{j=1}^{N-n} (1 + \operatorname{max}(0, \log \tfrac{n}{j}))).
    \]
    For $1 \leq N-n \leq n$, this quantity is at most 
    $$
    \begin{aligned}
    & n^{-1} ( (N-n) (1 +\log n) - \log  ((N-n)!) ) \\ & \leq n^{-1} ((N-n) (1 +\log n) - (N-n) \log  (N-n)  + (N-n))
    \\ & \leq n^{-1} (N-n) (2 +\log (n/(N-n))) 
    \end{aligned}
     $$
    and then we have, for all $N \geq n+1$, a domination by 
    $$n^{-1} (N-n) (1 + \max(0, \log (n/(N-n)))).$$
    We have now proven the first part of the lemma. 
    For the second part, we observe that under the assumptions of the lemma, 
    \begin{equation} \label{eq:Mnnexpected3}
      \begin{aligned}
	\Exp(\mathfrak{B}_{n,N_2} - \mathfrak{B}_{n,N_1})
	&\leq\sum_{j=N_1}^{N_2-1} 
	\sum_{k=1}^{n-1}
	\frac{1}{1+\frac{1}{\theta}(j+1)}
	\frac{\Exp(|c_k|^2)}{1+\frac{1}{\theta}(k+j-n+1)}
	\\ &+ \sum_{j=N_1}^{N_2-1} \frac{1}{1+\frac{1}{\theta}(j+1)}
	 \frac{1}{ 1 + \frac{j-n + 1}{\theta}}  
      \end{aligned}
    \end{equation}
    We can estimate for some constant $C_\theta >0,$
    \begin{equation*} 
      \begin{aligned}
	\Exp(\mathfrak{B}_{n,N_2} - \mathfrak{B}_{n,N_1})
	&\leq C_\theta\sum_{j=N_1}^{N_2} 
	\sum_{k=1}^{n-1}
	\frac{1}{1+\frac{1}{\theta}(j+1)}
	\frac{k^{\theta-1}}{1+\frac{1}{\theta}(k+j-n+1)}
	\\ &+ \sum_{j=N_1}^{N_2-1} \frac{1}{1+\frac{1}{\theta}(j+1)}
	 \frac{1}{ 1 + \frac{j-n + 1}{\theta}}  
      \end{aligned}
    \end{equation*}
    If $N_1 \geq 3n/2$, we have that $j-n \geq j/3$, and then the sum of the terms  corresponding to a given value of $j$ is dominated by $ n^{\theta}/j^2$, which proves the second part of the lemma.  
  \end{proof}

  Having an estimate on the secular coefficients of $\Phi_{N-1}^*,$ we may use \eqref{eq:chiN} to relate $c_n^{(N)}$ to $c_n.$  Recasting \eqref{eq:chiN} in terms of coefficients,
  \begin{equation} \label{eq:chiNcoeff}
    c_n^{(N)} = \mathscr{M}_{n,N-1} - \eta \overline{\mathscr{M}_{N-n,N-1}}.
  \end{equation}
  We deduce the following estimates relating the secular coefficients of C$ \beta$E to the coefficients of the HMC: 
  \begin{lemma}\label{lem:secdiff}
    There is constant $C_\theta$ so that for all $n,N \in \N$ with $1 \leq n \leq N/2$,
    \[
      \Exp |c_n - c_{n}^{(N)}|^2 
      \leq  
      \begin{cases}
	\frac{C_\theta n^{\theta}}{N}& \text{ if } \theta < 1, \\
	\frac{C_\theta n \log (N/n)}{N}& \text{ if } \theta = 1, \\
	\frac{C_\theta n}{N^{2- \theta}}& \text{ if } \theta > 1. \\
      \end{cases}
    \]
%
  \end{lemma}
  \begin{proof}
    Using \eqref{eq:chiNcoeff}, 
    \[
      \Exp |c_n - c_{n}^{(N)}|^2 
      =
      \Exp |c_n -\mathscr{M}_{n,N-1}|^2 
      +\Exp |{\mathscr{M}_{N-n,N-1}}|^2.
    \]
    We have $N \geq  2n$ and then $N-1 \geq 3n/2$ (except for $n =1$ and $N=2$, in which case the lemma is obvious for a suitable value of $C_{\theta}$). Hence, by  Lemma \ref{lem:expectedbracket}, we have, 
     after increasing the constant $C_\theta$ if needed,
   $$
      \Exp |c_n -\mathscr{M}_{n,N-1}|^2 
      \leq \frac{C_\theta n^{\theta}}{N},$$
      and
      \[
      \Exp |{\mathscr{M}_{N-n,N-1}}|^2 
      \leq  
      \begin{cases}
	\frac{C_\theta n^{\theta}}{N}& \text{ if } \theta < 1, \\
	\frac{C_\theta n \log (N/n)}{N}& \text{ if } \theta = 1, \\
	\frac{C_\theta n}{N^{2- \theta}}& \text{ if } \theta > 1. \\
      \end{cases}
    \]
    This latter part dominates the contribution 
    \(|{\mathscr{M}_{N-n,N-1}}|^2\)
    in all cases.
  \end{proof}
  From this point, the proof of Theorem \ref{th:l2phase-N} is a simple corollary.
  \begin{proof}[Proof of Theorem \ref{th:l2phase-N}]
    Under the assumption that $n/N \to 0,$
    \[
      (\Exp| c_n - c_n^{(N)}|^2)n^{1-\theta} \to 0
      \qquad
      \text{as } 
      n \to \infty,
    \]
    and so we conclude from Slutsky's theorem, \eqref{haake2nd} and Theorem \ref{th:l2phase} that the desired distributional convergence holds.

\end{proof}

\subsection{Fractional moments}
  \label{sec:qmoments}

  In this section, we prove Theorem \ref{thm:order}.  We recall from the statement of that theorem that $w_n$ is defined as
  \begin{equation}\label{eq:wn}
    w_n = n^{(\theta-1)/2},
    \quad
    w_n = (\log(1+n))^{-\frac{1}{4}},
    \quad \text{or} \quad
    w_n = n^{\sqrt{\theta}-1}(\log n)^{-\frac{3}{4}\sqrt{\theta}},
  \end{equation}
  in the cases $\theta \in (0,1),$ $\theta = 1$ or $\theta > 1$ respectively.

  We recall for convenience \eqref{eq:l2conda} and recall that $S^{(n)}$ is the shortest cycle in a Ewens distributed permutation of $n$ symbols (c.f.\ Lemma \ref{lem:shortees}):
  \begin{equation*} 
    \mathbb{E}(|c_{n}|^{2} ~\vert~ \Gfilt_q) 
    =
    |c_{n,q}|^2
    +
    \sum_{r=0}^{n} |c_{r,q}|^2 \frac{ \theta^{(n-r)}}{(n-r)!} 
    \Pr[ 
      S^{(n-r)} > q
    ].
  \end{equation*}
  Using Lemma \ref{lem:shortees}, we have a uniform bound 
  \begin{equation} \label{eq:unibound}
    \mathbb{E}(|c_{n}|^{2} ~\vert~ \Gfilt_q) 
    \leq
    |c_{n,q}|^2
    +
    C_\theta
    n^{\theta-1}
    F_q
    \quad
    \text{where}
    \quad 
    F_q \coloneqq
    \sum_{r=0}^\infty |c_{r,q}|^2 
    e^{-\theta h(q+1)}.
  \end{equation}
  Using Parseval's identity and \eqref{eq:cnqgen}, 
  \begin{equation} \label{eq:cnqparseval}
    \sum_{r=0}^\infty |c_{r,q}|^2 
    =
    \frac{1}{2\pi}
    \int_0^{2\pi}
    e^{\sqrt{\theta} G_q(e^{i\vartheta})}
    d\vartheta,
    \quad
    \text{where}
    \quad
    G_q(z)=2\Re \sum_{k=1}^q \frac{z^k \mathcal{N}_k}{\sqrt{k}}.
  \end{equation}
  Moreover the normalizing constant $e^{-\theta h(q+1)}$ is such that $F_q$ has expectation $1,$ and $F_q$ is thus an approximation to the mass of the chaos for $\theta < 1$.
  \begin{lemma} \label{lem:Fq}
  Let  $w_n$ as in \eqref{eq:wn}. If $\theta \in (0,1)$, 
  then
   $$ \Exp[n^{\theta-1}{F_n}/w_n^2]  = 1$$
   for all $n \geq 1$. 
  If $\theta = 1$, then for $p \in (0,1)$, 
  \[
    \sup\{  \Exp[({F_n}/w_n^2)^p] : n \in \N \} < \infty.
  \]
  Finally, if $\theta > 1$, there exists $u_{\theta} > 0$, depending only on $\theta$,  such that 
  \[
    \left\{ (\log n)^{-u_{\theta}} n^{\theta-1}{F_n}/w_n^2   : n \in \N \right\}
  \]
  is tight.
 \end{lemma}
\begin{proof}
  For $\theta  \in (0,1)$, the lemma means that  $F_n$ has expectation $1$, as remarked before. 
  For $\theta = 1,$ the lemma follows from \cite[Theorem 1.3]{JunnilaSaksman}.  
  For $\theta >1$ we use the same ingredients as those considered in the proof of regularity of the supercritical HMC. 
  We will show the equivalent statement that there exists a $u_\theta > 0$ so that
  \[
    (\log n)^{-u_{\theta}} n^{\theta-1}{F_n}/w_n^2
    \Prto[n] 0.
  \]
  The convergence to zero in probability is obtained by proving the convergence of the expectation, after restricting it to the event that 
  $G_n (e^{i\vartheta}) \leq 2 \log n + 10 \log \log n$ for all $\vartheta \in 2 \pi \mathbb{Z} / n$, whose probability goes to $1$ when $n \rightarrow \infty$. 
  Using Girsanov's theorem, 
  we get, for $n \geq 2$ and $|\vartheta - \vartheta' | = \mathcal{O}(1/n)$, that 
  \begin{align*}
    &   \mathbb{E}\left[  e^{\sqrt{\theta} G_n(e^{i\vartheta})} \mathbf{1}_{G_n(e^{i\vartheta'}) \leq 2 \log n + 10 \log \log n} \right]
    \\ 
    & = e^{\theta (\log n + \mathcal{O}(1))} \mathbb{P} \left[ G_n(e^{i\vartheta'} )+ \operatorname{Cov} ( G_n(e^{i\vartheta'} ), \sqrt{\theta}  G_n(e^{i\vartheta}) )  \leq 2 \log n + 10 \log \log n  \right]
    \\ & = (\mathcal{O}(n))^{\theta}  \mathbb{P} \left[ G_n(e^{i\vartheta'} ) \leq 2 ( 1- \sqrt{\theta} ) \log n + 10 \log \log n + c_{\theta} \right]. 
  \end{align*}
  for some $c_{\theta} \in \mathbb{R}$ depending only on $\theta$, since, from the fact that $\vartheta$ and $\vartheta'$ are close to each other, the covariance between $G_n(e^{i\vartheta'} )$ and $G_n(e^{i\vartheta})$ is 
  equal to $ 2 \log n + \mathcal{O}(1)$.  
  Hence, since $\theta > 1$, 
  \begin{equation}\label{eq:Gsubopt}
    \mathbb{E} \left[  e^{\sqrt{\theta} G_n(e^{i\vartheta})} \mathbf{1}_{G_n(e^{i\vartheta'}) \leq 2 \log n + 10 \log \log n} \right]
  \leq C_{\theta} n^{\theta - (1- \sqrt{\theta})^2}  (\log n  )^{\mu_{\theta}}
\end{equation}
  for $C_{\theta}, \mu_{\theta} > 0$ depending only on $\theta$. 
  Taking $\vartheta' \in 2 \pi \mathbb{Z}/n$ depending only of $\vartheta$ and $n$ in such a way that $|\vartheta - \vartheta' | = \mathcal{O}(1/n)$, and 
 integrating in $\vartheta$, we deduce 
\begin{align*}
& \mathbb{E} \left[ n^{\theta-1} F_n w_n^{-2} \mathbf{1}_{\sup_{\vartheta' \in 2 \pi \mathbb{Z}/n} G_n(e^{i\vartheta'}) \leq 2 \log n + 10 \log \log n} \right]
\\ &  \leq C'_{\theta} n^{\theta-1} n^{-\theta} n^{\theta - (1- \sqrt{\theta})^2}  (\log n  )^{\frac{3}{2}\sqrt{\theta} + \mu_{\theta}} n^{- 2\sqrt{\theta} + 2}  = C'_{\theta} (\log n  )^{\frac{3}{2}\sqrt{\theta} + \mu_{\theta}}
\end{align*}
for $C'_{\theta} > 0$. This proves the lemma with $u_{\theta} = \frac{3}{2}\sqrt{\theta}+ \mu_{\theta}$. 
\end{proof}

\begin{remark} \label{rem:sharpu}
  While beyond the scope of what we do here, the optimal power is $u_\theta =0,$ in which case the statement would be 
  \[
     \left\{ n^{\theta-1}{F_n}/w_n^2 :n \in \N \right\}
  \]
  is a tight family.  Indeed as a consequence of \cite{DingRoyZeitouni}, the maximum of $G_n(e^{i\vartheta})$ on the lattice $2\pi \Z/n$ is at most $2\log n - \frac{3}{2}\log\log n + y$ with a probability controlled uniformly in $n$ by $r.$  Moreover, one can establish (and herein lies the technical work) that the random walk $( G_{2^k}(e^{i\vartheta}) : k < \log_2 n)$ stays below $(2\log 2) (k-\tfrac{3}{4}\log k) + (k(n-k)/n)^{1/10} + y$ for all lattice $\vartheta$ again except with a probability that can be made small in $r$ uniformly in $n.$  On this good event,  $\mathcal{G}(n,y),$ it is now possible to show using the appropriate ballot theorem that for $\theta > 1$ 
  \[
    \mathbb{E} \left[  e^{\sqrt{\theta} G_n(e^{i\vartheta})} \one[ \mathcal{G}(n,y)]\right]
    \leq
    C_{\theta,y} \frac{e^{\sqrt{\theta} ( 2\log n - \frac{3}{2}\log\log n)}}{n},
  \]
  which improves \eqref{eq:Gsubopt} and from which would follow the claimed bound.
\end{remark}

\begin{lemma} \label{lem:limpmoment}

If $\theta \in (0,1)$, then
\[
  \sup\{  \Exp(|c_n/w_n|^2) : n \in \N \} < \infty.
\]
If $\theta = 1$, and $p \in (0,1)$, then
\[
  \sup\{  \Exp(|c_n/w_n|^{2p}) : n \in \N \} < \infty.
\] 
Finally, if $\theta > 1$, there exists $v_{\theta}>0$, depending only on $\theta$, such that 
\[
  \{ (c_n/w_n) (\log n )^{-v_{\theta}} : n \in \N\} 
\]
is tight.

\end{lemma}
\begin{proof}
The case $\theta \in (0,1)$ is already proven before. 
Moreover, using \eqref{eq:expcnq} and Lemma \ref{lem:Lnlowertail}, we know that  for $n \geq 3$, 
  \begin{equation}\label{eq:errorterm}
    \max_{1 \leq q < n/\log n}
    \Exp |c_{n,q}|^2 
    \leq
    \max_{1 \leq q < n/\log n}
    C_\theta n^{\theta-1} \Pr(L^{(n-q)} \leq q)
    \leq n^{-\omega(n)},
  \end{equation}
  where $\omega(n)\to\infty$  as $n \to \infty.$

  We now use \eqref{eq:unibound} for $q_n = \lfloor  n/\log n \rfloor$ when $n \geq 3$: 
   $$
    \mathbb{E}(|c_{n}|^{2} ~\vert~ \Gfilt_{q_n}) 
    \leq
    |c_{n,q_n}|^2
    +
    C_\theta
    n^{\theta-1}
    F_{q_n}.
  $$
For $\theta =1$, we take $p \in (0,1)$ and get, using H\"older's inequality for exponent $1/p>1$ and 
 sub-additivity of the $p$-th power: 
$$\mathbb{E}(|c_{n}|^{2p} ~\vert~ \Gfilt_{q_n}) \leq \left [ \mathbb{E}(|c_{n}|^{2} ~\vert~ \Gfilt_{q_n}) \right]^p
\leq |c_{n,q_n}|^{2p}
    +
    (C_1
    F_{q_n})^p,$$
which implies 
  \[
    \Exp(|c_n|^{2p})
    \leq
    \Exp(|c_{n,q_n}|^{2p})
    +
    \Exp(| C_1 F_{q_n}|^{p}).
  \]
  From \eqref{eq:errorterm}, and conclude
  \[
    \Exp(|c_n/w_n|^{2p})
    \leq \mathcal{O}(n^{-p\omega(n)}/w_n^2)
    +
    \Exp(| C_1  F_{q_n}/w_n^2|^{p}).
  \]
  We then use Lemma \ref{lem:Fq} to complete the proof, taking into account the fact that $w_n = (\log (1+n))^{-1/4}$ is equivalent to $w_{q_n}$ when $n$ goes to infinity. 
  When $\theta > 1$, we write
  
  \begin{align*}
    \mathbb{E}( \min (1, &  |c_{n} w_n^{-1} (\log n)^{-v_{\theta}} |^{2}  )~\vert~ \Gfilt_{q_n}) 
    \leq
    w_n^{-2} (\log n)^{-2 v_{\theta}}  |c_{n,q_n}|^2
    \\ &    +
    \min(1, C_\theta 
    w_n^{-2} (\log n)^{-2 v_{\theta}}   n^{\theta-1}
    F_{q_n}). 
  \end{align*}
  Since $w_n/ w_{q_n}$ and $n^{\theta-1}/ q_n^{\theta-1}$ are equivalent up to powers of $\log n$ (depending on $\theta$) when $n \rightarrow \infty$, and $\log n$ is equivalent to 
  $\log q_n$, we deduce from Lemma \ref{lem:Fq}  that 
  \[
    w_n^{-2} (\log n)^{-2v_{\theta}}   n^{\theta-1}
    F_{q_n} \Prto[n] 0
  \] 
  for any $v_\theta$ sufficiently large.
  Then 
  \[
    \mathbb{E} [ \min(1, C_\theta 
      w_n^{-2} (\log n)^{-2 v_{\theta}}   n^{\theta-1}
    F_{q_n})] \underset{n \rightarrow \infty}{\longrightarrow} 0,
  \]
  as soon as $v_{\theta}$ is large enough. 
  Since $\mathbb{E}[ |c_{n,q_n}|^2] $ decreases to $0$ faster than any power of $n$, we have 
  \[
    \mathbb{E}  [ w_n^{-2} (\log n)^{-2 v_{\theta}}  |c_{n,q_n}|^2] \underset{n \rightarrow \infty}{\longrightarrow} 0,
  \]
  and then 
  \[
    \mathbb{E}[ \min (1,  |c_{n} w_n^{-1} (\log n)^{-v_{\theta}} |^{2}) ] \underset{n \rightarrow \infty}{\longrightarrow} 0,
  \]
  which implies that 
  \[ 
    |c_{n} w_n^{-1} (\log n)^{-v_{\theta}} | \Prto[n] 0.
  \]
\end{proof}

We deduce the following result on the secular coefficients of $C \beta E$. 
\begin{theorem}\label{thm:trueCBEbound}

If $\theta \in (0,1)$, then
\[
  \sup\{  \Exp(|c^{(N)}_n/w_n|^2) : n \in \N, N \geq 2n \} < \infty.
\]
If $\theta = 1$, and $p \in (0,1)$, then
\[
  \sup\{  \Exp(|c^{(N)} _n/w_n|^{2p}) : n \in \N, N \geq 2n \} < \infty.
\]
Finally, if $\theta  \in (1,2)$, there exists $v_{\theta}, v'_{\theta} > 0$ depending only on $\theta$ so that with $N_0(n) = n^{\frac{3-2\sqrt{\theta}}{2-\theta}} (\log n)^{v'_\theta},$ 
for all $\delta > 0$, 
\[
  \sup_{N \geq N_0(n)}
\, \mathbb{P} [ (c^{(N)}_n/w_n) (\log n )^{-v_{\theta}}  \geq \delta ] \underset{n \rightarrow \infty}{\longrightarrow} 0.
\]

\end{theorem}
\begin{proof}
Let us first suppose that $\theta \leq 1$: we may assume $p \in (1/2, 1)$ when $\theta = 1$, by H\"older's inequality.  
  For $r \geq 1$ and any $k,n \in \N$ we have from Jensen's inequality
  \[
    \Exp |\mathscr{M}_{n,k}|^r
    =
    \Exp |\Exp( c_n ~\vert~ \filt_k)|^r
    \leq
    \Exp ( |c_n|^r).
  \]
  Hence recalling \eqref{eq:chiNcoeff}, for any $p \geq 1,$
  \[
    (\Exp |c_n^{(N)}|^r)^{1/r}
    \leq
    (\Exp ( |c_n|^r))^{1/r}
    +
    (\Exp ( |c_{N-n}|^r))^{1/r}.
  \]
   We can now  conclude the proof from Lemma \ref{lem:limpmoment}, taking $r = 2$ for $\theta \in (0,1)$ and $ r = 2p \in (1,2)$ for $\theta = 1$: 
   notice that $w_{N-n} \leq w_n$ since $\theta \leq 1$ and $N \geq 2n$. 
   
  Let us now suppose that $\theta \in  (1,2)$. In this case, because of Lemma \ref{lem:limpmoment}, it is sufficient to 
  prove that 
  \[
    \sup_{N \geq N_0(n)}\mathbb{E} [ |c_n - c_n^{(N)} |^2 w_n^{-2} (\log n)^{-2 v_{\theta}} ] \underset{n \rightarrow \infty}{\longrightarrow} 0.
  \]
  For suitable $v'_{\theta}$, this is a consequence of the bound
  \[
    \mathbb{E} [ |c_n - c_n^{(N)} |^2] \leq C_{\theta} n/N^{2-\theta}
  \]
  given by Lemma \ref{lem:secdiff}. 
  
\end{proof}
\begin{remark}
In the supercritical phase, we compare $c_n $ with $c_n^{(N)}$ by using the $L^2$ norm, which is expected not to be optimal, since the $L^2$ norm of $c_n$ does not gives
its correct order of magnitude. Hence, we expect that the exponent $(3 - 2 \sqrt{\theta} )/(2-\theta)$ can be improved.
\end{remark} 

\section{Sharpness of the tightness}\label{sec:sharptight}
We have seen that $(c_n (\log n)^{-v_{\theta}} /w_n)_{n \geq 1}$ is a tight family for random variables for $\theta > 0$ for some $v_{\theta}>0$ depending only on $\theta$, which 
 can be taken equal to $0$ for $\theta \in (0,1]$. 
 Moreover informally,  $c_n$ has order at most $w_n$, up to a logarithmic factor. It is natural to ask if this bound is optimal, or if $c_n$ has a smaller order of magnitude. 
 In the following section, we show that $c_n$ is no smaller than $w_n$ in order of magnitude (i.e.\ that $w_n/c_n$ is tight)
 when $\theta \in (0,1]$.  
\begin{theorem}\label{thm:antitank1}
For all $\theta \in (0,1]$, and all $u \in \mathbb{C}$ such that $|u| = 1$, one has
$$\sup_{n \geq 1}  \mathbb{P} [ |\Re( u c_n)| / w_n  \leq \delta ] \underset{ \delta \rightarrow 0}{\longrightarrow} 0,$$
\end{theorem} 
\begin{proof}
  Since the law of $c_n$ is rotational invariant in distribution, we can assume $u = 1$. 
  We have the following equality: 
  \[
    c_n = c_{n, \lfloor n/2 \rfloor} +  \sum_{q=0}^{\lceil \tfrac{n}{2}\rceil-1}  \frac{\mathcal{N}_{n-q}\theta^{1/2}}{\sqrt{n-q}} c_q,
  \]
  and then, since $(\mathcal{N}_{n-q})_{0 \leq q < n/2}$ are independent of $c_{n, \lfloor n/2 \rfloor}$ and $(c_q)_{0 \leq q < n/2}$, 
  \[
    c_n \lawequals c_{n, \lfloor n/2 \rfloor} + 
    \mathcal{N}\biggl( \sum_{q=0}^{\lceil \tfrac{n}{2}\rceil-1} \frac{ |c_q|^2 \theta}{n-q} \biggr)^{1/2} 
  \]
  where $\mathcal{N}$ is a standard complex Gaussian independent of the other variables which are involved in the formula. 
  Hence, by the fact that the law of $\Re(\mathcal{N})$ has a bounded density with respect to the Lebesgue measure: 
  \[
    \mathbb{P} [ |\Re(c_n)|  / w_n  \leq \delta   ~\vert~ \Gfilt_{\lfloor n/2 \rfloor}]
    \leq C\delta w_n
    \biggl( \sum_{q=0}^{\lceil \tfrac{n}{2}\rceil-1} \frac{ |c_q|^2 \theta}{n-q} \biggr)^{-1/2}
    ,  
  \]
  for some constant $C > 0$, which implies 
  \begin{equation} \label{eq:antitank1}
    \mathbb{P} [ |\Re(c_n)|  / w_n  \leq \delta ]
    \leq  \mathbb{E} 
    \biggl[ \min \biggl(
      1,  
      C\delta w_n
      \biggl( \sum_{q=0}^{\lceil \tfrac{n}{2}\rceil-1} \frac{ |c_q|^2 \theta}{n-q} \biggr)^{-1/2}    
      \biggr) 
    \biggr].
  \end{equation}
  Recalling \eqref{mthetan} there is a constant $c_\theta > 0$ so that for all $n>2$ and with $n_0 = \lceil \tfrac{n}{2}\rceil-1,$ 
  \[
    \sum_{q=0}^{n_0}
    \frac{ |c_q|^2 \theta}{n-q}   
    \geq \frac{\theta}{n}  \sum_{q=0}^{n_0} |c_q|^2
    \geq c_\theta (\sqrt{\log n})^{-\one[\theta=1]} n^\theta \mathcal{M}_{\theta,n_0}.
  \]
  The expression $\mathcal{M}_{\theta,n_0}$ is an approximation to the total mass.  By Lemma \ref{gmc-conv},
  this converges in probability as $n\to \infty$ to a multiple of $\mathcal{M}_\theta.$ In particular
  \[
    \biggl\{
      \biggl( \sum_{q=0}^{\lceil \tfrac{n}{2}\rceil-1} \frac{ |c_q|^2 \theta}{n-q} \biggr)^{-1/2}
      w_n  
      :
      n \in \N
    \biggr\}
  \]
  is tight (recall when $\theta \in (0,1), w_n =n^{\theta-1}$ and when $\theta =1,$ $w_n = ( \log(1+n) )^{-1/4}$).
  Using \eqref{eq:antitank1}
  \[
    \mathbb{P} [ |\Re(c_n)|  / w_n  \leq \delta ]
    \leq C \delta^{1/2} + 
    \mathbb{P} \biggl( \biggl( \sum_{q=0}^{\lceil \tfrac{n}{2}\rceil-1} \frac{ |c_q|^2 \theta}{n-q} \biggr)^{-1/2}
    w_n  \geq \delta^{-1/2} \biggr).
  \]
  The tightness property above shows that 
  $$ \sup_{n \geq 1} \mathbb{P} [ |\Re(c_n)|  / w_n  \leq \delta ] \underset{\delta \rightarrow 0}{\longrightarrow} 0.$$
  which completes the proof.
\end{proof}

We deduce a similar result for the secular coefficients $c_n^{(N)}$ when $N$ is sufficiently large with respect to $n$:

\begin{theorem}\label{thm:antitank2}
Let $\theta \in (0,1]$, $u \in \mathbb{C}$ on the unit circle, and let $\varphi$ be any function from $\mathbb{N}$ to $\mathbb{R}$ such that 
$\varphi(n) \geq 2$ for all $n \geq 1$, $\varphi(n)$ tends to infinity with $n$, and for $\theta = 1$, $\varphi(n)/ ( \sqrt{\log n}(\log \log n))$ tends to infinity with $n$. 
Then, 
$$\sup_{n \geq 1, N \geq n \varphi(n)}  \mathbb{P} [ |\Re( u c^{(N)}_n)| / w_n  \leq \delta ] \underset{\delta \rightarrow 0}{\longrightarrow} 0.$$

\end{theorem} 
\begin{proof}
  We assume $u = 1$. 
  We know that for all $\epsilon > 0$, there exists $\delta > 0$ depending only on $\epsilon$ and $\theta$, such that for  $n \geq 1$, 
  $$\mathbb{P} [ |\Re(  c_n)| / w_n  \leq 2 \delta ] \leq \epsilon.$$
  Now, 
  $$\mathbb{P} [ |\Re(  c^{(N)}_n)| / w_n  \leq  \delta ] \leq \mathbb{P} [ |\Re(  c_n)| / w_n  \leq 2 \delta ]  + \mathbb{P} [ |c_n- c^{(N)}_n |/w_n \geq \delta],$$
  and then, using Lemma \ref{lem:secdiff} and Markov's inequality,  
  $$\mathbb{P} [ |\Re(  c^{(N)}_n)| / w_n  \leq  \delta ] \leq \epsilon +  \delta^{-2} C_{\theta} w_n^{-2} n^{\theta}/N$$
  for $\theta \in (0,1)$ and 
  $$\mathbb{P} [ |\Re(  c^{(N)}_n)| / w_n  \leq  \delta ] \leq \epsilon +  \delta^{-2} C_{1} w_n^{-2} n \log (N/n) /N$$
  for $\theta = 1$. 
  Hence, 
  $$\mathbb{P} [ |\Re(  c^{(N)}_n)| / w_n  \leq  \delta ] \leq \epsilon +  \delta^{-2} C_{\theta} n/N \leq \epsilon + \delta^{-2} C_{\theta} /\varphi(n)$$
  for $\theta \in (0,1)$ and $N \geq n \varphi(n)$, 
  and 
  \begin{align*}
    \mathbb{P} [ |\Re(  c^{(N)}_n)| / w_n  \leq  \delta ]  & \leq \epsilon +  \delta^{-2} C_{1}  (\log (1+n))^{1/2} \log (N/n) (n/N) \\ & \leq \epsilon +2  \delta^{-2} 
    C_1 (\log (1+n))^{1/2}(\log (\varphi(n)))/\varphi(n)
  \end{align*} 
  for $\theta =1 $ and $N \geq n \varphi(n)$. The factor $2$ comes from the fact that $\varphi(n) \geq 2$ by assumption and $(\log a) / a \leq 2 (\log b)/ b$ for $a \geq b \geq 2$. 
  For $\theta \in (0,1)$, we have 
  $$\delta^{-2} C_{\theta} /\varphi(n) \underset{n \rightarrow \infty}{\longrightarrow} 0$$
  since $\varphi(n) \rightarrow \infty$. 
  For $\theta = 1$, we have 
  $$\min(\varphi(n), 2 + \log n) / ( \sqrt{\log n}(\log \log n)) \underset{n \rightarrow \infty}{\longrightarrow} \infty,$$
  and then 
  \begin{align*}
    &  (\log (1+n))^{1/2}(\log (\varphi(n)))/\varphi(n) 
    \\ & \leq 2  (\log (1+n))^{1/2} \log (\min(\varphi(n), 2 + \log n)) / \min(\varphi(n), 2 + \log n), 
    \\ & \leq 2 (\log (1+n))^{1/2} \log  (2 + \log n) / \min(\varphi(n), 2 + \log n) 
  \end{align*}
  tends to zero when $n$ goes to infinity, 
  which implies 
  $$ 2  \delta^{-2} 
  C_1 (\log (1+n))^{1/2}(\log (\varphi(n)))/\varphi(n) \underset{n \rightarrow \infty}{\longrightarrow} 0.$$
  Since we take $\delta$ depending only on $\epsilon$ and $\theta$, we deduce that there exists $n(\epsilon, \theta) \geq 1$ such that for all $n \geq n (\epsilon, \theta)$, 
  $$\delta^{-2} C_{\theta} /\varphi(n)  \leq \epsilon$$
  if $\theta \in (0,1)$, and 
  $$2  \delta^{-2} 
  C_1 (\log (1+n))^{1/2}(\log (\varphi(n)))/\varphi(n) \leq \epsilon$$
  if $\theta = 1$. This gives, in any case, 
  $$\mathbb{P} [  |\Re(  c^{(N)}_n)| / w_n  \leq  \delta ] \leq 2 \epsilon$$
  for $n \geq n(\epsilon, \theta)$, $N \geq \varphi(n)$. 
  Now, for $N \geq n \geq 1$,  $c^{(N)}_n$ is an elementary symmetric function of random points on the unit circle whose joint distribution 
  is absolutely continuous with respect to the distribution of $N$ i.i.d., uniform points on the unit circle.  For $N = n$, $|c^{(N)}_n|$ is equal to $1$ and then different from $0$. 
  For $N > n$,  if we fix $N-1$ of the $N$ points on the circle, we have that  $c^{(N)}_n$ is an affine function of the last point, and then the conditional probability
  that $c^{(N)}_n = 0$ vanishes as soon as one of the two coefficients of this affine function is non-zero. Taking the constant coefficient, which is the $n$-th symmetric function of the $N-1$ points which have been fixed, we deduce, by induction on $N$, that $\mathbb{P}[ c^{(N)}_n = 0] = 0$ for all $N \geq n \geq 1$. 
  Since the law of  $c^{(N)}_n$ is rotationally invariant, we have that $\mathbb{P}[ \Re(c^{(N)}_n) = 0  |  \, |c^{(N)}_n| ] = 0$ when $|c^{(N)}_n| \neq 0$, and 
  then $\mathbb{P} [  \Re(c^{(N)}_n) = 0 ] = 0$ since $|c^{(N)}_n|$ is almost surely different from zero. 
  Moreover, we know that $\Re(c_n) \neq 0$ almost surely: since $c^{(N)}_n$ converges to $c_n$ in law for $n$ fixed and $N \rightarrow \infty$, 
  we have that for each $n \geq 1$, $(w_n / |\Re(  c^{(N)}_n)|)_{N \geq n}$ is a tight family of real-valued random variables. 

  Hence, for each $n \geq 1$, there exists $\delta_n > 0$, depending only on $n$, $\epsilon$ and $\theta$, such that 
  $$\sup_{N \geq n} \mathbb{P} [  |\Re(  c^{(N)}_n)| / w_n  \leq  \delta_n ] \leq 2 \epsilon.$$
  Let us define 
  $$\delta_0 := \min \left( \delta, \min_{1 \leq n < n(\epsilon, \theta)} \delta_n \right).$$
  We have that $\delta_0 > 0$ depends only on $\epsilon$ and $\theta$, and that 
  $$\sup_{n \geq 1, N \geq n \varphi(n)} \mathbb{P} [  |\Re(  c^{(N)}_n)| / w_n  \leq  \delta_0] \leq 2 \epsilon,$$
  which implies 
  $$\underset{\delta \rightarrow 0}{\lim \sup} \sup_{n \geq 1, N \geq n \varphi(n)} \mathbb{P} [  |\Re(  c^{(N)}_n)| / w_n  \leq  \delta] \leq 2 \epsilon.$$
  We are done, by letting $\epsilon \rightarrow 0$. 
\end{proof}

%
%

\appendix

\section{Tauberian theory}

Recall that a function $L : [0,\infty) \to \R$ is called \emph{slowly varying} if it is measurable, eventually positive and satisfies
  \[
    L(\lambda u)/L(u) \to 1
    \quad 
    \text{as}
    \quad
    u\to \infty
    \quad
    \text{for all } \lambda > 0.
  \]
  The principal examples of such functions are logarithms, iterated logarithms, and powers thereof.  These functions have many useful analytic properties (see \cite[{Section IV.2}]{Kor04} for details).
  Such functions appear frequently in Tauberian theory, and in particular in the following:
  \begin{theorem}[{Karamata--Hardy--Littlewood Tauberian Theorem \cite[{IV.1.1}]{Kor04}}]
    Let $\sum_{k=0}^{\infty}a_{k}z^{k}$ converge for $|z|<1$. Suppose that for some number $\theta \geq 0$ and some slowly varying function $L$, we have the limit from below
    \begin{equation}
      \lim_{z \to 1^{-}}L(\tfrac{1}{1-z})(1-z)^{\theta}\sum_{k=0}^{\infty}a_{k}z^{k} = M
    \end{equation}
    Then subject to the condition
    \begin{equation}
      L(k)ka_{k} \geq -Ck^{\theta}, \qquad k \geq 1
    \end{equation}
    it follows that
    \begin{equation}
      \lim_{n \to \infty}\frac{L(n)\Gamma(\theta+1)}{n^{\theta}}\sum_{k=0}^{n}a_{k} = M.
    \end{equation}
    \label{th:tauber}
  \end{theorem}
  We need to justify this in a probabilistic situation:
  \begin{theorem}
    \label{th:tauberprob}
    Fix $\theta>0$ and a slowly varying function $L,$ and suppose for simplicity that $\{a_{k}\}_{k=0}^{\infty}$ are a family of non-negative random variables such that the series $\sum_{k=0}^{\infty}a_{k}z^{k}$ converges almost surely for $|z|<1$. Suppose further that we have the convergence in probability
    \begin{equation}
      L(\tfrac{1}{1-z})(1-z)^{\theta}\sum_{k=0}^{\infty}a_{k}z^{k} 
      \Prto[z][1_{-}]
      M \label{genfunconvp}
    \end{equation}
    with $M < \infty$ almost surely. Then it follows that
    \begin{equation}
      \frac{L(n) \Gamma(\theta+1)}{n^{\theta}}\sum_{k=0}^{n}a_{k} 
      \Prto[n]
      M.
    \end{equation}
  \end{theorem}

  \begin{proof}
    Let us define
    \begin{equation}
      X_{n} = L\big(\tfrac{1}{1-e^{-1/n}}\bigr)(1-e^{-1/n})^{\theta}\sum_{k=0}^{\infty}a_{k}g(e^{-k/n})
    \end{equation}
    where
    \begin{equation}
      g(x) = \begin{cases} 0,& \quad 0 \leq x < 1/e,\\ 1,& \quad 1/e \leq x \leq 1. \end{cases}
    \end{equation}
    If we can show that $X_{n} \Prto[n] M/\Gamma(\theta+1)$ we will have proved the result. To do this we construct polynomial approximations of $g$ and bound from above and below. It is important to note that the construction of the polynomials is deterministic and does not depend on the random coefficients $a_{k}$.

    For any polynomial $P$ without constant term, it is a direct computation using \eqref{genfunconvp}, using additivity of convergence in probability, and using the definition of slow variation, that
    \begin{equation}
      L\big(\tfrac{1}{1-e^{-1/n}}\bigr)
      (1-e^{-1/n})^{\theta}\sum_{k=0}^{\infty}a_{k}P(e^{-k/n}) 
      \Prto[n]
      \frac{M}{\Gamma(\theta)}\,\int_{0}^{\infty}t^{\theta-1}P(e^{-t})\,dt.
      \label{convprobpoly}
    \end{equation}
    We claim that for any $\epsilon>0$ we can find polynomials $P^{\pm}_{\epsilon}(x)$ on $[0,1]$ without constant term, such that $g(x) \leq P^{+}_{\epsilon}(x)$ (and $P^{-}_{\epsilon}(x) \leq g(x)$) and that $P^{+}_{\epsilon}(x)-g(x)$ (and $g(x) - P^{-}_{\epsilon}(x)$) are both small. Since the coefficients $a_{k}$ are non-negative, this allows us to sandwich $X_{n}$ via
    \begin{equation}
      X^{-}_{n,\epsilon} \leq X_{n} \leq X^{+}_{n,\epsilon} \label{sandwich}
    \end{equation}
    where
    \begin{equation}
      X^{\pm}_{n,\epsilon} := 
      L\big(\tfrac{1}{1-e^{-1/n}}\bigr)
      (1-e^{-1/n})^{\theta}\sum_{k=0}^{\infty}a_{k}P^{\pm}_{\epsilon}(e^{-k/n}).
    \end{equation}
    Now by \eqref{convprobpoly} we have $X^{\pm}_{n,\epsilon} \overset{p}{\longrightarrow} X^{\pm}_{\epsilon}$ where
    \begin{equation}
      \begin{split}
	X^{\pm}_{\epsilon} &= \frac{M}{\Gamma(\theta)}\,\left(\int_{0}^{\infty}t^{\theta-1}g(e^{-t})\,dt \pm c^{\pm}(\epsilon)\right)\\
	&=\frac{M}{\Gamma(\theta+1)}\left(1\pm\theta c^{\pm}(\epsilon)\right)
      \end{split}
    \end{equation}
    and
    \begin{equation}
      c^{\pm}(\epsilon) = \int_{0}^{\infty}t^{\theta-1}|P^{\pm}_{\epsilon}(e^{-t})-g(e^{-t})|\,dt.
    \end{equation}
    We claim that $P^{\pm}_{\epsilon}(x)$ can be chosen so that $0 \leq c^{\pm}(\epsilon) \leq c_{\theta}\epsilon$ for some absolute constant $c_{\theta}$ and so $\lim_{\epsilon \to 0}c^{\pm}(\epsilon)=0$. This is enough to show convergence in probability of $X_{n}$ to the limit $X = M/\Gamma(\theta+1)$, as we now demonstrate. By \eqref{sandwich}, we have for any $\delta>0$
    \begin{equation}
      \begin{split}
	\mathbb{P}(|X_{n}-X| > \delta) &\leq \mathbb{P}(|X^{+}_{n,\epsilon}-X|>\delta) + \mathbb{P}(|X^{-}_{n,\epsilon}-X|>\delta)\\
	&\leq \mathbb{P}(|X^{+}_{n,\epsilon}-X^{+}_{\epsilon}|+|X^{+}_{\epsilon}-X| > \delta)\\
	& + \mathbb{P}(|X^{-}_{n,\epsilon}-X^{-}_{\epsilon}|+|X^{-}_{\epsilon}-X| > \delta)\\
	&\leq \mathbb{P}(|X^{+}_{n,\epsilon}-X^{+}_{\epsilon}| > \delta/2) + \mathbb{P}(|X^{+}_{\epsilon}-X| > \delta/2) \\
	&+ \mathbb{P}(|X^{-}_{n,\epsilon}-X^{-}_{\epsilon}| > \delta/2) + \mathbb{P}(|X^{-}_{\epsilon}-X| > \delta/2).
      \end{split}
    \end{equation}
    Since $X^{\pm}_{n,\epsilon} \Prto[n] X^{\pm}_{\epsilon}$, two of the terms in the final inequality above vanish in the limit $n \to \infty$ for any fixed $\epsilon$. For the remaining terms, note that by definition
    \begin{equation}
      |X^{\pm}_{\epsilon}-X| = c^{\pm}(\epsilon)M/\Gamma(\theta)
    \end{equation}
    so that in the limit $\epsilon \to 0$ we have, using the almost sure finiteness of $M$,
    \begin{equation}
      \mathbb{P}(|X^{\pm}_{\epsilon}-X| > \delta/2) = \mathbb{P}(M > \delta\Gamma(\theta)/(2c^{\pm}(\epsilon)) \to 0.
    \end{equation}

    The existence of the approximating polynomials is classical, but for completeness we give the argument. We begin by sandwiching $g(x)$ by the continuous function $h^{\pm}_{\epsilon}(x)$ where $h^{\pm}_{\epsilon}(x)$ is equal to $g(x)$ outside $[1/e\mp \epsilon,1/e]$ and is linear inside the interval. Then by the Weierstrass approximation theorem we get polynomials $P^{\pm}_{\epsilon}(x)/x$ such that
    \begin{equation}
      \bigg{|}\frac{h^{\pm}_{\epsilon}(x)}{x}\pm \epsilon-\frac{P^{\pm}_{\epsilon}(x)}{x}\bigg{|} \leq \epsilon, \qquad 0 \leq x \leq 1. \label{weierstrass}
    \end{equation}
    Then $P^{+}_{\epsilon}(x) \geq h^{+}(x) \geq g(x)$ and $P^{-}_{\epsilon}(x) \leq h^{-}(x) \leq g(x)$. Now we can bound $c^{\pm}(\epsilon)$ using
    \begin{equation}
      c^{\pm}(\epsilon) \leq \int_{0}^{\infty}t^{\theta-1}|P^{\pm}_{\epsilon}(e^{-t})-h^{\pm}(e^{-t})|\,dt+\int_{0}^{\infty}t^{\theta-1}|h^{\pm}(e^{-t})-g(e^{-t})|\,dt
    \end{equation}
    The substitution $x=e^{-t}$ shows that
    \begin{equation}
      \begin{split}
	c^{\pm}(\epsilon) \leq &\int_{0}^{1}(-\log(x))^{\theta-1}\frac{|P^{\pm}_{\epsilon}(x)-h^{\pm}(x)|}{x}\,dx\\
	&+\bigg{|}\int_{1/e \pm \epsilon}^{1/e}(-\log(x))^{\theta-1}\frac{|h^{\pm}(x)-g(x)|}{x}\,dx\bigg{|} \label{cepsbound}
      \end{split}
    \end{equation}
    The integrand in the second term of \eqref{cepsbound} is uniformly bounded, so the integral is bounded by a constant times the length of the integration interval, which is $\epsilon$. By \eqref{weierstrass} the first term in \eqref{cepsbound} is bounded by
    \begin{equation}
      2\epsilon \int_{0}^{1}(-\log(x))^{\theta-1}\,dx = 2\epsilon \Gamma(\theta),
    \end{equation}
    and so $|c^{\pm}(\epsilon)| \leq c_{\theta} \epsilon$ as desired.
  \end{proof}

  \begin{lemma}
    \label{le:l2bound}
    Let $Z_{1} = X_{1}+iY_{1}$ and $Z_{2} = X_{2}+iY_{2}$ be complex valued random variables with characteristic functions
    \begin{equation}
      \begin{split}
	\Phi_{Z_1}(s,t) &= \mathbb{E}\left(e^{isX_{1}+itY_{1}}\right)\\
	\Phi_{Z_2}(s,t) &= \mathbb{E}\left(e^{isX_{2}+itY_{2}}\right)
      \end{split}
    \end{equation}
    Then 
    \begin{equation}
      |\Phi_{Z_1}(s,t)-\Phi_{Z_2}(s,t)| \leq (|s|+|t|)\sqrt{\mathbb{E}(|Z_1-Z_2|^{2})}.
    \end{equation}
  \end{lemma}

  \begin{proof}
    This follows from the standard inequality $|e^{ix}-1| \leq |x|$ valid for any real number $x$. We have
    \begin{align*}
      &|\Phi_{Z_1}(s,t)-\Phi_{Z_2}(s,t)| \leq \mathbb{E}\left(|e^{isX_{1}+itY_{1}}-e^{isX_{2}+iY_{2}}|\right)\\
      &= \mathbb{E}\left(|e^{is(X_{1}-X_{2})+it(Y_{1}-Y_{2})}-1|\right)\\
      &\leq \mathbb{E}\left(|s(X_{1}-X_{2})+t(Y_{1}-Y_{2})|\right)\\
      &\leq |s|\mathbb{E}\left(|X_{1}-X_{2}|\right)+|t|\mathbb{E}\left(|Y_{1}-Y_{2}|\right)\\
      &\leq |s|\sqrt{\mathbb{E}\left(|X_{1}-X_{2}|^{2}\right)}+|t|\sqrt{\mathbb{E}\left(|Y_{1}-Y_{2}|^{2}\right)}\\
      &\leq (|s|+|t|)\sqrt{\mathbb{E}\left(|Z_{1}-Z_{2}|^{2}\right)}.
    \end{align*}
  \end{proof}
  \section{Binomial sums}
  \begin{lemma}
    \label{infsumk}
    For any $0 < \theta < \frac{1}{2}$, we have
    \begin{equation}
      \sum_{k=0}^{\infty}\binom{k+\theta-1}{\theta-1}^{2} = \frac{\Gamma(1-2\theta)}{\Gamma(1-\theta)^{2}} \label{rhsident}.
    \end{equation}
  \end{lemma}
  \begin{proof}
    We have the generating function
    \begin{equation}
      \sum_{k=0}^{\infty}\binom{k+\theta-1}{\theta-1}z^{k} = (1-z)^{-\theta}, \qquad |z| < 1
    \end{equation}
    so that by Parseval's theorem
    \begin{equation}
      \sum_{k=0}^{\infty}\binom{k+\theta-1}{\theta-1}^{2} = \frac{1}{2\pi}\int_{0}^{2\pi}|1-e^{i\nu_1}|^{-2\theta}\,d\nu_1. \label{morrisintk2}
    \end{equation}
Now the right-hand side of \eqref{morrisintk2} coincides with that of \eqref{rhsident} by taking $k=2$ in the Morris integral \eqref{eq:morris} and using rotational invariance in the $\nu_{2}$ coordinate.
   \end{proof}

  \begin{lemma}
    \label{sumlemma}
    Let $a$ and $b$ be non-negative integers such that $b\geq a$, $c$ and $d$ real numbers such that $\binom{q+c}{d}$ is finite and well-defined on $a \leq q \leq b$. Then
    \begin{equation}
      \sum_{q=a}^{b}\binom{q+c}{d} = \binom{b+c+1}{d+1}-\binom{a+c}{d+1}. \label{sumlemmaeq}
    \end{equation}
  \end{lemma}
  \begin{proof}
    This follows from summing the binomial coefficient identity
    \begin{equation}
      \binom{q+c}{d} = \binom{q+c+1}{d+1} - \binom{q+c}{d+1},
    \end{equation}
    from $q=a$ to $q=b$, so that the sum telescopes and gives the right-hand side of \eqref{sumlemmaeq}.
  \end{proof}
 
  \bibliographystyle{alpha}
  \bibliography{convergence-zero-secular}

\newcommand{\etalchar}[1]{$^{#1}$}
\begin{thebibliography}{KRRGR18}

\bibitem[ABB17]{ABB17}
Louis-Pierre Arguin, David Belius, and Paul Bourgade.
\newblock Maximum of the characteristic polynomial of random unitary matrices.
\newblock {\em Comm. Math. Phys.}, 349(2):703--751, 2017.

\bibitem[ABB{\etalchar{+}}19]{ABBRS}
Louis-Pierre Arguin, David Belius, Paul Bourgade, Maksym Radziwi\l\l, and
  Kannan Soundararajan.
\newblock Maximum of the {R}iemann zeta function on a short interval of the
  critical line.
\newblock {\em Comm. Pure Appl. Math.}, 72(3):500--535, 2019.

\bibitem[ABR20]{ABR20}
Louis-Pierre Arguin, Paul Bourgade, and Maksym Radziwiłł.
\newblock The {F}yodorov-{H}iary-{K}eating conjecture. {I}.
\newblock {\em arXiv:2007.00988}, 2020.

\bibitem[ABT92]{ArratiaTavare2}
Richard Arratia, Andrew~D. Barbour, and Simon Tavar\'{e}.
\newblock Poisson process approximations for the {E}wens sampling formula.
\newblock {\em Ann. Appl. Probab.}, 2(3):519--535, 1992.

\bibitem[ABT03]{ABT03}
Richard Arratia, Andrew~D. Barbour, and Simon Tavar\'e.
\newblock {\em Logarithmic Combinatorial Structures: a Probabilistic Approach}.
\newblock European Mathematical Society, 2003.

\bibitem[AJJ20]{AruJegoJunnila}
Juhan {Aru}, Antoine {Jego}, and Janne {Junnila}.
\newblock {Density of imaginary multiplicative chaos via Malliavin calculus}.
\newblock {\em arXiv e-prints}, page arXiv:2008.11768, August 2020.

\bibitem[BJ10]{BarralJin}
Julien Barral and Xiong Jin.
\newblock Multifractal analysis of complex random cascades.
\newblock {\em Comm. Math. Phys.}, 297(1):129--168, 2010.

\bibitem[BJM10a]{BarralJinMandelbrot1}
Julien Barral, Xiong Jin, and Beno\^{\i}t Mandelbrot.
\newblock Convergence of complex multiplicative cascades.
\newblock {\em Ann. Appl. Probab.}, 20(4):1219--1252, 2010.

\bibitem[BJM10b]{BarralJinMandelbrot2}
Julien Barral, Xiong Jin, and Beno\^{\i}t Mandelbrot.
\newblock Uniform convergence for complex {$[0,1]$}-martingales.
\newblock {\em Ann. Appl. Probab.}, 20(4):1205--1218, 2010.

\bibitem[BP03]{BeckPixton}
Matthias Beck and Dennis Pixton.
\newblock The {E}hrhart polynomial of the {B}irkhoff polytope.
\newblock {\em Discrete Comput. Geom.}, 30(4):623--637, 2003.

\bibitem[BWW18]{BWW}
Nathana\"{e}l Berestycki, Christian Webb, and Mo~Dick Wong.
\newblock Random {H}ermitian matrices and {G}aussian multiplicative chaos.
\newblock {\em Probab. Theory Related Fields}, 172(1-2):103--189, 2018.

\bibitem[CG06]{ConreyGamburd}
Brian Conrey and Alex Gamburd.
\newblock Pseudomoments of the {R}iemann zeta-function and pseudomagic squares.
\newblock {\em J. Number Theory}, 117(2):263--278, 2006.

\bibitem[CK15]{CK15}
Tom Claeys and Igor Krasovsky.
\newblock Toeplitz determinants with merging singularities.
\newblock {\em Duke Math. J.}, 164(15):2897--2987, 2015.

\bibitem[CM09]{CanfieldMcKay}
E.~Rodney Canfield and Brendan~D. McKay.
\newblock The asymptotic volume of the {B}irkhoff polytope.
\newblock {\em Online J. Anal. Comb.}, (4):4, 2009.

\bibitem[CMN18]{CMN}
Reda Chhaibi, Thomas Madaule, and Joseph Najnudel.
\newblock On the maximum of the {${\rm C}\beta {\rm E}$} field.
\newblock {\em Duke Math. J.}, 167(12):2243--2345, 2018.

\bibitem[CN19]{ChhaibiNajnudel}
Reda {Chhaibi} and Joseph {Najnudel}.
\newblock {On the circle, $GMC^\gamma = \varprojlim C\beta E_n$ for $\gamma =
  \sqrt{\frac{2}{\beta}}, $ $( \gamma \leq 1 )$}.
\newblock {\em arXiv:1904.00578}, 2019.

\bibitem[DG06]{DiaconisGamburd}
Persi Diaconis and Alex Gamburd.
\newblock Random matrices, magic squares and matching polynomials.
\newblock {\em Electron. J. Combin.}, 11(2):Research Paper 2, 26, 2004/06.

\bibitem[DLLY09]{deLoeraLiuYoshida}
Jesus~A. De~Loera, Fu~Liu, and Ruriko Yoshida.
\newblock A generating function for all semi-magic squares and the volume of
  the {B}irkhoff polytope.
\newblock {\em J. Algebraic Combin.}, 30(1):113--139, 2009.

\bibitem[DRSV17]{DRSV17}
Bertrand Duplantier, R\'{e}mi Rhodes, Scott Sheffield, and Vincent Vargas.
\newblock Log-correlated {G}aussian fields: an overview.
\newblock In {\em Geometry, analysis and probability}, volume 310 of {\em
  Progr. Math.}, pages 191--216. Birkh\"{a}user/Springer, Cham, 2017.

\bibitem[DRZ17]{DingRoyZeitouni}
Jian Ding, Rishideep Roy, and Ofer Zeitouni.
\newblock Convergence of the centered maximum of log-correlated {G}aussian
  fields.
\newblock {\em Ann. Probab.}, 45(6A):3886--3928, 2017.

\bibitem[DS94]{DS94}
Persi Diaconis and Mehrdad Shahshahani.
\newblock On the eigenvalues of random matrices.
\newblock {\em J. Appl. Probab.}, 31A:49--62, 1994.
\newblock Studies in applied probability.

\bibitem[FB08]{FyodorovBouchaud}
Yan~V. Fyodorov and Jean-Philippe Bouchaud.
\newblock Freezing and extreme-value statistics in a random energy model with
  logarithmically correlated potential.
\newblock {\em J. Phys. A}, 41(37):372001, 12, 2008.

\bibitem[FHK12]{FHK12}
Yan~V. Fyodorov, Ghaith~A. Hiary, and Jonathan~P. Keating.
\newblock Freezing {T}ransition, {C}haracteristic {P}olynomials of {R}andom
  {M}atrices, and the {R}iemann {Z}eta-{F}unction.
\newblock {\em Phys. Rev. Lett.}, 108:170601, 2012.

\bibitem[FK14]{FK14}
Yan~V. Fyodorov and Jonathan~P. Keating.
\newblock Freezing transitions and extreme values: random matrix theory, and
  disordered landscapes.
\newblock {\em Philos. Trans. R. Soc. Lond. Ser. A Math. Phys. Eng. Sci.},
  372(2007):20120503, 32, 2014.

\bibitem[Har20]{harper}
Adam~J. Harper.
\newblock Moments of random multiplicative functions, {I}: low moments, better
  than squareroot cancellation, and critical multiplicative chaos.
\newblock {\em Forum Math. Pi}, 8:e1, 95, 2020.

\bibitem[HH80]{HH80}
Peter Hall and Christopher~C. Heyde.
\newblock {\em Martingale Limit Theory and Its Application}.
\newblock Academic Press, 1980.

\bibitem[HK15]{HartungKlimovsky1}
Lisa Hartung and Anton Klimovsky.
\newblock The glassy phase of the complex branching {B}rownian motion energy
  model.
\newblock {\em Electron. Commun. Probab.}, 20:no. 78, 15, 2015.

\bibitem[HK18]{HartungKlimovsky2}
Lisa Hartung and Anton Klimovsky.
\newblock The phase diagram of the complex branching {B}rownian motion energy
  model.
\newblock {\em Electron. J. Probab.}, 23:Paper No. 127, 27, 2018.

\bibitem[HKO01]{HKO01}
Christopher~P. Hughes, Jonathan~P. Keating, and Neil O'Connell.
\newblock On the characteristic polynomial of a random unitary matrix.
\newblock {\em Comm. Math. Phys.}, 220(2):429--451, 2001.

\bibitem[HKS{\etalchar{+}}96]{Haake}
Fritz Haake, Marek Kus, Hans-J{\"u}rgen Sommers, Henning Schomerus, and Karol
  Zyczkowski.
\newblock Secular determinants of random unitary matrices.
\newblock {\em Journal of Physics A: Mathematical and General}, 29(13):3641,
  1996.

\bibitem[JM15]{JM15}
Tiefeng Jiang and Sho Matsumoto.
\newblock Moments of traces of circular beta-ensembles.
\newblock {\em Ann. Probab.}, 43(6):3279--3336, 2015.

\bibitem[JS17]{JunnilaSaksman}
Janne Junnila and Eero Saksman.
\newblock Uniqueness of critical {G}aussian chaos.
\newblock {\em Electron. J. Probab.}, 22:Paper No. 11, 31, 2017.

\bibitem[JSV19]{JunnilaSaksmanViitasaari}
Janne Junnila, Eero Saksman, and Lauri Viitasaari.
\newblock On the regularity of complex multiplicative chaos.
\newblock 2019.

\bibitem[JSW18]{JunnilaSaksmanWebb}
Janne {Junnila}, Eero {Saksman}, and Christian {Webb}.
\newblock {Imaginary multiplicative chaos: Moments, regularity and connections
  to the Ising model}.
\newblock {\em arXiv e-prints}, page arXiv:1806.02118, June 2018.

\bibitem[KN04]{KillipNenciu}
Rowan Killip and Irina Nenciu.
\newblock Matrix models for circular ensembles.
\newblock {\em Int. Math. Res. Not.}, (50):2665--2701, 2004.

\bibitem[Kor04]{Kor04}
Jacob Korevaar.
\newblock {\em Tauberian Theory: A Century of Developments}.
\newblock Springer, 2004.

\bibitem[KRRGR18]{KRRGR}
Jonathan~P. Keating, Brad Rodgers, Edva Roditty-Gershon, and Zeev Rudnick.
\newblock Sums of divisor functions in {$\Bbb{F}_q[t]$} and matrix integrals.
\newblock {\em Math. Z.}, 288(1-2):167--198, 2018.

\bibitem[{Lac}20]{Lacoin}
Hubert {Lacoin}.
\newblock {A universality result for subcritical Complex Gaussian
  Multiplicative Chaos}.
\newblock {\em arXiv e-prints}, page arXiv:2003.14024, March 2020.

\bibitem[Led99]{Ledoux}
Michel Ledoux.
\newblock Concentration of measure and logarithmic {S}obolev inequalities.
\newblock In {\em S\'{e}minaire de {P}robabilit\'{e}s, {XXXIII}}, volume 1709
  of {\em Lecture Notes in Math.}, pages 120--216. Springer, Berlin, 1999.

\bibitem[LOS18]{LOS18}
Gaultier Lambert, Dmitry Ostrovsky, and Nick Simm.
\newblock Subcritical multiplicative chaos for regularized counting statistics
  from random matrix theory.
\newblock {\em Comm. Math. Phys.}, 360(1):1--54, 2018.

\bibitem[LRV15]{LacoinRhodesVargas}
Hubert Lacoin, R\'{e}mi Rhodes, and Vincent Vargas.
\newblock Complex {G}aussian multiplicative chaos.
\newblock {\em Comm. Math. Phys.}, 337(2):569--632, 2015.

\bibitem[Mac15]{Macdonald}
Ian~G. Macdonald.
\newblock {\em Symmetric functions and {H}all polynomials}.
\newblock Oxford Classic Texts in the Physical Sciences. The Clarendon Press,
  Oxford University Press, New York, second edition, 2015.
\newblock With contribution by A. V. Zelevinsky and a foreword by Richard
  Stanley, Reprint of the 2008 paperback edition [ MR1354144].

\bibitem[MRV16]{MRV}
Thomas Madaule, R\'{e}mi Rhodes, and Vincent Vargas.
\newblock Glassy phase and freezing of log-correlated {G}aussian potentials.
\newblock {\em Ann. Appl. Probab.}, 26(2):643--690, 2016.

\bibitem[MRW17]{MezzadriReynolds}
Francesco Mezzadri, Alexi~K. Reynolds, and Brian Winn.
\newblock Moments of the eigenvalue densities and of the secular coefficients
  of {$\beta$}-ensembles.
\newblock {\em Nonlinearity}, 30(3):1034--1057, 2017.

\bibitem[Naj18]{NajnudelZeta}
Joseph Najnudel.
\newblock On the extreme values of the {R}iemann zeta function on random
  intervals of the critical line.
\newblock {\em Probab. Theory Related Fields}, 172(1-2):387--452, 2018.

\bibitem[NSW20]{NSW}
Miika Nikula, Eero Saksman, and Christian Webb.
\newblock Multiplicative chaos and the characteristic polynomial of the {CUE}:
  the {$L^1$}-phase.
\newblock {\em Trans. Amer. Math. Soc.}, 373(6):3905--3965, 2020.

\bibitem[PZ18]{PaquetteZeitouni}
Elliot Paquette and Ofer Zeitouni.
\newblock The maximum of the {CUE} field.
\newblock {\em Int. Math. Res. Not. IMRN}, (16):5028--5119, 2018.

\bibitem[Rem20]{Remy}
Guillaume Remy.
\newblock The {F}yodorov-{B}ouchaud formula and {L}iouville conformal field
  theory.
\newblock {\em Duke Math. J.}, 169(1):177--211, 2020.

\bibitem[RV14]{RhodesVargasSurvey}
R\'{e}mi Rhodes and Vincent Vargas.
\newblock Gaussian multiplicative chaos and applications: a review.
\newblock {\em Probab. Surv.}, 11:315--392, 2014.

\bibitem[Sta89]{Stanley}
Richard~P. Stanley.
\newblock Some combinatorial properties of {J}ack symmetric functions.
\newblock {\em Adv. Math.}, 77(1):76--115, 1989.

\bibitem[Sta99]{StanleyVol2}
Richard~P. Stanley.
\newblock {\em Enumerative combinatorics. {V}ol. 2}, volume~62 of {\em
  Cambridge Studies in Advanced Mathematics}.
\newblock Cambridge University Press, Cambridge, 1999.
\newblock With a foreword by Gian-Carlo Rota and appendix 1 by Sergey Fomin.

\bibitem[SW20]{SaksmanWebb}
Eero Saksman and Christian Webb.
\newblock The {R}iemann zeta function and {G}aussian multiplicative chaos:
  {S}tatistics on the critical line.
\newblock {\em Ann. Probab.}, 48(6):2680--2754, 2020.

\bibitem[Web15]{Webb}
Christian Webb.
\newblock The characteristic polynomial of a random unitary matrix and
  {G}aussian multiplicative chaos---the {$L^2$}-phase.
\newblock {\em Electron. J. Probab.}, 20:no. 104, 21, 2015.

\bibitem[Wu00]{Wu}
Liming Wu.
\newblock A new modified logarithmic {S}obolev inequality for {P}oisson point
  processes and several applications.
\newblock {\em Probab. Theory Related Fields}, 118(3):427--438, 2000.

\end{thebibliography}
     \end{document}